\documentclass[leqno]{article}

\usepackage[frenchb,english]{babel}
\usepackage[utf8]{inputenc}
\usepackage{amsmath}
\usepackage{amssymb}
\usepackage{amsfonts}
\usepackage{enumerate}
\usepackage{vmargin}
\usepackage[all]{xy}
\usepackage{mathrsfs}
\usepackage{mathtools}
\usepackage{lmodern}
\usepackage{slashed}
\usepackage[colorlinks=true,linkcolor=blue,pagebackref=true]{hyperref}%
\setmarginsrb{3cm}{3cm}{3.5cm}{3cm}{0cm}{0cm}{1.5cm}{3cm}
\usepackage{comment}
\usepackage{tikz}
\usetikzlibrary{patterns}

\newcommand{\Cbb}{\ensuremath{\mathbb{C}}}
\newcommand{\C}{\mathrm{C}}

\newcommand{\E}{\ensuremath{\mathbb{E}}}

\newcommand{\N}{\ensuremath{\mathbb{N}}}
\newcommand{\B}{\mathrm{B}} 
\let\H\relax 
\newcommand{\H}{\mathrm{H}}

\let\L\relax 
\newcommand{\L}{\mathrm{L}}
\newcommand{\J}{\mathcal{J}} 

\newcommand{\SO}{\mathrm{SO}}
\newcommand{\Aut}{\mathrm{Aut}}

\newcommand{\M}{\mathrm{M}}

\newcommand{\RCD}{\mathrm{RCD}}
\newcommand{\dist}{\mathrm{dist}}

\newcommand{\Spin}{\mathrm{Spin}}

\newcommand{\Rie}{\mathrm{R}} 
\newcommand{\g}{\mathfrak{g}}
\newcommand{\Scal}{\mathrm{Scal}} 

\newcommand{\X}{\mathcal{X}} 
\let\div\relax
\newcommand{\div}{\mathrm{div}}
\let\cal\relax
\newcommand{\cal}{\mathcal}
\newcommand{\Z}{\ensuremath{\mathbb{Z}}}
\newcommand{\R}{\ensuremath{\mathbb{R}}}
\newcommand{\T}{\mathrm{T}}

\newcommand{\W}{\mathrm{W}}

\newcommand{\Id}{\mathrm{Id}}
\newcommand{\VN}{\mathrm{VN}}

\newcommand{\la}{\langle}
\newcommand{\ra}{\rangle}

\renewcommand{\leq}{\ensuremath{\leqslant}}
\renewcommand{\geq}{\ensuremath{\geqslant}}
\newcommand{\qed}{\hfill \vrule height6pt  width6pt depth0pt}
\newcommand{\bnorm}[1]{ \big\| #1  \big\|}

\newcommand{\norm}[1]{\left\Vert#1\right\Vert}

\newcommand{\Iso}{\mathrm{Iso}} 
\newcommand{\Stab}{\mathrm{Stab}} %
\newcommand{\xra}{\xrightarrow}
\newcommand{\co}{\colon}

\newcommand{\ot}{\otimes}
\newcommand{\ovl}{\overline}
\newcommand{\otvn}{\ovl\ot}

\newcommand{\cb}{\mathrm{cb}}
\newcommand{\fin}{\mathrm{fin}} 

\newcommand{\HdR}{\mathrm{HdR}} 

\let\i\relax 
\newcommand{\i}{\mathrm{i}}
\newcommand{\ov}{\overset}

\newcommand{\vol}{\mathrm{vol}} 

\newcommand{\ad}{\mathrm{ad}}
\newcommand{\Ad}{\mathrm{Ad}}
\newcommand{\Ric}{\mathrm{Ric}} 

\newcommand{\Inv}{\mathrm{Inv}} 
\newcommand{\GL}{\mathrm{GL}} 

\newcommand{\epsi}{\varepsilon}
\newcommand{\Cl}{\mathrm{Cl}} 
\renewcommand{\d}{\mathop{}\mathopen{}\mathrm{d}} 
\newcommand{\e}{\mathrm{e}} 

\DeclareMathOperator{\Span}{span} 
\DeclareMathOperator{\tr}{Tr} 
\let\ker\relax 
\DeclareMathOperator{\ker}{Ker} 
\DeclareMathOperator{\Ran}{Ran} 
\DeclareMathOperator{\gap}{Gap} 
\DeclareMathOperator{\Gap}{Gap} 
\DeclareMathOperator{\dom}{dom} 
\let\Re\relax 
\DeclareMathOperator{\Re}{Re} 
\DeclareMathOperator{\End}{End}

\DeclareMathOperator{\grad}{grad}
\DeclareMathOperator{\card}{card} 
 
\DeclareMathOperator{\diam}{diam} 

\newtheorem{thm}{Theorem}[section]
\newtheorem{defi}[thm]{Definition}
\newtheorem{prop}[thm]{Proposition}
\newtheorem{conj}[thm]{Conjecture}

\newtheorem{cor}[thm]{Corollary}
\newtheorem{lemma}[thm]{Lemma}

\newtheorem{remark}[thm]{Remark}
\newtheorem{example}[thm]{Example}

\newenvironment{proof}[1][]{\noindent {\it Proof #1} : }{\hbox{~}\qed
\smallskip
}

\usepackage{tocloft}
\numberwithin{equation}{section}
\usepackage[nottoc,notlot,notlof]{tocbibind}

\let\OLDthebibliography\thebibliography
\renewcommand\thebibliography[1]{
  \OLDthebibliography{#1}
  \setlength{\parskip}{0pt}
  \setlength{\itemsep}{0pt plus 0.3ex}
}

\newcommand\reallywidehat[1]{\arraycolsep=0pt\relax%
\begin{array}{c}
\stretchto{
  \scaleto{
    \scalerel*[\widthof{\ensuremath{#1}}]{\kern-.5pt\bigwedge\kern-.5pt}
    {\rule[-\textheight/2]{1ex}{\textheight}} 
  }{\textheight} %
}{0.5ex}\\           
#1\\                 
\rule{-1ex}{0ex}
\end{array}
}

\begin{document}
\selectlanguage{english}
\title{\bfseries{A reverse Riesz estimate combined with a spectral gap implies a Poincar\'e inequality}}
\date{}
\author{\bfseries{C\'edric Arhancet}}
\maketitle


\begin{abstract}
Working at the level of a sectorial operator $A$ on a Banach space $X$ and an unbounded operator $\partial$ defined on a subspace of $X$ in another Banach space $Y$, we show that a single reverse Riesz estimate $\|A^\alpha x\|_X \lesssim \|\partial x\|_Y$ for some $0 < \alpha < 1$, combined with the condition $0 \in \rho(A_0)$, where $A_0$ is the part of $A$ on the closure of the range of $A$, implies the Poincar\'e inequality $\|x - P(x)\|_X \lesssim \|\partial x\|_Y$, where $P$ is the Abel-ergodic projection onto the kernel of $A$. The condition $0 \in \rho(A_0)$ is the natural abstract substitute for a spectral gap, and is sharp already in the Hilbertian case. We also obtain a companion divergence inequality. The arguments are remarkably short, yet the principle is genuinely unifying: it covers commutative and noncommutative situations on the same footing and can be used with arbitrary Banach spaces. As a consequence, we recover Poincar\'e inequalities associated with hypercontractive semigroups and extend their validity to semigroups satisfying only a reduced spectral gap and a reverse Riesz estimate. We then illustrate the flexibility of the method across a wide spectrum of geometries, ranging from Riemannian manifolds, Lie groups, metric measure spaces, spin manifolds to genuinely noncommutative settings such as group von Neumann algebras, semigroups of Schur multipliers, $q$-Ornstein-Uhlenbeck semigroups and quantum tori, where we sometimes establish new inequalities and otherwise recover classical ones from a single principle.
\end{abstract}


\makeatletter
 \renewcommand{\@makefntext}[1]{#1}
 \makeatother

{
  \hypersetup{linkcolor=blue}
 \tableofcontents
}


\section{Introduction}
\label{sec:Introduction}

In his famous memoir \cite[p.~253--259]{Poi90}, Poincar\'e proved the following inequality for a bounded open convex subset $\Omega$ of $\R^d$ of finite measure. There exists a constant $C_\Omega$ such that for any smooth function $f \co \Omega \to \R$ with $\int_\Omega f=0$ we have
\begin{equation}
\label{Poincare-intro}
\norm{f}_{\L^2(\Omega)} 
\leq C_\Omega \norm{\nabla f}_{\L^2(\Omega,\ell^2_d)}. 
\end{equation}
We refer to \cite{All12} and \cite{Maw12} for a historical account of the work of Poincar\'e on his inequality and to \cite[Proposition 2.58 p.~123]{FLW22} for a modern proof, for an open connected subset  $\Omega$ of finite measure with the $\W^{1,2}$-extension property, i.e., there exists a bounded linear operator $E \co \W^{1,2}(\Omega) \to \W^{1,2}(\R^d)$ such that for any $f \in \W^{1,2}(\Omega)$ we have $(Ef)(x) = f(x)$ for almost every $x \in \Omega$. Indeed, this result admits the slightly more general form  
\begin{equation}
\label{eq:intro-classical}
\norm{ f - f_\Omega }_{\L^2(\Omega)} \leq C_\Omega \norm{ \nabla f }_{\L^2(\Omega, \ell^2_d)},
\quad f_\Omega \ov{\mathrm{def}}{=} \frac{1}{|\Omega|} \int_\Omega f, \quad f \in \W^{1,2}(\Omega).
\end{equation}
So, in this classical form, the Poincar\'e inequality asserts that a function is controlled, in $\L^2$-norm, by its gradient, once its average has been subtracted. The subtraction of the mean is not a cosmetic normalization. Constant functions lie in the kernel of the gradient, so no estimate of the form \eqref{eq:intro-classical} can hold unless one first subtracts the mean. Behind \eqref{eq:intro-classical} stands a single spectral fact. As in \cite[pp.~170-171]{FLW22}, we can consider the Neumann Laplacian $-\Delta_\Omega^\mathrm{N}$ which is an unbounded positive selfadjoint operator on the complex Hilbert space $\L^2(\Omega)$ with domain
\begin{align*}
\MoveEqLeft
\dom -\Delta_\Omega^\mathrm{N}
=
\left\{
f \in \W^{1,2}(\Omega) :
 f \text{ has a weak Laplacian } \Delta f \in \L^2(\Omega) 
\right.\\
&\left.
\text{ and for all } g \in \W^{1,2}(\Omega),
\int_\Omega \nabla f \cdot \nabla g \d x 
=-\int_\Omega (\Delta f) g \d x
\right\},
\end{align*}
whose kernel consists of the constants, since $\Omega$ is connected. 
On the orthogonal complement of that kernel the spectrum of the Laplacian $-\Delta_\Omega^\mathrm{N}$ is bounded below. Indeed, according to \cite[p.~536]{Rob91}, the square of the best constant in \eqref{eq:intro-classical} is exactly the reciprocal of the bottom $\lambda_1$ of the spectrum of $-\Delta_\Omega^\mathrm{N}$ away from $0$.

In the case where the bounded open subset $\Omega$ is convex, it is worth noting that the constant in \eqref{eq:intro-classical} can be taken to be $\frac{\diam(\Omega)}{\pi}$. This bound is sharp, since it is attained for intervals, as shown by Bebendorf in \cite{Beb03}, correcting an error in \cite{PaW60}, more than a century after Poincar\'e's seminal paper. In conclusion, a Poincar\'e inequality is the analytic counterpart of a spectral gap.

The preceding discussion is not specific to the Neumann Laplacian. More generally, one may replace $-\Delta_\Omega^\mathrm{N}$ by any positive selfadjoint unbounded operator $A$ acting on a Hilbert space $\L^2(\Omega)$ for some measure space $\Omega$. Then the relevant question is whether the operator $A$ has a spectral gap, that is, there exists $c >0$ such that $\sigma(A) \subset \{0\} \cup [c,\infty)$ and to compute the constant $c$. This viewpoint is ubiquitous rather than anecdotal. Under reasonable assumptions on $A$, the same spectral gap is what forces the associated semigroup $(\e^{-tA})_{t \geq 0}$ to return to equilibrium exponentially fast. 
It therefore provides a quantitative form of ergodicity, see \cite[Section 4.2.2]{BGL14} and \cite[Section 1.2]{Wan05}. It also sits at the bottom of the hierarchy of functional inequalities. Indeed, it is implied by logarithmic Sobolev inequalities, see \cite[Proposition 5.1.3 p.~238]{BGL14} for a precise statement, and it implies an exponential measure concentration, see \cite[Section 4.4.3]{BGL14}. On a geometric side, Poincar\'e inequalities are connected through the Lichnerowicz and Buser estimates to lower bounds on the Ricci curvature and to isoperimetry. For example, if $M$ is a compact Riemannian manifold without boundary of dimension $d \geq 2$ with Ricci curvature $\Ric(M) \geq k > 0$, then the Lichnerowicz estimate \cite[theorem p.~107]{ScY94} is the inequality $\lambda_1(-\Delta) \geq \frac{d}{d-1}k$ for the Laplacian $\Delta$ on $M$. 
Thus the Poincar\'e inequality organizes analysis, probability, spectral geometry and numerical analysis around one common object: the spectral gap.

Actually, if $1 \leq p < \infty$, it is also possible, according to \cite[Theorem 13.27 p.~432]{Leo17}, to generalize this inequality to $\L^p$-spaces. For an open connected subset $\Omega$ of finite measure with the $\W^{1,p}$-extension property, we have
\begin{equation}
\label{eq:intro-classical-Lp}
\norm{ f - f_\Omega }_{\L^p(\Omega)} 
\lesssim_\Omega \norm{ \nabla f }_{\L^p(\Omega, \ell^p_d)},
\quad f \in \W^{1,p}(\Omega).
\end{equation}
Beyond domains of $\R^d$, $\L^p$-Poincar\'e inequalities have been extensively studied in a wide variety of geometric and analytic frameworks, including Markov semigroups \cite{BGL14}, Riemannian manifolds \cite[Theorem 2.10 p.~40]{Heb96} \cite{Mil09a}, Lie groups \cite{RuS11}, discrete hypercube $\{-1,1\}^n$ \cite{EfL08} \cite[Theorem 1.4 p.~297]{Tal93}\footnote{\thefootnote. This result is stated for real functions. 
} and, increasingly, on genuinely noncommutative spaces such as quantum tori \cite{XXY18}, mixed $Q$-Gaussian algebras \cite{JLZZ24}, group von Neumann algebras \cite{JLZZ24} \cite{JuZ15a} 
and crossed products \cite{Zen14}.

The framework of this paper is that of a Banach space $X$, a second Banach space $Y$ playing the role of the space in which the gradients take their values and two operators. The first is a sectorial operator $A$ on $X$. The second is a densely defined unbounded operator $\partial \co \dom \partial \subset X \to Y$, which incarnates a gradient or some kind of derivation. The spectral gap is encoded intrinsically. Writing $A_0$ for the part of the operator $A$ on the Banach space $\overline{\Ran A}$, we say that $A$ admits a spectral gap if $0 \in \rho(A_0)$. For a positive unbounded selfadjoint operator $A$ acting on a Hilbert space, this condition is equivalent to $\inf \sigma(A_0) > 0$, that is, to the strict positivity of the classical gap. It is thus the correct Banach-space formulation of ``$A$ admits a spectral gap''. We will show in Lemma \ref{lem-reduced-invertibility-implies-splitting} that this condition entails that $A$ is Abel-ergodic, meaning that $X$ splits topologically as
\begin{equation}
\label{eq:intro-splitting}
X 
= \ker A \oplus \ovl{\Ran A}.
\end{equation}
The associated ergodic projection $P \co X \to X$ onto $\ker A$ is the abstract substitute for the expectation onto fixed points. In the case of an operator $-A$ generating a bounded strongly continuous semigroup of operators on a reflexive space $X$, it is delivered automatically by mean ergodicity through the Cesàro averages of the semigroup. 

Our main result is the next theorem which couples this gap with a single reverse Riesz estimate.

\begin{thm}
\label{thm:intro-main}
Let $X$ and $Y$ be Banach spaces. Let $A$ be a sectorial operator on $X$. Suppose that $A$ admits a spectral gap, i.e., $0 \in \rho(A_0)$. Consider the ergodic projection $P \co X \to X$ onto the subspace $\ker A$. Let $\partial \co \dom \partial \subset X \to Y$ be an unbounded operator with $\dom \partial \subset \dom A^{\frac{1}{2}}$, and assume the reverse Riesz estimate
\begin{equation}
\label{eq:intro-reverse-riesz}
\bnorm{ A^{\frac{1}{2}}(x) }_X 
\leq C \norm{ \partial(x) }_Y,
\quad x \in \dom \partial.
\end{equation}
Then
\begin{equation}
\label{eq:intro-poincare}
\norm{ x - P(x) }_X 
\leq C \bnorm{ A_0^{-\frac{1}{2}} }_{\ovl{\Ran A} \to \ovl{\Ran A}} \norm{ \partial(x) }_Y,
\quad x \in \dom \partial.
\end{equation}
\end{thm}
We shall see below in Theorem \ref{thm-direct-Poincare-via-negative-powers} that the exponent $\frac12$ can be replaced by an arbitrary parameter $\alpha \in (0,1)$. This flexibility may be useful in irregular settings, such as fractals, where the classical exponent $\frac12$ associated with Riesz transforms is no longer necessarily the most natural choice. See, for instance \cite{Fen26}.

The proof is transparent. Indeed, we have $A_0^\frac{1}{2}(x - P x) = A^\frac{1}{2} x$ for any $x \in \dom A^\frac{1}{2}$, and the boundedness of the fractional power $A_0^{-\frac{1}{2}}$ turns the reverse Riesz estimate into \eqref{eq:intro-poincare} in a couple of lines. We regard this brevity as a feature rather than a shortcut: the mathematical content of the paper lies in identifying \emph{which} two assumptions are the right ones, after which the geometry evaporates. In the Hilbertian setting, when $A$ is connected to a closed densely defined operator $\partial \co \dom \partial \subset H \to K$, where $H$ and $K$ are Hilbert spaces, through a formula of the form $A=\partial^* \partial$, the gap assumption is not only sufficient but necessary\footnote{\thefootnote. Indeed, consider the densely defined closed nonnegative symmetric form
\[
\mathfrak a(x,y)
\ov{\mathrm{def}}{=}
\la \partial x,\partial y\ra_K,
\qquad
\dom\mathfrak a=\dom\partial.
\]
Its associated positive selfadjoint operator is $A=\partial^*\partial$. Hence, by the second representation theorem
\cite[Theorem VI.2.23 p.~331]{Kat76}, we have $\dom A^{\frac12}=\dom \partial$ and $\la A^{\frac12}x,A^{\frac12}y \ra_H
=
\langle\partial x,\partial y\rangle_K$ for any $x,y \in \dom \partial$. In particular, we have
\[
\bnorm{A^{\frac12}x}_H=\norm{\partial x}_K,
\quad
x \in \dom\partial.
\]
Moreover, we have $\ker A=\ker\partial$ and $\ovl{\Ran A}=(\ker A)^\perp$. Therefore, if $P$ denotes the
orthogonal projection onto $\ker\partial$ and $A_0$ is the part of $A$ in $\ovl{\Ran A}$, the estimate $\norm{x-Px}_H\leq C\norm{\partial x}_K$, where $x \in \dom \partial$, is equivalent to
\[
\norm{y}_H 
\leq C\bnorm{A_0^{\frac12}y}_H,
\quad
y \in \dom A_0^{\frac12}.
\]
By \cite[Exercise 2 p.~279]{Sch12} and the spectral theorem, the latter is equivalent to $\inf\sigma(A_0) > 0$, hence to $0 \in \rho(A_0)$.} and the reverse Riesz estimate is in practice always satisfied. So the result is sharp already at its most classical point.

If $A
=\partial^\dagger \partial$, the same circle of ideas yields, under some assumptions, a companion divergence inequality $\norm{y}_{Y}
\lesssim \norm{\partial^\dagger(y)}_{X}$ for any $y \in \partial(\dom A)$, in which the roles of $\partial$ and of a divergence operator $\partial^\dagger$ are exchanged. See Theorem~\ref{thm-Poincare-divergence-corrected}. It is the dual counterpart of \eqref{eq:intro-poincare} and is proved by the same principle applied to the negative power $A_0^{-\frac{1}{2}}$.

Finally, it is well known \cite{AdW15} \cite[Corollary 4.2 p.~583]{EfL08} 
\cite{HuT21} \cite{JuW26} that Poincar\'e inequalities imply concentration estimates in some sense. In Section \ref{sec-concentration}, we refine this standard principle by tracking the order of magnitude of the constants with respect to $p$, in a setting sufficiently general for our purposes.

\paragraph{Scope of the result}
The strength of Theorem~\ref{thm:intro-main} is its indifference to the underlying geometry. It imposes no restriction on the exponent $p$ in the case of an $\L^p$-Poincar\'e inequality (or an $\L^p$-$\L^q$-Poincar\'e), it accommodates arbitrary Banach spaces, and it treats commutative and noncommutative situations on the same footing, allowing even the use of non-semifinite von Neumann algebras or weights on classical measure spaces. 

Our result recovers the Poincar\'e-type \textit{consequence} of \cite[Corollary B p.~3]{JLZZ24} on noncommutative $\L^p$-spaces associated with noncommutative \textit{probability spaces}, where \textit{hypercontractivity} and \textit{Markovianity} of operators of some semigroup were an important assumption, which are not satisfied in several important cases. We also refer to \cite{JuZ15b} for a result under a curvature-dimension assumption, and to \cite{JuW26} for an approach based on Markov dilations and amalgamated free products. The arguments in all these works rely on substantially more technical machinery, whereas the proof given here is comparatively direct. Their results and proofs are complementary to ours.

One should not expect the present abstract approach, whose main strength lies in its generality, to recover the optimal constants. 
Its advantage is instead to identify a common mechanism behind $\L^p$-Poincar\'e inequalities across a broad class of examples. Section~\ref{sec-examples} illustrates the method across a deliberately wide spectrum of settings: Riemannian manifolds (compact, locally symmetric of non-compact type, and Cartan--Hadamard), Riemannian spin manifolds through the Dirac operator, metric measure spaces satisfying the Riemannian curvature-dimension condition $\RCD(K,N)$, compact Lie groups equipped either with subelliptic Hörmander gradients or with bi-invariant metrics, quantum tori, $q$-Ornstein--Uhlenbeck semigroups, group von Neumann algebras, and semigroups of Schur multipliers. In some of these we obtain inequalities that appear to be new even in the commutative range, while in others we recover classical statements. In all of them the inequality follows from the single previous principle. We have made no attempt to be exhaustive.

\paragraph{Structure of the paper}
The paper is organized as follows. In Section~\ref{sec-preliminaries} we collect the operator-theoretic
preliminaries needed throughout the article, with particular emphasis on sectorial operators and Abel-ergodic sectorial operators. Section~\ref{sec-gaps} is devoted to the relation between reduced spectral gaps and uniform exponential stability of semigroups on the reduced space. This provides the spectral input which will later replace the usual Hilbertian gap condition. In Section~\ref{Sec-Poincare}, we prove the main abstract Banach Poincar\'e principle: a reverse Riesz estimate, combined with the invertibility of the reduced part of the generator, yields a Poincar\'e inequality with respect to the Abel-ergodic projection. Section~\ref{sec-divergence} develops the companion divergence inequality, which may be viewed as the dual first-order estimate associated with the same spectral mechanism. In Section~\ref{sec-Riesz-equivalence}, we explain how suitable Riesz estimates imply reverse Riesz estimates. Section \ref{sec-examples} illustrates the scope of the method in a variety of commutative and noncommutative settings. We discuss Riemannian manifolds, spin manifolds, metric measure spaces satisfying an $\RCD(K,N)$ condition, 
Markov semigroups on von Neumann algebras, compact Lie groups, quantum tori, $q$-Ornstein--Uhlenbeck semigroups, group von Neumann algebras and semigroups of Schur multipliers. These examples are meant to show that the argument is not tied to a particular geometry: the same reduced spectral gap and Riesz estimate mechanism recovers classical Poincar\'e inequalities and yields new ones in various situations. Finally, in Section \ref{sec-concentration}, we prove a concentration inequality.

\section{Preliminaries}
\label{sec-preliminaries}
In this paper, the symbols $\approx$ and $\lesssim$ denote an equality or an inequality up to multiplicative constants.

\subsection{Operator theory and sectorial operators}
\paragraph{Unbounded operators} Let $X$ and $Y$ be two Banach spaces. 
Consider an unbounded operator $T \co \dom T \subset X \to Y$ with dense domain. According to \cite[Problem 5.27 p.~168]{Kat76}, we have
\begin{equation}
\label{lien-ker-image}
\ker T^*
=(\Ran T)^\perp.
\end{equation}
For any $x \in \dom T$ and any $y \in \dom T^*$, we have the equality
\begin{equation}
\label{crochet-duality}
\langle T(x),y \rangle_{Y,Y^*}
=\langle x,T^*(y) \rangle_{X,X^*}.
\end{equation}

We will use the following result in some Examples.

\begin{prop}
\label{prop-generator-gradient-factorization}
Let $X$ and $Y$ be reflexive Banach spaces. Let $A \co \dom A \subset X \to X$ be a densely defined closed operator. Let $\partial \co \dom \partial \subset X \to Y$ and $\tilde{\partial} \co \dom \tilde{\partial} \subset X^* \to Y^*$ be densely defined closed operators. Assume that $\dom A \subset \dom \partial$ and $ 
\dom A^* \subset \dom \tilde{\partial}$. Suppose that there exist subspaces $
C \subset \dom A \cap \dom \partial$ and $\tilde{C} \subset \dom A^* \cap \dom \tilde{\partial}$ such that $C$ is a core for both $A$ and $\partial$, $\tilde{C}$ is a core for both $A^*$ and $\tilde{\partial}$, and
\begin{equation}
\label{eq-abstract-gradient-factorization-core}
\la A x,y\ra_{X,X^*}
=
\la \partial x,\tilde{\partial}y\ra_{Y,Y^*},
\quad
x\in C, y \in \tilde{C}.
\end{equation}
Then $
A
=
(\tilde{\partial})^*\partial$ and $A^*
=
\partial^*\tilde{\partial}$.
\end{prop}

\begin{proof}
Since $A$ and $\partial$ are closed and $\dom A \subset \dom \partial$, the inclusion $j \co \dom A \hookrightarrow \dom \partial$ is continuous if the spaces are equipped with the graph norms, by the closed graph theorem. Indeed, suppose that $(x_n)$ is a sequence in $\dom A$ such that $x_n \to x$ in the graph norm of $A$ and $j(x_n) \to y$ in the graph norm of $\partial$. The first convergence implies in particular that $
x_n \to x$ in $X$. The second convergence also implies that $
x_n
=
j(x_n)
\to y$ in $X$.
By uniqueness of the limit in $X$, we obtain $x=y$. Since $x \in \dom A$, we have $j(x)=x=y$. Consequently, the graph of $j$ is closed. Similarly, since $\dom A^* \subset \dom \widetilde{\partial}$, the inclusion $\dom A^* \hookrightarrow \dom \tilde{\partial}$ is continuous. We first prove that $
A
\subset
(\tilde{\partial})^*\partial$. Let $x\in\dom A$. Since $C$ is a core for $A$, there exists a sequence
$(x_n)$ in $C$ such that $x_n\to x$ and $ Ax_n\to Ax$ in $X$. By the continuity of the preceding inclusion, $\partial x_n\to\partial x$ in $Y$. Let $y \in \dom \tilde{\partial}$. Since $\tilde{C}$ is a core for $\tilde{\partial}$, there exists a sequence $(y_m)$ in $\tilde{C}$ such that $y_m\to y$ and $\tilde{\partial}y_m \to \tilde{\partial}y$. Passing successively to the limit in $\la A x_n,y_m\ra_{X,X^*} \ov{\eqref{eq-abstract-gradient-factorization-core}}{=} \la \partial x_n,\tilde{\partial}y_m\ra_{Y,Y^*}$, we obtain
\[
\la Ax,y\ra_{X,X^*}
=
\la\partial x,\tilde{\partial}y\ra_{Y,Y^*},
\quad
y \in \dom \tilde{\partial}.
\]
By the definition of the Banach adjoint, this means that $\partial x \in \dom(\tilde{\partial})^*
$ and $
(\tilde{\partial})^*\partial x=Ax$. This proves $A
\subset
(\tilde{\partial})^*\partial$. Now, we prove that
\begin{equation}
\label{eq-dual-first-inclusion-factorization}
A^*
\subset
\partial^*\tilde{\partial}.
\end{equation}
Let $y \in \dom A^*$. Since $\tilde{C}$ is a core for $A^*$, there exists a sequence $(y_m)$ in $\tilde{C}$ such that $y_m \to y$ and $A^*y_m\to A^*y$. By the continuity of the inclusion $\dom A^*\hookrightarrow \dom \tilde{\partial}$, we also have $\tilde{\partial}y_m \to \tilde{\partial}y$. Let $x \in \dom\partial$. Since $C$ is a core for $\partial$, there exists a sequence $(x_n)$ in $C$ such that $x_n \to x$ and $\partial x_n \to \partial x$. Using
\[
\la x_n,A^*y_m \ra_{X,X^*}
\ov{\eqref{crochet-duality}}{=}
\la Ax_n,y_m\ra_{X,X^*}
\ov{\eqref{eq-abstract-gradient-factorization-core}}{=}
\la\partial x_n,\tilde{\partial}y_m\ra_{Y,Y^*}
\]
and passing successively to the limit, we obtain
\[
\la x,A^*y\ra_{X,X^*}
=
\la\partial x,\tilde{\partial}y\ra_{Y,Y^*},
\quad
x \in \dom \partial.
\]
Hence $\tilde{\partial}y\in\dom\partial^*$ and $
\partial^*\tilde{\partial}y=A^*y$, which proves \eqref{eq-dual-first-inclusion-factorization}. Set $
B
\ov{\mathrm{def}}{=}
(\tilde{\partial})^*\partial$. Since $A
\subset
(\tilde{\partial})^*\partial$, we have $A \subset B$. Since $A$ is densely defined, $B$ is densely defined as well, and therefore $B^*$ is defined. Taking adjoints gives $
B^*
\subset
A^*$. On the other hand, since $\tilde{\partial}$ is closed and $X$ and $Y$ are reflexive, we have $
\tilde{\partial}^{**} = \tilde{\partial}$ by \cite[Theorem 5.29 p.~168]{Kat76}. By the usual inclusion \cite[Problem 5.26 p.~168]{Kat76} for adjoints of products of unbounded operators, we have
\[
\partial^*\tilde{\partial}
=
\partial^*\tilde{\partial}^{**}
\subset
\big((\tilde{\partial})^*\partial\big)^*
=
B^*.
\]
Together with \eqref{eq-dual-first-inclusion-factorization}, this gives $
A^*
\subset
B^*$. Combining this inclusion with $B^*
\subset
A^*$, we conclude that $
B^*=A^*$. Taking adjoints once more and using the reflexivity of $X$, we obtain $
B^{**}
=
A^{**}
=
A$. Since $B\subset B^{**}$, we have $B\subset A$. Together with $A \subset B$, this proves $A=(\tilde{\partial})^*\partial$. Finally, taking adjoints in this identity and using \eqref{eq-dual-first-inclusion-factorization} yields $A^*=\partial^*\tilde{\partial}$. 
\end{proof}

For any angle $\theta \in (0,\pi)$, we introduce the open sector symmetric around the positive real half-axis with opening angle $2\theta$
\begin{equation}
\label{def-sigma-omega}
\Sigma_{\theta} 
\ov{\mathrm{def}}{=} \big\{ z \in \mathbb{C} \backslash \{ 0 \} : \: | \arg z | < \theta \big\}.
\end{equation}
It is useful to put $\Sigma_0 \ov{\mathrm{def}}{=} (0,\infty)$.

Background material on sectorial operators and spectral theory can be found in the books \cite{Haa06} and \cite{HvNVW18}.
Let $A \co \dom A \subset X \to X$ be a closed densely defined linear operator acting on a Banach space $X$. We say that $A$ is a $\theta$-sectorial operator for some angle $\theta \in (0,\pi)$ if its spectrum $\sigma(A)$ is a subset of the closed sector $\ovl{\Sigma_\theta}$ and if the set $\big\{zR(z,A) : z \in \mathbb{C} \backslash \ovl{\Sigma_\theta}\big\}$ is bounded in the algebra $\B(X)$ of bounded operators acting on $X$, where $R(z,A) \ov{\mathrm{def}}{=} (z\,\Id-A)^{-1}$ is the resolvent operator. We caution the reader that this definition may differ across the literature. 
The operator $A$ is said to be sectorial if it is a $\theta$-sectorial operator for some $\theta \in (0,\pi)$. In this situation, we can introduce the angle of sectoriality
$\omega_{\sec}(A) 
\ov{\mathrm{def}}{=} \inf\{ \theta \in (0,\pi) : A \textrm{ is $\theta$-sectorial} \}$. 
According to \cite[Example 10.1.2 p.~362]{HvNVW18}, if $-A$ is the generator of a bounded strongly continuous semigroup $(T_t)_{t \geq 0}$ of bounded operators then the operator $A$ is sectorial with $\omega_{\sec}(A) \leq \frac{\pi}{2}$. Furthermore, by \cite[Example 10.1.3 p.~362]{HvNVW18} and \cite[Proposition 3.4.4 p.~79]{Haa06}, the operator $A$ is sectorial with $\omega_{\sec}(A) < \frac{\pi}{2}$ if and only if $-A$ generates a bounded holomorphic (equivalently, bounded analytic) strongly continuous semigroup $(T_t)_{t \geq 0}$, that is, there exist an angle $\theta \in (0,\frac{\pi}{2})$ and a bounded holomorphic extension $\Sigma_\theta \to \B(X)$, $z \mapsto T_z$.

\begin{example}\normalfont
\label{Example-analytic-positive-selfadjoint}
If $A$ is a positive selfadjoint operator with dense domain on a Hilbert space, then \cite[Example 3.7.5 p.~150]{ABHN11} shows that $(\e^{-tA})_{t \geq 0}$ is a bounded holomorphic strongly continuous semigroup.
\end{example}

If $A$ is a sectorial operator on a \textit{reflexive} Banach space $X$, we have by \cite[Proposition 2.1.1 (h) p.~21]{Haa06} or \cite[Proposition 10.1.9 p.~367]{HvNVW18} a topological decomposition
\begin{equation}
\label{decompo-reflexive}
X
=\ker A \oplus \ovl{\Ran A}.
\end{equation}

\paragraph{Fractional powers} 
References on fractional powers include \cite{ABHN11}, \cite{Haa06}, \cite{Haa18}, \cite{MCSA01} and \cite{HvNVW23}. If $A$ is a sectorial operator on a Banach space $X$ and if $\alpha>0$ satisfies $\alpha\omega_{\sec}(A) < \pi$, then $A^\alpha$ is sectorial and $\omega_{\sec}(A^\alpha) = \alpha\omega_{\sec}(A)$ by \cite[Proposition 3.1.2]{Haa06} and \cite[Theorem 15.2.7 p.~440]{HvNVW23}. 
By \cite[Proposition 3.1.1 (d) p.~61]{Haa06} combined with a duality argument relying on \cite[Problem 5.27 p.~168]{Kat76} and \cite[Proposition 1.10.15 (c) p.~93]{Meg98}, we have, for any complex number $\alpha \in \mathbb{C}$ with $\Re \alpha > 0$, the equalities
\begin{equation}
\label{inclusion-range}
\ker A^\alpha=\ker A
\quad \text{and} \quad
\ovl{\Ran A^\alpha} 
=\ovl{\Ran A}.
\end{equation}
Finally, whenever $A$ is densely defined and $0 < \Re\alpha < 1$, the subspace $\dom A$ is a core for the unbounded operator $A^\alpha$ by \cite[Proposition 3.1.1 (h) p.~61]{Haa06}.  In Section \ref{sec-Markov}, we will use the following observation.

\begin{prop}
\label{core-A-alpha}
Let $A$ be a sectorial operator on a Banach space $X$. Let $\alpha \in \mathbb{C}$ with $0 < \Re \alpha < 1$. Every core of the unbounded operator $A$ is also a core of $A^\alpha$.
\end{prop}

\begin{proof}
If $0 < \Re\alpha < 1$, then
\[
A^\alpha x
=
A^\alpha(\Id+A)^{-1}(\Id+A)x,
\quad x\in\dom A.
\]
Fix $\theta \in (\omega_{\sec}(A),\pi)$. Note that the function $z \mapsto \frac{z^\alpha}{1+z}$ belongs to the algebra $\H^\infty_0(\Sigma_\theta)$ by \cite[Example 2.2.5 p.~29]{Haa18}. So the operator $A^\alpha(\Id+A)^{-1}$ is bounded by \cite[p.~30]{Haa18}. Hence there exists a constant
$K_\alpha \geq 0$ such that
\begin{equation}
\label{majo-graph}
\bnorm{A^\alpha x}_X
\leq
K_\alpha\big(\norm{x}_X+\norm{Ax}_X\big),
\quad x\in\dom A.
\end{equation}
Let $C$ be a core of the operator $A$. If $x \in \dom A$, there exists a sequence $(x_n)$ in $C$ such that $x_n \to x$ and $Ax_n\to Ax$ in $X$. By \eqref{majo-graph}, applied to $x_n-x$, we obtain
\[
\bnorm{A^\alpha(x_n-x)}_X
\leq
K_\alpha
\big(
\norm{x_n-x}_X+\norm{A(x_n-x)}_X
\big)
\longrightarrow 0.
\]
Hence $C$ is dense in $\dom A$ for the graph norm of $A^\alpha$. Since $\dom A$ is itself a core of $A^\alpha$, we conclude that $C$ is a core of $A^\alpha$.
\end{proof}



%
%

\subsection{Abel-ergodic sectorial operators}

We introduce the following definition.

\begin{defi}
\label{def-Abel-ergodic}
Let $A$ be a sectorial operator acting on a Banach space $X$. We say that $A$ is Abel-ergodic if we have a topological decomposition $X=\ker A \oplus \ovl{\Ran A}$.
\end{defi}

By \cite[Corollary 4.3.2 p.~262]{ABHN11} or \cite[Proposition 10.1.7 (3) p.~364]{HvNVW18}, this is equivalent to $P(x) \ov{\mathrm{def}}{=} \lim_{\lambda \to 0, \lambda >0} \lambda R(\lambda,-A)x$ exists for all $x \in X$. In this case, the map $P \co X \to X$ is a bounded projection on $\ker A$ along the subspace $\ovl{\Ran A}$, called ergodic projection of $A$. This means that
\begin{equation}
\label{Ran-ker-P}
\Ran P = \ker A
\quad \text{and} \quad
\ker P
=\ovl{\Ran A}.
\end{equation}

\begin{example}
\label{sectorial-Abel-ergodic}
\normalfont
By \cite[Proposition 10.1.9 p.~367]{HvNVW18}, every sectorial operator on a \textit{reflexive} Banach space $X$ is Abel-ergodic.
\end{example}

If $-A$ generates a bounded strongly continuous semigroup $(T_t)_{t \geq 0}$ of bounded operators on $X$, following \cite[Definition 4.3.3 p.~262]{ABHN11}, we say that $(T_t)_{t \geq 0}$ is Abel-ergodic if the sectorial operator $A$ is Abel-ergodic. The semigroup $(T_t)_{t \geq 0}$ is said to be mean-ergodic (or Ces\`aro ergodic) \cite[Definition 4.3.3 p.~262]{ABHN11} \cite[Definition 1.1.8 p.~9]{Eme07} if the Ces\`aro means admit limits in norm, that is, for any $x \in X$ the limit 
\begin{equation}
\label{def-de-P-lim}
P(x) 
\ov{\mathrm{def}}{=} \lim_{t \to \infty} \frac{1}{t} \int_{0}^{t} T_s x \d s
\end{equation}
exists. By \cite[Proposition 4.3.4 a) p.~263]{ABHN11}, this implies that $(T_t)_{t \geq 0}$ is Abel-ergodic and that $P \co X \to X$ coincides with the ergodic projection of $A$. In this case, we say that $P$ is the mean ergodic projection of $(T_t)_{t \geq 0}$. It is obvious that
\begin{equation}
\label{mean-et-Tt}
T_tP
=PT_t
=P, \quad t \geq 0.
\end{equation}
Conversely, according to \cite[Proposition 4.3.4 b) p.~263]{ABHN11}, if the strongly continuous semigroup $(T_t)_{t \geq 0}$ is \textit{bounded} and Abel-ergodic then $(T_t)_{t \geq 0}$ is mean-ergodic.


\begin{example}
\label{bounded-ergodic}
\normalfont
By \cite[Corollary 4.3.5 p.~263]{ABHN11}, every \textit{bounded} strongly continuous semigroup $(T_t)_{t \geq 0}$ on a \textit{reflexive} Banach space $X$ is mean-ergodic.
\end{example}

\begin{example}
\normalfont
According to \cite[Proposition 3.1.4 p.~120]{Eme07}, any one-parameter Markov semigroup $(T_t)_{t \geq 0}$ of operators acting on an $\L^1$-space possessing a strictly positive invariant density is mean ergodic.
\end{example}

Suppose that $X$ is reflexive. If $(T_t)_{t \geq 0}$ is a strongly continuous semigroup of bounded operators on $X$, the adjoint semigroup $(T_t^*)_{t \geq 0}$ is strongly continuous on the Banach space $X^*$ by \cite[Proposition p.~44]{EnN00}. According to \cite[Theorem 4.9 p.~33]{Gol85}, its infinitesimal generator is $-A^*$, where $-A$ is the infinitesimal generator of $(T_t)_{t \geq 0}$. Moreover, if $(T_t)_{t \geq 0}$ is mean-ergodic, it is easy to check with \eqref{def-de-P-lim} that the adjoint $P^* \co X^* \to X^*$ is the mean ergodic projection of this semigroup and that we have
\begin{equation}
\label{Ran-ker-P*}
\Ran P^*= \ker A^*
\quad \text{and} \quad 
\ker P^*
=\ovl{\Ran A^*} .
\end{equation}


\section{Spectral gaps and uniformly exponentially stable semigroups}
\label{sec-gaps}

We introduce the following definition.

\begin{defi}
\label{defi-spectral-gap}
Let $X$ be a Banach space. Consider a sectorial operator $A$ acting on $X$. We say that $A$ admits a spectral gap if $0 \in \rho(A_0)$, where $A_0$ is the part of $A$ on $\ovl{\Ran A}$. 
\end{defi}

We will describe equivalent properties in Proposition \ref{prop-reduced-exp-stability-closed-range} under reasonable assumptions. We start by recalling some background on uniformly exponentially stable semigroups and spectral bounds. Recall that a strongly continuous semigroup $\cal{T}=(T_t)_{t \geq 0}$ of operators acting on a Banach space $X$ with infinitesimal generator $-A$ is uniformly exponentially stable \cite[p.~298]{EnN00} if the exponential growth bound 
\begin{equation}
\label{exponential-growth-bound}
\omega(\cal{T})
\ov{\mathrm{def}}{=} \inf \big\{ \omega \in \R : \text{there exists } M\geq 0 \text{ such that } \norm{T_t}_{X \to X}
\leq M \e^{\omega t} \text{ if } t \geq 0 \big\}
\end{equation}
is strictly negative, i.e., there exists $M \geq 0$ and $\alpha > 0$ such that
\begin{equation}
\label{uniformly-exponentially-stable}
\norm{T_t}_{X \to X}
\leq M \e^{-\alpha t}, \quad t \geq 0.
\end{equation}
According to \cite[Proposition 1.7 p.~299]{EnN00}, this property is equivalent to the condition $\lim_{t \to \infty} \norm{T_t}_{X \to X}=0$ or the existence of some $t_0 > 0$ such that $\norm{T_{t_0}}_{X \to X} < 1$. It is worth noting that we have the following inequality \cite[p.~346]{ABHN11} between the spectral bound 
\begin{equation}
\label{spectral-bound}
s(-A) 
\ov{\mathrm{def}}{=} \sup \{\Re \lambda : \lambda \in \sigma(-A)\}
= -\inf \{\Re \lambda : \lambda \in \sigma(A)\}
\end{equation}
of the generator $-A$ and the exponential growth bound:
\begin{equation*}
s(-A) 
\leq \omega(\cal{T}).
\end{equation*}
We warn the reader that this inequality may be strict for a strongly continuous semigroup acting on a Hilbert space, see \cite[Example 5.1.10 p.~347]{ABHN11}, which describes a semigroup with $\omega(\cal{T})=1$ and $s(-A)=0$. However, if the semigroup is in addition holomorphic then $
\omega(\cal{T}) =s(-A)$ by \cite[Theorem 5.1.12 p.~350]{ABHN11}.


If $X$ is reflexive and if $(T_t)_{t \geq 0}$ is a strongly continuous semigroup of operators on $X$ with infinitesimal generator $-A$, recall that by \cite[Lemma 4.4 p.~338]{EnN00} the subspace $\ovl{\Ran A}$ is invariant under $(T_t)_{t \geq 0}$ and that the semigroup $\cal{T}_0 \ov{\mathrm{def}}{=} (T_t|_{\ovl{\Ran A}})_{t \geq 0}$ is a strongly continuous semigroup on the Banach space $\ovl{\Ran A}$. According to \cite[Corollary p.~61]{EnN00}, its infinitesimal generator is the part of the operator $-A$ in $\ovl{\Ran A}$.

\begin{prop}
\label{prop-reduced-exp-stability-closed-range}
Let $X$ be a reflexive Banach space and let $\cal{T}=(T_t)_{t \geq 0}$ be a bounded holomorphic semigroup on $X$ with infinitesimal generator $-A$. Consider the restricted semigroup $\cal{T}_0=(T_t|_{\ovl{\Ran A}})_{t \geq 0}$ with infinitesimal generator $-A_0$. The following assertions are equivalent.
\begin{enumerate}
\item $A$ admits a spectral gap, i.e., $0 \in \rho(A_0)$.
\item The semigroup $\cal{T}_0$ is uniformly exponentially stable.
\item The spectral bound of $-A_0$ is strictly negative: $s(-A_0) < 0$.
\item The subspace $\Ran A$ is closed in $X$.
\end{enumerate}
The equivalence between assertions~1 and~4  holds for every Abel-ergodic sectorial operator, without any semigroup assumption.
\end{prop}

\begin{proof}
2 $\iff$ 3: Since $\cal{T}$ is analytic, the restricted semigroup $\cal{T}_0$ is also analytic. Hence, by \cite[Theorem 5.1.12 p.~350]{ABHN11}, the spectral bound of its generator coincides with its exponential growth bound: $\omega(\cal{T}_0)=s(-A_0)$. This proves the equivalence between (2) and (3).

3 $\Rightarrow$ 1: Assume $s(-A_0) < 0$. By \eqref{spectral-bound} applied to $A_0$, we infer that $\inf \{\Re \lambda : \lambda \in \sigma(A_0)\} > 0$, and therefore $0 \in \rho(A_0)$. 

1 $\Rightarrow$ 3: Suppose that $0 \in \rho(A_0)$. Since $\rho(A_0)$ is open, there exists $\epsi > 0$ such that
\[
\sigma(A_0) \cap \{z \in \mathbb{C} : |z| < \epsi\}
=\emptyset.
\]
Since $\cal{T}_0$ is bounded analytic, $A_0$ is sectorial of angle strictly smaller than $\frac{\pi}{2}$. In particular, there exists an angle $\theta \in (0,\frac{\pi}{2})$ such that $\sigma(A_0) \subset \ovl{\Sigma_\theta}$. We get
\[
\Re \lambda \geq \epsi \cos \theta,
\quad \lambda \in \sigma(A_0).
\]
By \eqref{spectral-bound} applied to $A_0$, we deduce that $
s(-A_0)
\leq
-\epsi \cos \theta
<0$. 

1 $\Rightarrow$ 4: Assume first that $0 \in \rho(A_0)$. Then $A_0 \co \dom A_0 \to \ovl{\Ran A}$ is surjective. Consequently, we have $\ovl{\Ran A}=\Ran A_0 \subset \Ran A \subset \ovl{\Ran A}$. Hence $\Ran A=\ovl{\Ran A}$. So $\Ran A$ is closed.

4 $\Rightarrow$  1: Conversely, assume that $\Ran A$ is closed. Then $X_0=\ovl{\Ran A}=\Ran A$. Note that $A_0$ is injective. Indeed, by construction $
\ker A_0
=
\ker A \cap \ovl{\Ran A}
=
\{0\}$ by the direct sum decomposition $
X=\ker A \oplus \ovl{\Ran A}$. We next show that $A_0 \co \dom A_0 \to \ovl{\Ran A}$ is surjective. Let $y \in \ovl{\Ran A}$. Since $\ovl{\Ran A}=\Ran A$, there exists $x \in \dom A$ such that $
y=Ax$. Write $
x=x_1+x_0$ with $x_1 \in \ker A$ and $x_0 \in \ovl{\Ran A}$. Since $x_1 \in \ker A \subset \dom A$ and $x \in \dom A$, the element $x_0=x-x_1$ belongs to $\dom A \cap X_0=\dom A_0$. Moreover, we have
\[
A_0x_0
=Ax_0
=A(x-x_1)
=Ax=y.
\]
Thus $\Ran A_0=\ovl{\Ran A}$. We have proved that $A_0 \co \dom A_0 \subset \ovl{\Ran A} \to \ovl{\Ran A}$ is bijective. Since $A_0$ is closed, its inverse is bounded by the closed graph theorem. Therefore $0 \in \rho(A_0)$.
\end{proof}

The following result generalizes \cite[Exercise 4.5.4 p.~57]{Are05} and allows us to obtain a spectral gap by interpolation.

\begin{prop}
\label{prop-stability-interpolation}
Let $(X_0,X_1)$ be a compatible couple of Banach spaces. For $0 < \theta < 1$, set $X_\theta \ov{\mathrm{def}}{=} (X_0,X_1)_\theta$ for the complex interpolation space. Let $(T_0(t) \co X_0 \to X_0)_{t \geq 0}$ and $(T_1(t) \co X_1 \to X_1)_{t \geq 0}$ be two consistent strongly
continuous semigroups on $X_0$ and $X_1$ of bounded linear operators, in the sense that $T_0(t)x=T_1(t)x$ for any $x \in X_0 \cap X_1$ and any $t \geq 0$. Assume that there exist constants $M_0,M_1 \geq 0$ and $\alpha > 0$ such that
\[
\norm{T_0(t)}_{X_0 \to X_0} \leq M_0,
\quad
\text{and}
\quad
\norm{T_1(t)}_{X_1 \to X_1} \leq M_1 \e^{-\alpha t},
\quad t \geq 0.
\]
Then, for every $0\leq \theta \leq 1$, the operators $T_0(t)$ and $T_1(t)$ induce a bounded operator $T_\theta(t) \co X_\theta \to X_\theta$ and
\[
\norm{T_\theta(t)}_{X_\theta \to X_\theta}
\leq
M_0^{1-\theta} M_1^\theta \e^{-\theta \alpha t},
\quad t \geq 0.
\]
In particular, if $0<\theta\leq 1$, then the semigroup $(T_\theta(t))_{t \geq 0}$ is uniformly exponentially stable on $X_\theta$.
\end{prop}

\begin{proof}
The cases $\theta=0$ and $\theta=1$ are precisely the two assumptions. Let $0 < \theta < 1$.

Fix $t\geq 0$. By the consistency assumption, $T_0(t)$ and $T_1(t)$ define the same operator on $X_0 \cap X_1$. Hence, by the complex interpolation theorem for linear operators, they induce a bounded operator $T_\theta(t) \co X_\theta \to X_\theta$ satisfying
\[
\norm{T_\theta(t)}_{X_\theta \to X_\theta}
\leq
\norm{T_0(t)}_{X_0 \to X_0}^{1-\theta}
\norm{T_1(t)}_{X_1 \to X_1}^\theta .
\]
Using the assumptions, we obtain
\[
\norm{T_\theta(t)}_{X_\theta \to X_\theta}
\leq
M_0^{1-\theta}
\bigl(M_1 \e^{-\alpha t}\bigr)^\theta
=
M_0^{1-\theta} M_1^\theta \e^{-\theta \alpha t}.
\]
This proves the desired estimate. Since $\theta \alpha >0$ when $0 < \theta \leq 1$, the last estimate gives uniform exponential stability on $X_\theta$. Note that the semigroup $(T_\theta(t))_{t \geq 0}$ is strongly continuous on $X_\theta$ by a standard argument.
\end{proof}

Let $(T_t)_{t \geq 0}$ be a bounded strongly continuous semigroup with generator $-A$ on a Banach space $X$. By \cite[Theorem 4.10 p.~342]{EnN00} $(T_t)_{t \geq 0}$ is uniformly mean ergodic, i.e., the Ces\`aro means of \eqref{def-de-P-lim} converge in norm to $P$, if and only if the subspace $\Ran A$ is closed in $X$. So the conditions of Proposition \ref{prop-reduced-exp-stability-closed-range} are equivalent to the uniformly mean ergodicity of the semigroup.

In the reflexive case, more than \eqref{def-de-P-lim} can be said about the ergodic projection $P$.

\begin{prop}
\label{prop-strong-limit-projection}
Consider a strongly continuous bounded semigroup $(T_t)_{t \geq 0}$ on a reflexive Banach space $X$ generated by $-A$ with mean ergodic projection $P$. Assume that the semigroup $\cal{T}_0 \ov{\mathrm{def}}{=} \big(T_t|_{\ovl{\Ran A}}\big)_{t \geq 0}$ is uniformly exponentially stable. Then for all $f \in X$ we have 
$\lim_{t \to \infty} T_t(f)
=P(f)$. 
\end{prop}

\begin{proof}
By definition, there exist constants $M \geq 0$ and $\omega > 0$ such that 
\begin{equation}
\label{uniformly-exponentially-stable-bis}
\norm{T_t|_{\ovl{\Ran A}}}_{\ovl{\Ran A} \to \ovl{\Ran A}}
\ov{\eqref{uniformly-exponentially-stable}}{\leq} M \e^{-\omega t}, \quad t \geq 0
\end{equation}
Note that $\Id-P$ is the projection onto $\ovl{\Ran A}$ along $\ker A$. For any $f \in X$, we obtain
\begin{align*}
\MoveEqLeft
\bnorm{T_t(f)-P(f)}_{X}
\ov{\eqref{mean-et-Tt}}{=}\bnorm{T_t(\Id-P)(f)}_{X}
\leq \norm{T_t|_{\ovl{\Ran A}}}_{\ovl{\Ran A} \to \ovl{\Ran A}} \bnorm{(\Id-P)}_{X \to \ovl{\Ran A}} \norm{f}_{X} \\
&\ov{\eqref{uniformly-exponentially-stable-bis}}{\leq} M\norm{\Id-P}_{X \to X} \e^{-\omega t} \norm{f}_{X}
\xra[t \to \infty]{} 0.      
\end{align*}
We conclude that $T_t(f) \to P(f)$ as $t \to \infty$. 
\end{proof}

Now, we present some examples which satisfy the condition $0 \in \rho(A_0)$.

\paragraph{Operators acting on Hilbert spaces}
Assume that $H$ is a Hilbert space and that $A$ is a positive selfadjoint operator on $H$. Set $
H_0 \ov{\mathrm{def}}{=} \ovl{\Ran A}$. Then $H_0=(\ker A)^\perp$ by \eqref{lien-ker-image} combined with the results \cite[Proposition 2.6.6 p.~225]{Meg98} and \cite[Theorem 2.5.16 p.~216]{Meg98}. Let $A_0$ be the part of $A$ in $H_0$. Since $A$ is selfadjoint, the closed subspace $H_0$ reduces $A$ (in the sense of \cite[Definition 9.8.1 p.~251]{deO09}),
and $A_0$ is the restriction of $A$ to $H_0$. We define the spectral gap by
\begin{equation}
\label{eq-reduced-spectral-gap}
\gap(A)
\ov{\mathrm{def}}{=}
\inf \left\{
\frac{\la Af,f\ra}{\norm{f}^2_H} : f \in \dom A \cap (\ker A)^\perp, f \neq 0
\right\},
\end{equation}
with the convention that $\inf\varnothing=\infty$.

\begin{prop}
\label{prop-gap-and-s}
Let $H$ be a Hilbert space. Consider a positive selfadjoint operator $A$ acting on $H$. We have $
\gap(A)=\inf \sigma(A_0)$. Consequently, we have
\[
\gap(A) > 0
\quad \Longleftrightarrow \quad
s(-A_0) < 0
\quad \Longleftrightarrow \quad
0 \in \rho(A_0).
\]
\end{prop}

\begin{proof}
If $H_0 = \{0\}$, then $A = 0$. In this degenerate case, $\gap(A)=\infty$ and $0 \in \rho(A_0)$. We use the convention $s(-A_0) = -\infty$. We may therefore suppose that $H_0 \neq \{0\}$. Note that $\dom A_0 = \dom A \cap \ovl{\Ran A} = \dom A \cap (\ker A)^\perp$. Using the Rayleigh quotient \cite[Proposition 15 p.~296]{Bou23} of the selfadjoint operator $A_0$, we obtain
\begin{align*}
\MoveEqLeft
s(-A_0)
\ov{\eqref{spectral-bound}}{=} -\inf \sigma(A_0)
=-\inf \left\{
\frac{\la A_0 f,f\ra}{\norm{f}_H^2}
:
f \in \dom A_0,\ f \neq 0
\right\}
\ov{\eqref{eq-reduced-spectral-gap}}{=} -\gap(A).         
\end{align*}
\end{proof}

The following result is folklore but we are unable to find a suitable reference.

\begin{prop}
\label{prop-Gap-equivalente}
Let $A$ be a positive selfadjoint operator acting on a Hilbert space $H$. Then $\gap(A) = \inf(\sigma(A) \backslash\{0\})$, with the convention that $\inf\varnothing=\infty$. Equivalently, we have
\begin{equation}
\label{eq-gap-largest-spectral-gap}
\gap(A)
=\sup\left\{ c \geq 0 : \sigma(A)\subset\{0\}\cup[c,\infty) \right\}.
\end{equation}
In particular, the following assertions are equivalent:
\begin{enumerate}
\item $\gap(A) > 0$,
\item there exists $c > 0$ such that $\sigma(A) \subset \{0\} \cup [c,\infty)$.
\end{enumerate}
Moreover, if $\gap(A)<\infty$, then
\begin{equation}
\label{eq-spectrum-gap-sharp}
\sigma(A) \subset \{0\} \cup [\gap(A),\infty),
\end{equation}
and $\gap(A)$ is the largest constant $c\geq 0$ for which this inclusion holds.
\end{prop}

\begin{proof}
Set $H_0 \ov{\mathrm{def}}{=} \ovl{\Ran A}$ As observed previously, we have $H_0=(\ker A)^\perp$, the subspace $H_0$ reduces $A$, and the part $A_0$ of $A$ in $H_0$ is the restriction of $A$ to $H_0$. Hence, with respect to the orthogonal decomposition $H = \ker A\oplus H_0$, we have $A=0_{\ker A}\oplus A_0$. Consequently, we have
\begin{equation*}
\sigma(A)
=\sigma(0_{\ker A})\cup\sigma(A_0)
\subset
\{0\}\cup\sigma(A_0).
\end{equation*}
Let $\lambda \in \mathbb C\backslash\{0\}$. With respect to the preceding decomposition, we have
\[
\lambda\Id_H-A
=
\lambda\Id_{\ker A}
\oplus
(\lambda\Id_{H_0}-A_0).
\]
Since $\lambda\Id_{\ker A}$ is invertible, we obtain that $
\lambda \in \rho(A)
\iff
\lambda \in \rho(A_0)$. Consequently, we have
\begin{equation}
\label{eq-spectrum-reduced-away-zero}
\sigma(A)\backslash\{0\}
=
\sigma(A_0)\backslash\{0\}.
\end{equation}
Suppose first that $0 \notin \sigma(A_0)$. Then $
\sigma(A_0)
\overset{\eqref{eq-spectrum-reduced-away-zero}}{=}
\sigma(A)\backslash\{0\}$, and Proposition~\ref{prop-gap-and-s} gives $
\gap(A)
=
\inf\sigma(A_0)
=
\inf\big(\sigma(A)\backslash\{0\}\big)$. Suppose now that $0 \in \sigma(A_0)$. Since $
\ker A_0
=
\ker A\cap H_0
=
\{0\}$, the point $0$ is not an eigenvalue of $A_0$. On the other hand, by \cite[Corollary 5.11 p.~94]{Sch12} every isolated spectral point of a selfadjoint operator is an eigenvalue. Hence $0$ is not isolated in $\sigma(A_0)$. Since $A_0$ is positive, there exists a sequence $(\lambda_n)$ in $\sigma(A_0)\backslash\{0\}$ such that $\lambda_n \to 0$. Therefore
\[
\inf\sigma(A_0)
=
0
=
\inf\big(\sigma(A_0)\backslash\{0\}\big).
\]
Using \eqref{eq-spectrum-reduced-away-zero} and Proposition \ref{prop-gap-and-s}, we again obtain
\[
\gap(A)
=
\inf\sigma(A_0)
=
\inf\big(\sigma(A)\backslash\{0\}\big).
\]
Now let $c \geq 0$. We have $\sigma(A)\subset\{0\}\cup[c,\infty)$ if and only if $
c
\leq
\inf\big(\sigma(A)\backslash\{0\}\big)
=
\gap(A)$. Taking the supremum over all such $c$ gives
\[
\gap(A)
=
\sup\left\{
c\geq0:
\sigma(A)\subset\{0\}\cup[c,\infty)
\right\}.
\]
The remaining assertions follow immediately.
\end{proof}

Now, we describe the link with Abel-ergodic operators.

\begin{lemma}
\label{lem-reduced-invertibility-implies-splitting}
Let $A$ be a sectorial operator on a Banach space $X$. Let $A_0$ be the part of $A$ in $\ovl{\Ran A}$. If $0 \in \rho(A_0)$, then $X = \ker A\oplus\ovl{\Ran A}$, i.e., $A$ is Abel-ergodic.
\end{lemma}

\begin{proof}
Since $A(\Id+A)^{-1} = \Id-(\Id+A)^{-1}$ is bounded from $X$ into $\Ran A \subset \ovl{\Ran A}$, and since $A_0^{-1} \co \ovl{\Ran A} \to \ovl{\Ran A}$ is bounded, the operator
\begin{equation}
\label{def-of-P}
P
\ov{\mathrm{def}}{=} (\Id+A)^{-1}-A_0^{-1} A(\Id+A)^{-1}
\end{equation}
is bounded on $X$. Let $x \in X$. Since $(\Id+A)^{-1}x \in \dom A$ and $A_0^{-1} A(\Id+A)^{-1} x$ belongs to the subspace $\dom A_0 \subset \dom A$, we have $P(x) \in \dom A$. Moreover, we have
\[
APx
=
A (\Id+A)^{-1} x-A A_0^{-1} A(\Id+A)^{-1}  x
=
A (\Id+A)^{-1} x-A_0 A_0^{-1} (A (\Id+A)^{-1} x)
=
0.
\]
Hence $P(x)$ belongs to the space $\ker A$. On the other hand, using the equality $(\Id+A) (\Id+A)^{-1}=\Id$, we obtain that
\[
x-Px
\ov{\eqref{def-of-P}}{=} x-(\Id+A)^{-1}x+ A_0^{-1} A (\Id+A)^{-1} x
= A (\Id+A)^{-1} x+ A_0^{-1} A (\Id+A)^{-1} x.
\]
belongs to the space $\ovl{\Ran A}$. Thus we deduce that $X=\ker A + \ovl{\Ran A}$. Moreover, we have $\ker A \cap \ovl{\Ran A}
=
\ker A_0
=
\{0\}$, since $0 \in \rho(A_0)$. Therefore $X=\ker A\oplus \ovl{\Ran A}$. Since the map $P$ is bounded, this direct sum is topological.
\end{proof}

\paragraph{Generators with compact resolvent}
Let $(T_t)_{t \geq 0}$ be a bounded strongly continuous semigroup with generator $-A$ on a Banach space $X$ such that the operator $-A$ has compact resolvent. According to \cite[Corollary 4.11 p.~344]{EnN00}, the semigroup $(T_t)_{t \geq 0}$ is uniformly mean ergodic and the mean ergodic projection $P \co X \to X$ has finite rank. We have  
\begin{equation}
\label{compact-resolvent}
0 \in \rho(A_0).
\end{equation}

\begin{example}
\normalfont
Let $\Omega$ be a finite measure space. By \cite[Theorem 6.6.6 p.~626]{Sim15} combined with \cite[Theorem 4.29 p.~119]{EnN00}, the generator $A$ on $\L^2(\Omega)$ of an ultracontractive semigroup, i.e., each $T_t$ induces a bounded operator $T_t \co \L^2(\Omega) \to \L^\infty(\Omega)$, has a compact resolvent. We refer to \cite{Arh24b} for a generalization.
\end{example}

\paragraph{Hypercontractive semigroups}
Let $\cal{M}$ be a von Neumann algebra equipped with a normalized normal finite faithful trace $\tau$. An $\L^p$-contractive semigroup $(T_t)_{t \geq 0}$ is a strongly continuous semigroup of contractions acting on the noncommutative $\L^2$-space $\L^2(\cal{M})$ such that $\norm{T_tf}_{\L^p(\cal{M})} \leq \norm{f}_{\L^p(\cal{M})}$ for any $f \in \cal{M}$ and any $p \in [1,\infty]$. Following \cite[Definition~2.3 p.~10]{JLZZ24}, we introduce the following definition. This definition is more restrictive than the one used in \cite[p.~618]{Sim15} and \cite[p.~371]{DGS92}, where the boundedness of the operator $T_{t_0} \co \L^2(\cal{M}) \to \L^4(\cal{M})$ is required.

\begin{defi}
Let $\cal{M}$ be a von Neumann algebra equipped with a normalized normal finite faithful trace $\tau$.  An $\L^p$-contractive semigroup $(T_t)_{t \geq 0}$ is called hypercontractive if there exists some $t_0 > 0$ such that $T_{t_0}$ induces a contractive operator from $\L^2(\cal{M})$ into the space $\L^4(\cal{M})$.
%
\end{defi}

A weak* continuous contraction $T \co \cal{M} \to \cal{M}$ is said to be selfadjoint \cite[p.~49]{JMX06} if $\tau(T(x)y^*)=\tau(xT(y)^*)$ for any $x, y \in \cal{M}$. By \cite[p.~49]{JMX06}, such an operator induces a contraction $T \co \L^p(\cal{M}) \to \L^p(\cal{M})$ for any $1 \leq p < \infty$, which is  selfadjoint if $p=2$.

A weak* continuous semigroup $(T_t)_{t \geq 0}$ of operators acting on $\cal{M}$ is said to be a Markov semigroup if each $T_t$ is a weak* continuous selfadjoint unital completely positive operator on $\cal{M}$. Such a semigroup induces a strongly continuous semigroup of contractions on the noncommutative $\L^p$-space $\L^p(\cal{M})$ for any $1 \leq p < \infty$, which is bounded holomorphic by \cite[Proposition 5.4 p.~51, Lemma 3.1 p.~26]{JMX06} if $p>1$. Moreover, for any element $x \in \cal{M}$, we have 
\begin{equation}
\label{trace-preserving}
\tau(T_t(x))
=\tau(T_t(x)1)
=\tau(xT_t(1)) 
= \tau(x).
\end{equation}
If $-A$ is the generator on $\L^2(\cal{M})$, then $A$ is a positive selfadjoint operator on the Hilbert space $\L^2(\cal{M})$ and each $T_t$ is selfadjoint on $\L^{2}(\cal{M})$. In particular, we have an $\L^p$-contractive semigroup. 

By \cite[Theorem 2.4]{KuN79}  (see also \cite[Corollary 1.27 p.~48]{Hel09}), such a semigroup is weak* mean ergodic and the corresponding projection onto the weak* closed fixed-point subalgebra $\{ x \in \cal{M} : T_t(x)=x \text{ for any } t \geq 0 \}$ is a conditional expectation $\E \co \cal{M} \to \cal{M}$. By \cite[Theorem 1.2]{KuN79}, it also satisfies the equalities
\begin{equation}
\label{E-Tt}
T_t\E
=\E T_t
=\E, \quad t \geq 0.
\end{equation}
Since $\E$ belongs to the closed convex hull of the set $\{T_t : t \geq 0\}$ in the point weak* topology, a standard argument using complex interpolation shows that the map $\E \co \L^\infty(\cal{M}) \to \L^\infty(\cal{M})$ admits a contractive $\L^p$-extension $\E_p \co \L^p(\cal{M}) \to \L^p(\cal{M})$ for any $1 \leq p \leq \infty$. We will use the classical notation  
$\L_0^p(\cal{M})
\ov{\mathrm{def}}{=} \ker \E_p$. 
If $-A_p$ is the infinitesimal generator of the strongly continuous semigroup $(T_{t})_{t \geq 0}$ on the Banach space $\L^p(\cal{M})$ then we have $\L_0^p(\cal{M})=\ovl{\Ran A_p}$ and the decomposition $\L^p(\cal{M})=\L_0^p(\cal{M}) \oplus \ker A_p$, where $1 \leq p < \infty$. In other words, $\ker A_p$ is contractively complemented by the contractive projection $\E_p \co \L^p(\cal{M}) \to \L^p(\cal{M})$ which is defined by interpolation by the two compatible contractive projections $\E \co \L^\infty(\cal{M}) \to \L^\infty(\cal{M})$ and $\E_1 \co \L^1(\cal{M}) \to \L^1(\cal{M})$. The space $\L_0^p(\cal M)=\ker\E_p$ is complemented by $Q_p\ov{\mathrm{def}}{=}\Id-\E_p$, with
$\norm{Q_p}_{\L^p(\cal M)\to\L^p(\cal M)}\leq2$.


%

Using a standard perturbation argument around the constant function $1$, we can prove the following quantitative spectral gap estimate. The argument refines \cite[Lemmas 5.1--5.2, pp.~22--23]{JLZZ24} by exploiting the fact that the Markov semigroup preserves selfadjointness and the trace. This yields the sharper spectral gap $\frac{\log 3}{2t_0}$ instead of the lower bound $\frac{1}{4t_0}$ stated in \cite[Lemma 5.2 p.~22]{JLZZ24}.

\begin{prop}
\label{prop-hypercontractivity-spectral-gap}
Let $(T_t)_{t \geq 0}$ be a Markov semigroup which is hypercontractive. Let $t_0 > 0$ be such that $
T_{t_0} \co \L^2(\cal M) \to \L^4(\cal M)$ is a contraction. Then $
\ker A_2=\mathbb{C}1$. Moreover, we have
\begin{equation}
\label{spectral-gap-hypercontractive}
\sigma(A_2)
\subset
\{0\}
\cup
\left[
\frac{\log 3}{2t_0},\infty
\right).
\end{equation}
In particular, we have $\gap(A_2)
\geq
\frac{\log 3}{2t_0}$.
\end{prop}

\begin{proof}
We first prove an estimate for selfadjoint elements of trace zero. Let $x=x^* \in \cal M$ satisfy $\tau(x)=0$. Since the operator $T_{t_0}$ is unital, we have
\[
T_{t_0}(1+\epsi x)
=
1+\epsi T_{t_0}(x),
\quad
\epsi \in \R.
\]
Moreover, $T_{t_0}(x)$ is selfadjoint since $T_{t_0}$ is positive. By the trace-preserving property \eqref{trace-preserving}, we also have $\tau\big(T_{t_0}(x)\big)=\tau(x)=0$. Since $T_{t_0} \co \L^2(\cal M) \to \L^4(\cal M)$ is contractive, for any $\varepsilon \in \R$ we obtain
\begin{equation}
\label{hypercontractive-epsilon}
\bnorm{1+\varepsilon T_{t_0}(x)}_{\L^4(\cal M)}^4
\leq
\bnorm{1+\varepsilon x}_{\L^2(\cal M)}^4.
\end{equation}
Since $T_{t_0}(x)$ is selfadjoint, the left-hand side satisfies
\begin{align*}
\MoveEqLeft
\bnorm{1+\varepsilon T_{t_0}(x)}_{\L^4(\cal M)}^4 
=
\tau\left(
\big(1+\varepsilon T_{t_0}(x)\big)^4
\right)\\
&=
1
+4\varepsilon\tau\big(T_{t_0}(x)\big)
+6\varepsilon^2\tau\big(T_{t_0}(x)^2\big)
+4\varepsilon^3\tau\big(T_{t_0}(x)^3\big)
+\varepsilon^4\tau\big(T_{t_0}(x)^4\big)\\
&=
1
+6\varepsilon^2
\bnorm{T_{t_0}(x)}_{\L^2(\cal M)}^2
+4\varepsilon^3\tau\big(T_{t_0}(x)^3\big)
+\varepsilon^4
\bnorm{T_{t_0}(x)}_{\L^4(\cal M)}^4.
\end{align*}
On the other hand, since $x=x^*$ and $\tau(x)=0$, we have
\[
\bnorm{1+\varepsilon x}_{\L^2(\cal M)}^2
=
\tau\left((1+\varepsilon x)^2\right)
=
1+\varepsilon^2\norm{x}_{\L^2(\cal M)}^2.
\]
Consequently, we have
\begin{equation}
\label{expansion-L2-hypercontractive}
\bnorm{1+\varepsilon x}_{\L^2(\cal M)}^4
=
1
+2\varepsilon^2\norm{x}_{\L^2(\cal M)}^2
+\varepsilon^4\norm{x}_{\L^2(\cal M)}^4.
\end{equation}
Combining \eqref{hypercontractive-epsilon} and
\eqref{expansion-L2-hypercontractive}, subtracting $1$, dividing by
$\varepsilon^2$ for $\varepsilon\neq0$, and letting
$\varepsilon\to0$, we obtain
\[
6\bnorm{T_{t_0}(x)}_{\L^2(\cal M)}^2
\leq
2\norm{x}_{\L^2(\cal M)}^2.
\]
Hence
\begin{equation}
\label{L2-contraction-centered-selfadjoint}
\bnorm{T_{t_0}(x)}_{\L^2(\cal M)}
\leq
\frac1{\sqrt{3}}
\norm{x}_{\L^2(\cal M)}
\end{equation}
for any selfadjoint element $x\in\cal M$ satisfying $\tau(x)=0$. Now, we extend \eqref{L2-contraction-centered-selfadjoint} to arbitrary elements of trace zero. Let $x\in\cal M$ satisfy $\tau(x)=0$ and write
\[
x=a+\i b,
\qquad
a\ov{\mathrm{def}}{=}\frac{x+x^*}{2},
\qquad
b\ov{\mathrm{def}}{=}\frac{x-x^*}{2\i}.
\]
Then $a=a^*$, $b=b^*$ and $
\tau(a)=\tau(b)=0$. Since $\tau$ is tracial, we have
\begin{align*}
\MoveEqLeft
\norm{x}_{\L^2(\cal M)}^2
=
\tau(x^*x)
=\tau\big((a-\i b)(a+\i b)\big)
=
\tau(a^2)+\tau(b^2)+\i\tau(ab-ba)
=
\norm{a}_{\L^2(\cal M)}^2
+
\norm{b}_{\L^2(\cal M)}^2.
\end{align*}
Similarly, since $T_{t_0}$ preserves selfadjoint elements, we have
\[
\bnorm{T_{t_0}(x)}_{\L^2(\cal M)}^2
=
\bnorm{T_{t_0}(a)}_{\L^2(\cal M)}^2
+
\bnorm{T_{t_0}(b)}_{\L^2(\cal M)}^2.
\]
Using \eqref{L2-contraction-centered-selfadjoint} for $a$ and $b$, we infer that
\begin{equation}
\label{L2-contraction-centered}
\bnorm{T_{t_0}(x)}_{\L^2(\cal M)}
\leq
\frac1{\sqrt{3}}
\norm{x}_{\L^2(\cal M)}
\end{equation}
for any $x \in \cal M$ satisfying $\tau(x)=0$. Set $
\L^2_{\tau,0}(\cal M)
\ov{\mathrm{def}}{=}
\big\{
x\in\L^2(\cal M):\tau(x)=0
\big\}$. The subspace $
\cal M \cap \L^2_{\tau,0}(\cal M)$ is dense in $\L^2_{\tau,0}(\cal M)$. Indeed, if $x \in \L^2_{\tau,0}(\cal M)$ and $x_n \in \cal M$ satisfies $x_n\to x$ in $\L^2(\cal M)$, then $
x_n-\tau(x_n)1
\to
x$ in $\L^2(\cal M)$, since
\[
|\tau(x_n)|
=
|\tau(x_n-x)|
\leq
\norm{x_n-x}_{\L^2(\cal M)}.
\]
Therefore \eqref{L2-contraction-centered} extends by density to
\begin{equation}
\label{L2-contraction-centered-full}
\bnorm{T_{t_0}(x)}_{\L^2(\cal M)}
\leq
\frac1{\sqrt{3}}
\norm{x}_{\L^2(\cal M)},
\quad
x \in \L^2_{\tau,0}(\cal M).
\end{equation}
Now, we identify the kernel of $A_2$. Let $x \in \ker A_2$. Then $
T_t(x)=x$ for any $t \geq 0$. Since $T_t(1)=1$, the element $
x_0
\ov{\mathrm{def}}{=}
x-\tau(x)1$ also satisfies
\[
T_t(x_0)=x_0,
\quad
t \geq 0,
\]
and belongs to $\L^2_{\tau,0}(\cal M)$. Applying \eqref{L2-contraction-centered-full} gives
\[
\norm{x_0}_{\L^2(\cal M)}
=
\bnorm{T_{t_0}(x_0)}_{\L^2(\cal M)}
\leq
\frac1{\sqrt{3}}
\norm{x_0}_{\L^2(\cal M)}.
\]
Hence $x_0=0$. Thus $x=\tau(x)1$, and therefore $
\ker A_2=\mathbb{C}1$. Consequently, we have $
\L^2_{\tau,0}(\cal M)
=
(\mathbb{C}1)^\perp
=
(\ker A_2)^\perp$. Since $A_2$ is positive selfadjoint, this closed subspace reduces $A_2$. Let $A_{2,0}$ denote the restriction of $A_2$ to $\L^2_{\tau,0}(\cal M)$. By \eqref{L2-contraction-centered-full}, we have
\[
\bnorm{\e^{-t_0A_{2,0}}}_{\L^2_{\tau,0}(\cal M)
\to\L^2_{\tau,0}(\cal M)}
\leq
\frac1{\sqrt{3}}.
\]
If $\L^2_{\tau,0}(\cal M)=\{0\}$, the spectral assertion is immediate. Otherwise, by the spectral theorem, we have
\begin{align*}
\bnorm{\e^{-t_0A_{2,0}}}_{\L^2_{\tau,0}(\cal M)
\to\L^2_{\tau,0}(\cal M)}
&=
\sup_{\lambda\in\sigma(A_{2,0})}
\e^{-t_0\lambda}
=
\e^{-t_0\inf\sigma(A_{2,0})}.
\end{align*}
It follows that $
\e^{-t_0\inf\sigma(A_{2,0})}
\leq
\frac1{\sqrt{3}}$. Taking logarithms yields $
\inf\sigma(A_{2,0})
\geq
\frac{\log3}{2t_0}$. Since $
\L^2(\cal M)
=
\mathbb{C}1
\oplus
\L^2_{\tau,0}(\cal M)$ is a reducing decomposition for $A_2$, we conclude that $
\sigma(A_2)
\subset
\{0\}
\cup
\left[
\frac{\log3}{2t_0},\infty
\right)$. 
\end{proof}

In particular, the hypercontractivity assumption implies a spectral gap. Let $A_{2,0}$ be the part of $A_2$ in $\L^2_0(\cal M)$. Since $\L^2_0(\cal{M})$ reduces $A_2$, the operator $A_{2,0}$ is simply the restriction of $A_2$ to $\L^2_0(\cal{M})$. Combining \cite[Theorem 8.3.13 p.~215]{deO09} with Proposition \ref{prop-hypercontractivity-spectral-gap}, we obtain the inclusion $
\sigma(A_{2,0})
\subset
\left[
\frac{\log3}{2t_0},\infty
\right)$.
Consequently, we have 
\begin{equation}
\label{hyer-gap}
0 \in \rho(A_{2,0}).
\end{equation}

The lower bound $\frac{\log 3}{2t_0}$ in Proposition \ref{prop-hypercontractivity-spectral-gap} is optimal, as shown by the following example.

\begin{example} \normalfont
\label{rem-optimal-hypercontractive-gap}
\normalfont
Indeed, consider the probability space $
\Omega
\ov{\mathrm{def}}{=}
\{-1,1\}$ equipped with the uniform probability measure $
\mu
\ov{\mathrm{def}}{=}
\frac12\delta_{-1}+\frac12\delta_1$, and the finite commutative von Neumann algebra $
\cal{M}
\ov{\mathrm{def}}{=}
\L^\infty(\Omega,\mu)$ equipped with the normalized trace
$
\tau(f)
\ov{\mathrm{def}}{=}
\frac12\big(f(-1)+f(1)\big)$ where $
f\in\cal{M}$. Let $\varepsilon \co \Omega\to\{-1,1\}$ be the function defined by $
\varepsilon(\omega)
\ov{\mathrm{def}}{=}
\omega$ where $\omega \in \Omega$. Any element $f \in \cal{M}$ admits a unique representation $f = a+b\epsi$, with $a,b \in \mathbb C$. Fix $\lambda > 0$. We consider the Markov semigroup $(T_t)_{t \geq 0}$ acting on $\cal{M}$ by
\begin{equation}
\label{eq-two-point-Markov-semigroup}
T_t(a+b\epsi)
\ov{\mathrm{def}}{=}
a+\e^{-\lambda t}b\epsi,
\quad
a,b \in \mathbb C,
\quad
t\geq0.
\end{equation}
Equivalently, if $f \in \cal{M}$, then
\[
T_t(f)
=
\tau(f)1
+
\e^{-\lambda t}\big(f-\tau(f)1\big).
\]
In particular, each $T_t$ is unital, positive, trace preserving and selfadjoint on $\L^2(\Omega,\mu)$. Since $\cal{M}$ is commutative, each $T_t$ is completely positive. Hence $(T_t)_{t\geq0}$ is a Markov semigroup on $\cal{M}$, and by consistency it acts contractively on every space $\L^p(\Omega,\mu)$, where $1\leq p\leq\infty$.

Its infinitesimal generator $-A_p$ on the Banach space $\L^p(\Omega,\mu)$ is given by $
A_p(a+b\epsi)
=
\lambda b\epsi$. Consequently, we have $
\ker A_2=\mathbb C1$, $
\L^2_0(\Omega,\mu)=\mathbb C\epsi$, and $
A_{2,0}
=
\lambda\Id_{\mathbb C\epsi}$. Thus
\begin{equation}
\label{eq-two-point-gap}
\gap(A_2)=\lambda.
\end{equation}
Put $
c
\ov{\mathrm{def}}{=}
\e^{-\lambda t_0}$. For any $a,b \in \mathbb C$, we have $\norm{a+b\epsi}_{\L^2(\Omega,\mu)}^2
=
|a|^2+|b|^2$. On the other hand, we have
\begin{align}
\MoveEqLeft
\label{inter-fin}
\bnorm{T_{t_0}(a+b\epsi)}_{\L^4(\Omega,\mu)}^4 
=
\frac12
\left(
|a+cb|^4+|a-cb|^4
\right) \\
&=
\big(|a|^2+c^2|b|^2\big)^2
+
4c^2\big(\Re(\overline{a}b)\big)^2
\leq
|a|^4
+
6c^2|a|^2|b|^2
+
c^4|b|^4. \nonumber
\end{align}
If $
c
=
\frac1{\sqrt3}$, then
\begin{align*}
\bnorm{T_{t_0}(a+b\epsi)}_{\L^4(\Omega,\mu)}^4
\ov{\eqref{inter-fin}}{\leq}
|a|^4
+
2|a|^2|b|^2
+
\frac19|b|^4
\leq
\big(|a|^2+|b|^2\big)^2
=\norm{a+b\epsi}_{\L^2(\Omega,\mu)}^4.
\end{align*}
Hence $
\norm{T_{t_0}}_{\L^2(\Omega,\mu) \to \L^4(\Omega,\mu)}
\leq1$ whenever $
\e^{-\lambda t_0}
=
\frac1{\sqrt3}$. Choosing $
\lambda
=
\frac{\log3}{2t_0}$ we therefore obtain a Markov semigroup for which $T_{t_0}$ is contractive from $\L^2(\Omega,\mu)$ into $\L^4(\Omega,\mu)$ and, by \eqref{eq-two-point-gap},
\[
\gap(A_2)
=
\frac{\log3}{2t_0}.
\]
\end{example}

\section{Banach Poincar\'e inequalities}
\label{Sec-Poincare}

The following result is the main result of this paper.

\begin{thm}
\label{thm-direct-Poincare-via-negative-powers}
Consider some $\alpha \in (0,1)$. Let $X$ and $Y$ be Banach spaces. Consider a sectorial operator $A$. Let $A_0$ be the part of $A$ in $\ovl{\Ran A}$. Suppose that $A$ admits a spectral gap, i.e., $0 \in \rho(A_0)$. Denote by $P \co X \to X$ on $\ker A$ its ergodic projection provided by Lemma \ref{lem-reduced-invertibility-implies-splitting}. Then
\begin{equation}
\label{main}
\norm{x-P(x)}_X
\leq
\bnorm{A_0^{-\alpha}}_{\ovl{\Ran A} \to \ovl{\Ran A}}
\norm{A^\alpha x}_X, \quad x \in \dom A^\alpha.
\end{equation}
Consider, in addition, an unbounded operator $\partial \co \dom \partial \subset X \to Y$ such that $\dom\partial \subset \dom A^{\alpha}$ and assume that the reverse Riesz estimate
\begin{equation}
\label{eq-lower-Riesz-direct-Poincare}
\bnorm{A^{\alpha}(x)}_X
\leq
C\norm{\partial(x)}_Y,
\quad x \in \dom\partial
\end{equation}
holds for some constant $C \geq 0$. Then, for any $x \in \dom\partial$, we have
\begin{equation}
\label{eq-direct-Poincare}
\norm{x-P(x)}_X
\leq
C\bnorm{A_0^{-\alpha}}_{\ovl{\Ran A} \to \ovl{\Ran A}}\norm{\partial(x)}_Y.
\end{equation}
\end{thm}

\begin{proof}
Let $x \in \dom A^\alpha$. Note that $P(x) \in \ker A$ and that $x-P(x)$ belongs to $\ovl{\Ran A} \ov{\eqref{inclusion-range}}{=} \ovl{\Ran A^\alpha}$. Since $\ker A \ov{\eqref{inclusion-range}}{=} \ker A^\alpha$, the element $P(x)$ belongs to the subspace $\dom A^{\alpha}$ and we have $A^{\alpha} P(x) =0$. We deduce that $x-P(x)$ belongs to the subspace $\dom A^{\alpha} \cap \ovl{\Ran A^\alpha}$
and
\begin{equation}
\label{equa-divers-34}
A_0^{\alpha}(x-P(x))
=A^{\alpha}(x-P(x))
=A^\alpha x-A^{\alpha} P(x)
=A^\alpha x,
\end{equation}
where we use \cite[Proposition 3.1.1 (i) p.~61]{Haa06} in the first equality. Since \(0 \in \rho(A_0)\), the operator $A_0^{-1} \co \ovl{\Ran A} \to \ovl{\Ran A}$ is a well-defined bounded operator. By \cite[Corollary 15.2.10 p.~444]{HvNVW23} or \cite[Proposition 3.2.3 p.~72]{Haa06}, the negative power $A_0^{-\alpha}$ is bounded on the Banach space $\ovl{\Ran A}$. Therefore
\begin{equation}
\label{inter-67}
x-P(x)
=
A_0^{-\alpha}A_0^{\alpha}(x-P(x))
\ov{\eqref{equa-divers-34}}{=}
A_0^{-\alpha}A^{\alpha}x.
\end{equation}
We deduce that
\begin{align*}
\MoveEqLeft
\norm{x-P(x)}_X
\ov{\eqref{inter-67}}{=} \norm{A_0^{-\alpha}A^{\alpha}x}_X 
\leq
\bnorm{A_0^{-\alpha}}_{\ovl{\Ran A} \to \ovl{\Ran A}}
\norm{A^\alpha x}_X,
\end{align*}
that is \eqref{main}. Consequently, using the reverse Riesz estimate \eqref{eq-lower-Riesz-direct-Poincare}, we obtain if $x \in \dom \partial$
$$
\norm{x-P(x)}_X
\ov{\eqref{inter-67}}{=} \norm{A_0^{-\alpha}A^{\alpha}x}_X 
\leq C\bnorm{A_0^{-\alpha}}_{\ovl{\Ran A} \to \ovl{\Ran A}}
\norm{\partial(x)}_Y.
$$
\end{proof}


\section{Banach divergence inequalities}
\label{sec-divergence}

Consider some Banach spaces $X$ and $Y$. Consider a sectorial operator $A$ with dense domain acting on $X$. Following \cite{Arh26b}, an abstract pair $(\partial,\partial^\dagger)$ of gradient and divergence operators compatible with $A$ consists of a closed operator $\partial \co \dom \partial \subset X \to Y$ with dense domain and a closed operator
$\partial^\dagger \co \dom \partial^\dagger \subset Y \to X$ with dense domain such that
\begin{equation}
\label{facto-with-diamond}
\dom A
=\big\{x \in \dom \partial : \partial x \in \dom \partial^\dagger \big\},
\qquad
Ax
=\partial^\dagger \partial x
\quad \text{for all } x \in \dom A.
\end{equation}

The formulation of the following result was inspired by the inequality proved in \cite[Theorem 2.9]{Nee15}. If $\Omega$ is a bounded Lipschitz domain with mixed boundary conditions defined on boundary parts $\Gamma_\tau$ and $\Gamma_\nu$ of the boundary $\Gamma$, \cite[p.~2]{PaV20} presents the similar inequality $\norm{H}_{\L^2(\Omega,\mathbb C^n)} \lesssim \norm{\div H}_{\L^2(\Omega)}$ for any $H \in
\dom \div_{\Gamma_\tau} \cap \ovl{\Ran \grad_{\Gamma_\nu}}$.
 
\begin{thm}
\label{thm-Poincare-divergence-corrected}
Let $X$ and $Y$ be Banach spaces. Consider some sectorial operator $A$. Let $A_0$ be the part of $A$ in $\ovl{\Ran A}$. Assume that $0 \in \rho(A_0)$. Let $(\partial,\partial^\dagger)$ be a compatible pair with $A$. Suppose that $0 < \alpha < 1$. Assume that $\dom A^\alpha \subset \dom\partial$ and that there exists a constant $K \geq 0$ such that
\begin{equation}
\label{eq-Riesz-half-upper}
\norm{\partial x}_{Y}
\leq K \bnorm{A^{\alpha}(x)}_{X},
\quad x \in \dom A^{\alpha}.
\end{equation}
Then
\begin{equation}
\label{eq-Poincare-div-core}
\norm{y}_{Y}
\leq K \bnorm{A_0^{-(1-\alpha)}}_{\ovl{\Ran(A)} \to \ovl{\Ran(A)}}
\norm{\partial^\dagger(y)}_{X},
\quad y \in \partial(\dom A).
\end{equation}
Let \(\partial^\dagger|_{\ovl{\Ran \partial}}\) denote the restriction of \(\partial^\dagger\) to the closed subspace \(\ovl{\Ran\partial}\), with domain
\[
\dom \partial^\dagger|_{\ovl{\Ran \partial}}
\ov{\mathrm{def}}{=}
\dom(\partial^\dagger)\cap \ovl{\Ran\partial}.
\]
If, in addition, the subspace $\partial(\dom A)$ is a core of the operator $\partial^\dagger|_{\ovl{\Ran \partial}}$, then
\begin{equation}
\label{eq-Poincare-div-full}
\norm{y}_{Y}
\leq K \bnorm{A_0^{-(1-\alpha)}}_{\ovl{\Ran(A)} \to \ovl{\Ran(A)}}\norm{\partial^\dagger(y)}_{X},
\quad y \in \dom(\partial^\dagger)\cap \ovl{\Ran\partial}.
\end{equation}
\end{thm}

\begin{proof}
We denote by $P \co X \to X$ the ergodic projection on the subspace $\ker A$. Note that the operator $\Id-P$ is the projection onto $\ovl{\Ran A}$ along the subspace $\ker A$. By \cite[p.~24]{Haa06}, the operator $A_0$ is sectorial on the subspace $\ovl{\Ran A}$ and, by assumption, $0 \in \rho(A_0)$. Hence $A_0^{-(1-\alpha)}$ is a bounded operator on $\ovl{\Ran A}$ by \cite[Corollary 15.2.10 p.~444]{HvNVW23} or \cite[Proposition 3.2.3 p.~72]{Haa06}. Let $x \in \dom A$ and set $x_0 \ov{\mathrm{def}}{=} (\Id-P)x=x-P(x)$. Note that $P(x)$ belongs to $\ker A$, which is a subspace of $\dom A$. Thus $x_0$ belongs to $\dom A \cap \ovl{\Ran A}$, i.e., $x_0 \in \dom A_0$ and
$$
A(x)
=A(x_0+P(x))
=A(x_0)
=A_0(x_0).
$$
Using \eqref{inclusion-range}, we see that $A^{\alpha}P(x)=0$. We get
\[
\norm{\partial P(x)}_{Y}
\ov{\eqref{eq-Riesz-half-upper}}{\leq}  K \bnorm{A^{\alpha}P(x)}_{X}
=0.
\]
Thus $\partial P(x)=0$. Hence
\begin{equation}
\label{inter-720}
\partial(x)
=\partial((\Id-P+P)x)
=\partial((\Id-P)x)+\partial P(x)
=\partial((\Id-P)x)
=\partial(x_0).
\end{equation}
Taking the norms, we obtain
\[
\norm{\partial(x)}_{Y}
\ov{\eqref{inter-720}}{=}
\norm{\partial(x_0)}_{Y}
\ov{\eqref{eq-Riesz-half-upper}}{\leq}
K \bnorm{A^{\alpha}(x_0)}_{X}.
\]
Since $x_0 \in \dom A_0$ and since the fractional powers of $A$ and $A_0$ agree by \cite[Proposition 3.1.1 (i) p.~61]{Haa06} on the subspace $\ovl{\Ran(A)}$, we have $A^{\alpha}(x_0)=A_0^{\alpha}(x_0)=A_0^{-(1-\alpha)}A_0(x_0)$. Consequently, we obtain
\begin{align*}
\MoveEqLeft
\norm{\partial(x)}_{Y}
\leq
K \bnorm{A_0^{-(1-\alpha)}}_{\ovl{\Ran(A)} \to \ovl{\Ran(A)}}
\norm{A_0(x_0)}_{X} 
=
K \bnorm{A_0^{-(1-\alpha)}}_{\ovl{\Ran(A)} \to \ovl{\Ran(A)}}
\norm{A(x)}_{X}  \\       
&\ov{\eqref{facto-with-diamond}}{=} K \bnorm{A_0^{-(1-\alpha)}}_{\ovl{\Ran(A)} \to \ovl{\Ran(A)}}
\norm{\partial^\dagger\partial(x)}_{X}.
\end{align*}
This proves \eqref{eq-Poincare-div-core} for all $y \in \partial(\dom A)$.

Now, assume that the subspace $\partial(\dom A)$ is a core of the operator $\partial^\dagger|_{\ovl{\Ran \partial}}$. Let $y \in \dom(\partial^\dagger) \cap \ovl{\Ran \partial}$. Then there exists a sequence $(x_n)$ of elements in $\dom A$ such that $\partial(x_n) \to y$ in $Y$ and $\partial^\dagger\partial(x_n) \to \partial^\dagger(y)$ in $X$. Applying \eqref{eq-Poincare-div-core} to $\partial(x_n)$, we get
\[
\norm{\partial(x_n)}_{Y}
\leq
K \bnorm{A_0^{-(1-\alpha)}}_{\ovl{\Ran(A)} \to \ovl{\Ran(A)}}
\norm{\partial^\dagger\partial(x_n)}_{X}.
\]
Passing to the limit yields
\[
\norm{y}_{Y}
\leq
K \bnorm{A_0^{-(1-\alpha)}}_{\ovl{\Ran(A)} \to \ovl{\Ran(A)}}
\norm{\partial^\dagger(y)}_{X}.
\]
This proves \eqref{eq-Poincare-div-full}.
\end{proof}

\section{Riesz estimates and reverse Riesz estimates} 
\label{sec-Riesz-equivalence}

%
%
%
%
Inspired by \cite[Definition 3.3]{Arh26b}, we introduce the following definition. 

\begin{defi}
Consider some Banach spaces $X$ and $Y$. Consider a sectorial operator $A$ with dense domain acting on $X$. Let $\alpha \in (0,1)$. Let $\partial \co \dom \partial \subset X \to Y$ be an unbounded operator.
\begin{enumerate}
\item We say that the operator $A$ admits the $(\alpha,\partial)$-Riesz estimate if $A^{\alpha}$ admits a core $C$ which is a subspace of $\dom \partial$ satisfying
\begin{equation}
\label{estim-Riesz-123}
\bnorm{\partial(x)}_{Y} 
\lesssim \norm{A^{\alpha}(x)}_{X}, \quad  x \in C.
\end{equation} 
\item We say that the operator $A$ admits the reverse $(\alpha,\partial)$-Riesz estimate if $\dom \partial$ is a subspace of $\dom A^{\alpha}$ and if
\begin{equation}
\label{estim-Riesz-bis}
\bnorm{A^{\alpha}(x)}_{X} 
\lesssim \norm{\partial(x)}_{Y}, \quad  x \in \dom \partial.
\end{equation} 
\end{enumerate}
\end{defi}

Now, we prove some elementary consequences of these estimates in the case where the unbounded operator $\partial$ is closable.

\begin{prop}
\label{Prop-derivation-closable-sgrp-bis}
Consider some Banach spaces $X$ and $Y$. Consider a sectorial operator $A$ with dense domain acting on $X$. Let $\alpha \in (0,1)$. Assume that the unbounded operator $\partial \co \dom \partial \subset X \to Y$ is closable. 

\begin{enumerate}
\item If $A$ admits the $(\alpha,\partial)$-Riesz estimate, we have $\dom A^{\alpha} \subset \dom \ovl{\partial} $. Moreover, we have
\begin{equation} 
\label{alpha-Riesz}
\bnorm{\ovl{\partial}(x)}_{Y} 
\lesssim \bnorm{A^{\alpha}(x)}_{X}, \quad x \in \dom A^{\alpha}. 
\end{equation}

\item If $A$ admits the reverse $(\alpha,\partial)$-Riesz estimate, we have $\dom \ovl{\partial} \subset \dom A^{\alpha}$. Moreover, we have
\begin{equation} 
\label{Equivalence-square-root-domaine-Schur-sgrp-bis}
\bnorm{A^{\alpha}(x)}_{X} 
\lesssim \bnorm{\ovl{\partial}(x)}_{Y}, \quad x \in \dom \ovl{\partial}. 
\end{equation}	
\end{enumerate}
\end{prop}

\begin{proof}
1. Let $x \in \dom A^{\alpha}$. The subspace $C$ is a core of the operator $A^{\alpha}$. 
So we can find a sequence $(x_n)$ of elements of $C$ such that $x_n \to x$ and $A^{\alpha}(x_n) \to A^{\alpha}(x)$. For any integers $n,m \geq 1$, we obtain
\begin{align*}
\MoveEqLeft
\norm{x_n-x_m}_{X} + \bnorm{\partial(x_n) - \partial(x_m)}_{Y} 
\ov{\eqref{estim-Riesz-123}}{\lesssim} \norm{x_n-x_m}_{X} + \norm{A^{\alpha}(x_n)-A^{\alpha}(x_m)}_{X}
\end{align*} 
which shows that $(x_n)$ is a Cauchy sequence in $\dom \ovl{\partial}$ equipped with the graph norm. Since $\ovl{\partial}$ is closed, the space $\dom \ovl{\partial}$ is complete for the graph norm. We infer that there exists $x' \in \dom \ovl{\partial}$ such that $x_n \to x'$ and $\ovl{\partial}(x_n) \to \ovl{\partial}(x')$. By uniqueness of the limit, we have $x=x'$. It follows that $x \in \dom \ovl{\partial}$. This proves the inclusion $\dom A^{\alpha} \subset \dom \ovl{\partial}$. Moreover, for any integer $n \geq 1$, we have $
\bnorm{\partial(x_n)}_{Y} 
\ov{\eqref{estim-Riesz-123}}{\lesssim} \norm{A^{\alpha}(x_n)}_{X}$. Since $x_n \to x$ in $\dom \ovl{\partial}$ and in $\dom A^{\alpha}$ both equipped with the graph norm, we conclude that $
\bnorm{\ovl{\partial}(x)}_{Y} 
\lesssim \bnorm{A^{\alpha}(x)}_{X}$. 

2. Let $x \in \dom \ovl{\partial}$. The subspace $\dom \partial$ is a core of the operator $\ovl{\partial}$. 
So we can find a sequence $(x_n)$ of elements of $\dom \partial$ such that $x_n \to x$ and $\partial(x_n) \to \ovl{\partial}(x)$. For any integers $n,m \geq 1$, we obtain
\begin{align*}
\MoveEqLeft
\norm{x_n-x_m}_{X} + \bnorm{A^{\alpha}(x_n) - A^{\alpha}(x_m)}_{X} 
\ov{\eqref{estim-Riesz-bis}}{\lesssim} \norm{x_n-x_m}_{X} + \norm{\partial(x_n)-\partial(x_m)}_{Y}
\end{align*} 
which shows that $(x_n)$ is a Cauchy sequence in $\dom A^{\alpha}$ equipped with the graph norm. Since $A^{\alpha}$ is closed, the space $\dom A^{\alpha}$ is complete for the graph norm. We infer that there exists $x' \in \dom A^{\alpha}$ such that $x_n \to x'$ and $A^{\alpha}(x_n) \to A^{\alpha}(x')$. By uniqueness of the limit, we have $x=x'$. It follows that $x \in \dom A^{\alpha}$. This proves the inclusion $\dom \ovl{\partial} \subset \dom A^{\alpha}$. Moreover, for any integer $n \geq 1$, we have  $
\bnorm{A^{\alpha}(x_n)}_{X} 
\ov{\eqref{estim-Riesz-bis}}{\lesssim} \norm{\partial(x_n)}_{Y}$. 
Since $x_n \to x$ in $\dom\ovl{\partial}$ and in $\dom A^{\alpha}$ both equipped with the graph norm, we conclude that $
\bnorm{A^{\alpha}(x)}_{X} 
\lesssim \bnorm{\ovl{\partial}(x)}_{Y}$. 
\end{proof}

The same argument as in \cite[Proposition 3.7]{Arh26b} shows the following result, which says that a suitable dual Riesz estimate implies a reverse Riesz estimate. Here, we consider an abstract pair $(\partial,\partial^\dagger)$ of gradient and divergence operators compatible with a sectorial operator $A$, as in Section \ref{sec-divergence}. Note that in this case, we have $\dom A \subset \dom \partial$.

\begin{prop}
\label{prop-duality-Riesz}
Suppose that $X$ is reflexive. Consider an abstract pair $(\partial,\partial^\dagger)$ of gradient and divergence operators compatible with a sectorial operator $A$. Let $\alpha \in (0,1)$. Let $C'$ be a subspace of $\dom (A^*)^\alpha \cap \dom(\partial^\dagger)^*$ which is a core
of $(A^*)^\alpha$. We denote by $P \co X \to X$ the bounded projection onto the subspace $\ker A$ associated to the decomposition $X \ov{\eqref{decompo-reflexive}}{=} \ker A \oplus \ovl{\Ran A}$. Then the estimate
\begin{equation}
\label{Riesz-555}
\norm{(\partial^\dagger)^*(z)}_{Y^*}
\leq K\bnorm{(A^*)^{\alpha}(z)}_{X^*}, \quad z \in C'
\end{equation}
implies the estimate 
$$
\bnorm{A^{1-\alpha}(x)}_{X}
\leq K \norm{\Id-P}_{X \to X} \norm{\partial(x)}_{Y}, \quad x \in \dom A.
$$
\end{prop}

\section{Illustrations and discussions in various contexts}
\label{sec-examples}

\subsection{Poincar\'e inequalities on Riemannian manifolds}
\label{sec-Riemann-manifolds}

Let $(M,g)$ be a smooth Riemannian manifold. If $\Rie$ is the Riemann tensor, for each point $x \in M$, we define, as in \cite[Definition p.~60]{Cha06}, the bilinear form $\Ric_x \co \T_x M \times \T_x M \to \R$ by
$$
\Ric_x(Y,Z) 
\ov{\mathrm{def}}{=} \tr\big(X \mapsto \Rie_x(Y,X)Z\big), \quad Y,Z \in \T_x M.
$$
The Ricci endomorphism $\widehat{\Ric}_x \co \T_x M \to \T_x M$ is the unique linear map defined by
\begin{equation}
\label{Ricci-endomorphism}
g_x(\widehat{\Ric}_xY,Z)
\ov{\mathrm{def}}{=} \Ric_x(Y,Z),
\quad Y,Z \in \mathrm{T}_x M.
\end{equation}
Since $\Ric_x$ is symmetric according to \cite[p.~60]{Cha06}, the operator $\widehat{\Ric}_x$ is selfadjoint with respect to $g_x$.

Let $M$ be a complete connected Riemannian manifold of dimension $d$, whose Ricci curvature is bounded from below, i.e., there exists a constant $K \in \R$ such that $\Ric_x(X,X) \geq K g_x(X,X)$ for any $X \in \T_x M$. Assume in addition that there exists a constant $a > 0$ such that
\begin{equation}
\label{spectral-gap-Kunst}
\sigma(-\Delta) 
\subset \{0\} \cup [a,\infty).
\end{equation}
Suppose that $1 < p <\infty$. Then by \cite[Corollary 1.2 p.~39]{JKW10}, we have a Riesz equivalence
\begin{equation}
\label{Riesz-Kunst-bis}
\norm{ (-\Delta)^{\frac{1}{2}}f}_{\L^p(M)} 
\approx_{p,M}  \norm{\nabla f}_{\L^p(M,\T M)}, \quad f \in \C_0^{\infty}(M).
\end{equation}
This result is a variant of a classical result \cite{Bak87} of Bakry. By \cite[Proposition 3.4]{Arh26b}, we deduce that the unbounded operator $\nabla \co \C_0^{\infty}(M) \subset \L^p(M) \to \L^p(M,\T M)$ is closable with closure denoted $\nabla_p$, that $\dom \nabla_p=\dom (-\Delta_p)^{\frac{1}{2}} $ and that
\begin{equation}
\label{Riesz-Kunst}
\norm{ (-\Delta_p)^{\frac{1}{2}}f}_{\L^p(M)} 
\approx_{p,M}  \norm{\nabla_p f}_{\L^p(M,\T M)}, \quad f \in \dom \nabla_p.
\end{equation}
First, we record the following standard fact.

\begin{prop}
\label{prop-ergodic}
Let $M$ be a complete smooth Riemannian manifold. Suppose that $1 < p < \infty$. The heat semigroup $(\e^{t \Delta_p})_{t \geq 0}$ is mean ergodic and bounded holomorphic on the Banach space $\L^p(M)$.
\end{prop}

\begin{proof}
Recall that by \cite[Exercise 7.36 p.~208]{Gri09} the heat semigroup $(\e^{t \Delta_p})_{t \geq 0}$ is contractive on the Banach space $\L^p(M)$. Since this space is reflexive, we deduce by Example \ref{bounded-ergodic} that the semigroup $(\e^{t \Delta_p})_{t \geq 0}$ is mean ergodic. Moreover, it is known (essentially proved in \cite{Gri09}) that $(\e^{t \Delta_p})_{t \geq 0}$ is a diffusion semigroup in the sense of \cite[p.~49]{JMX06}. We conclude that the semigroup is bounded holomorphic on the Banach space $\L^p(M)$ by \cite[Proposition 5.4 p.~51, Lemma 3.1 p.~26]{JMX06} or \cite[Theorem 1 p.~67]{Ste70}. 
\end{proof}



\paragraph{Locally symmetric spaces}

Let $M$ be a Riemannian manifold. A geodesic symmetry at $x \in M$ is a map $s_x \ov{\mathrm{def}}{=} \exp_x \circ (-\Id_{\T_xM}) \circ \exp_x^{-1} \co \cal{U}_x \to \cal{U}_x$, where the exponential map $\exp_x \co U_x \to \cal{U}_x$ is a diffeomorphism from a neighbourhood $U_x$ of the origin of $\mathrm{T}_xM$, symmetric with respect to 0, onto a normal neighbourhood $\cal{U}_x$ of $x$. The manifold $M$ is called locally symmetric if $s_x$ is an isometry for all $x \in M$. It is known \cite[Theorem 8.1.1 p.~232]{Wol11} that a Riemannian manifold is a locally symmetric space if and only if its curvature tensor $R$ is parallel, that is,
\begin{equation}
\label{nabla-R-=-zero}
\nabla R
=0.
\end{equation} 
The manifold is said to be a symmetric space if for any $x \in M$, the geodesic symmetry at $x$ extends to a globally defined isometry of $M$. In particular, every symmetric space is locally symmetric. Moreover, if $M$ is complete and connected, then by \cite[Corollary 12.7 p.~351]{Lee18}, $M$ is locally symmetric if and only if it is isometric to a quotient $\Gamma \backslash \tilde{M}$, where $\tilde{M}$ is a symmetric space and where $\Gamma$ is a discrete group that acts freely, properly, and isometrically on $\tilde{M}$. In this case, $\tilde{M}$ can be taken to be the Riemannian universal covering of $M$. We will say that a symmetric space is irreducible \cite[p.~242]{Wol11} if the universal cover $\tilde{M}$ of $M$ is not isometric to a non-trivial product of symmetric spaces.

It is worth noting that each connected symmetric space $M$ is isometric to a quotient $G/K$, where $G \ov{\mathrm{def}}{=} \Iso_0(M)$ is the connected component of the identity of the group of isometries of $M$ and $K \ov{\mathrm{def}}{=} \Stab_G(x)$ is the compact subgroup of $G$ which leaves a fixed point $x \in M$. Moreover, the symmetry $s_x$ gives rise to an involutive automorphism $\sigma_x \co G \to G$, $g \mapsto s_x gs_x$. If $G^\sigma = \{g \in G : \sigma(g) = g\}$ is the fixed point set of $\sigma$, then $(G^\sigma)_0 \leq K \leq G^\sigma$. We refer to \cite[Theorem 3.3 p.~208]{Hel01} for a precise statement and complements. Conversely, $(G,K)$ is called a Riemannian symmetric pair if $G$ is a connected
Lie group, $K$ is a closed subgroup of $G$ such that $\Ad_G(K)$ is a compact subgroup of $\GL(\g)$, and  there exists an 
involutive automorphism $\sigma \co G \to G$ such that $(G^\sigma)_0 \subset K \subset G^\sigma$. Then, by \cite[Proposition 3.4 p.~209]{Hel01}, $G/K$, endowed with any $G$-invariant Riemannian metric, is a Riemannian symmetric space.

We will use the following elementary lemma.

\begin{lemma}
\label{lemma-Ricci-bounded-below}
Let $M$ be a connected locally symmetric space. The Ricci curvature is bounded from below.
\end{lemma}

\begin{proof}
Note that by \eqref{nabla-R-=-zero}, we have $\nabla R=0$. Taking the trace, we obtain, since covariant differentiation commutes with contractions \cite[Proposition 4.15 p.~95]{Lee18}, that the Ricci tensor of a locally symmetric space is parallel, that is, $\nabla \Ric=0$. Let $x,y \in M$. Consider a smooth curve $\gamma \co [0,1] \to M$ from $x$ to $y$. If we consider the parallel transport map $P_{0,1}^\gamma \co \mathrm{T}_xM \to \mathrm{T}_yM$, we deduce that for any $X,Y \in \mathrm{T}_xM$, we have $\Ric_x(X,Y)=\Ric_y(P_{0,1}^\gamma X,P_{0,1}^\gamma Y)$. Using \eqref{Ricci-endomorphism}, we infer that 
$$
g_x\big(\widehat{\Ric}_xX,Y\big)
=g_y\big(\widehat{\Ric}_y P_{0,1}^\gamma X,P_{0,1}^\gamma Y\big)
=g_x\big((P_{0,1}^\gamma)^*\widehat{\Ric}_y P_{0,1}^\gamma X, Y\big).
$$ 
So by uniqueness of the Ricci endomorphism, we obtain that $\widehat{\Ric}_x=(P_{0,1}^\gamma)^*\widehat{\Ric}_y P_{0,1}^\gamma$.
Hence the eigenvalues of the Ricci endomorphism $\widehat{\Ric}_x \co \mathrm{T}_xM \to \mathrm{T}_xM$ are independent of the point $x \in M$. Let $\lambda$ be the smallest one. We get
\[
\Ric_x(X,X)
\ov{\eqref{Ricci-endomorphism}}{=} g_x(\widehat{\Ric}_xX,X)
\geq \lambda g_x(X,X),
\quad x \in M, X \in \mathrm{T}_x M.
\]
Thus $\Ric \geq \lambda g$ on $M$.
\end{proof} 

A connected Riemannian symmetric space $M$ is said to be of non-compact type if it is isometric to a homogeneous space $G/K$, where $G$ is a connected semisimple Lie group with finite center and no compact factors, and $K$ is a maximal compact subgroup of $G$. It is worth noting that by \cite[Corollary 8.3.13 p.~245]{Wol11} a Riemannian symmetric space of non-compact type is simply connected.

Let $\Gamma$ be a discrete group that acts freely, properly, and isometrically on $G/K$. Recall that, for any $x,y \in G/K$ the orbit counting function is defined by $N(R, x, y) \ov{\mathrm{def}}{=} \card \{\gamma \in \Gamma  : \dist(x, \gamma y) \leq R\}$. We can introduce the critical exponent
$$
\delta(\Gamma) 
= \limsup_{R \to \infty} \frac{\log N(R, x, y)}{R},
$$
which is a measure for the exponential growth rate of $\Gamma$ orbits in $G/K$. The definition of the critical exponent $\delta(\Gamma)$ does not depend on the choice of the points $x,y \in G/K$. 
According to \cite[Lemma 3.1 p.~76]{Web08}, we have $\delta(\Gamma)\leq 2\norm{\rho}$, where $\rho$ is half the sum of the positive roots. 



By \cite[Theorem 4.5 p.~82]{Web08} \cite{Web07} (see also \cite{Leu04}) the following holds:
\begin{enumerate}
	\item if $0 \leq \delta(\Gamma) \leq \rho_{\min}$ then $\inf \sigma(-\Delta_2) = \norm{\rho}^2$,

	\item if $\rho_{\min} \leq \delta(\Gamma)  \leq \norm{\rho}$ then $\norm{\rho}^2 - (\delta(\Gamma)  - \rho_{\min})^2 \leq \inf \sigma(-\Delta_2)  \leq \norm{\rho}^2$,

	\item if $\norm{\rho} \leq \delta(\Gamma)  \leq 2\norm{\rho}$ then 
	$$
	\max\{0,\norm{\rho}^2  - ( \delta(\Gamma) -\rho_{\min})^2\} \leq \inf \sigma(-\Delta_2) \leq \norm{\rho}^2 - (\delta(\Gamma) - \norm{\rho})^2,
	$$
\end{enumerate}
where $\inf \sigma(-\Delta_2) $ is the bottom of the spectrum of the Laplacian $-\Delta_2$ on the Hilbert space $\L^2(\Gamma \backslash G/K)$ and where we refer to \cite{Web08} for the definition of $\rho_{\min}$. It is clear that we can have $\inf \sigma(-\Delta) >0$. Now, we obtain in this case the following $\L^p$-Poincar\'e inequality.

\begin{thm}
Let $M=\Gamma \backslash G/K$ be a complete locally symmetric space such that $G$ is a connected semisimple Lie group with finite center and no compact factors and where $\Gamma$ is a discrete group that acts freely, properly, and isometrically on $G/K$. Assume that $\inf \sigma(-\Delta_2) >0$. Suppose that $1 < p < \infty$. Then
\begin{equation*}
\norm{f}_{\L^p(M)}
\lesssim_{p,M} \norm{\nabla_p f}_{\L^p(M,\mathrm{T}M)}, \quad f \in \dom \nabla_p.
\end{equation*}
\end{thm}

\begin{proof}
By Proposition \ref{prop-ergodic}, the semigroup $(\e^{t \Delta_p})_{t \geq 0}$ is mean ergodic and bounded holomorphic. We have $0 \in \rho(-\Delta_2)$. According to Proposition \ref{prop-reduced-exp-stability-closed-range}, $(\e^{t \Delta_2})_{t \geq 0}$ is uniformly exponentially stable on $\L^2(M)$. By interpolation with Proposition \ref{prop-stability-interpolation}, we deduce that $(\e^{t \Delta_p})_{t \geq 0}$ is uniformly exponentially stable on $\L^p(M)$. By \eqref{uniformly-exponentially-stable}, we infer that there exist constants $C,\omega_p > 0$ such that
\[
\norm{\e^{t\Delta_p}}_{\L^p(M) \to \L^p(M)}
\leq C\e^{-\omega_p t},
\quad t\geq 0.
\]
Taking the limit when $t \to \infty$ and using Proposition \ref{prop-strong-limit-projection}, we obtain  that the mean projection $P \co \L^p(M) \to \L^p(M)$ is equal to 0. The Ricci curvature is bounded from below by Lemma \ref{lemma-Ricci-bounded-below}. From \eqref{Riesz-Kunst}, we see that the reverse Riesz estimate
\[
\bnorm{(-\Delta_p)^{\frac{1}{2}}f}_{\L^p(M)}
\lesssim_{p,M}
\norm{\nabla_p f}_{\L^p(M,\mathrm{T}M)}, \quad f \in \dom \nabla_p
\]
holds (see also \cite{MaM09} for a previous partial result). Now, we conclude using Theorem \ref{thm-direct-Poincare-via-negative-powers} with $\alpha=\frac{1}{2}$, $\partial=\nabla_p$ and $P=0$.
\end{proof}

\begin{remark} \normalfont
Let $M$ be a symmetric space of non-compact type. According to \cite[Proposition 2.2 p.~780]{Tay89}, for any $p \in [1,\infty)$, the spectrum of the closure $-\Delta_{p}$ of the Laplacian operator on $\L^p(M)$ is given by 
$$
\sigma(-\Delta_{p}) 
= \Big\{ \norm{\rho}^2 + z^2 : z \in \mathbb{C}, |\mathrm{Im} \, z| \leq \big|\tfrac{2}{p}-1\big|\norm{\rho} \Big\}.
$$
Furthermore, if $p > 2$ any point in the interior of this parabolic region is an eigenvalue for $-\Delta_{p}$ and eigenfunctions corresponding to these eigenvalues are given by spherical functions.
\end{remark}

\begin{remark} \normalfont
A proof of \eqref{Riesz-Kunst} is sketched in \cite[Theorem 6.1 p.~75]{Str83} for rank-one symmetric space $G/K$ of non-compact type, where $G$ is a non-compact connected semi-simple Lie group of real rank one, and $K$ a maximal compact subgroup. 
\end{remark}


\paragraph{Cartan-Hadamard manifolds}
Let $M$ be a Cartan--Hadamard manifold, i.e., a complete simply connected Riemannian manifold with non-positive sectional curvature. Assume in addition that the sectional curvature is bounded above by a \textit{strictly negative} constant $-k^2$ and that the dimension of $M$ is $d \geq 2$. By \cite[p.~360]{McK70} (see also \cite{Pin81}), the spectrum of its Laplacian $-\Delta_2$ on $\L^2(M)$ is included in the interval $\big[\frac{(d-1)^2k^2}{4},\infty\big)$. This implies that $\ker \Delta_2=\{0\}$, $0 \in \rho(-\Delta_2)$ and it admits by Proposition \ref{prop-gap-and-s} an $\L^2$-Poincar\'e inequality, namely 
\begin{equation}
\label{Poincare-Cartan}
\norm{f}_{\L^2(M)}
\lesssim_M \norm{\nabla_2 f}_{\L^2(M,\mathrm{T}M)}, \quad f \in \C^\infty_0(M).
\end{equation}
Indeed, Strichartz essentially proved in \cite[Theorem 5.4 p.~68]{Str83} the following $\L^p$-Poincar\'e inequality, which generalizes \eqref{Poincare-Cartan}.

\begin{thm}[Strichartz]
\label{thm-Strichartz}
Suppose that $1 < p < \infty$. Let $M$ be a Cartan--Hadamard manifold with dimension $d \geq 2$, whose sectional curvature is bounded above by a strictly negative constant $-k^2$. Then
\begin{equation*}
\norm{f}_{\L^p(M)}
\lesssim \frac{p}{(d-1)k} \norm{\nabla_p f}_{\L^p(M,\mathrm{T}M)}, \quad f \in \C^\infty_0(M).
\end{equation*}
\end{thm}
Suppose that $1 < p < \infty$. Let $M$ be a Cartan--Hadamard manifold whose sectional curvature is bounded above by a strictly negative constant \textit{and whose curvature tensor, together with its first two covariant derivatives, is bounded.} According to a result of Lohou\'e \cite[Corollaire 14 p.~191]{Loh85}, we have the Riesz equivalence
\begin{equation}
\label{Lohoue}
\norm{\nabla f}_{\L^p(M,\mathrm{T}M)} 
\approx_{p,M} \bnorm{(-\Delta)^{\frac{1}{2}} f}_{\L^p(M)} ,\quad f \in \C^\infty_0(M).
\end{equation}
For instance, the assumption on the curvature tensor is satisfied by the real hyperbolic space
$\mathrm{Hyp}^d$ of dimension $d \geq 2$. Indeed, if the sectional curvature is normalized to be $-1$, then by \cite[Proposition II.3.1 p.~65]{Cha06} we have
\[
\Rie(X,Y)Z
=-\bigl(g(X,Z)Y-g(Z,Y)X\bigr).
\]
Since the Levi-Civita connection is compatible with $g$ according \cite[Theorem 5.10 p.~122]{Lee18}, we have $\nabla g=0$ by \cite[Proposition 5.5 p.~118]{Lee18}. A simple computation \cite{Lee09} left to the reader shows that 
$\nabla \Rie=0$. This implies that $\nabla^2 \Rie=0$. Thus the curvature tensor and its first two covariant derivatives are bounded.

So, using Theorem \ref{thm-direct-Poincare-via-negative-powers}, we can give a proof of a qualitative version of Theorem \ref{thm-Strichartz}:
\begin{equation}
\label{My-Stri}
\norm{f}_{\L^p(M)}
\lesssim_{p,M} \norm{\nabla_p f}_{\L^p(M,\mathrm{T}M)}, \quad f \in \dom \nabla_p,
\end{equation}
under the previous additional bounded-curvature assumption.

\begin{proof}[of \eqref{My-Stri}]
By Proposition \ref{prop-ergodic}, the semigroup $(\e^{t \Delta_p})_{t \geq 0}$ is mean ergodic and bounded holomorphic. We have seen that $0 \in \rho(-\Delta_2)$. By Proposition \ref{prop-reduced-exp-stability-closed-range}, $(\e^{t \Delta_2})_{t \geq 0}$ is uniformly exponentially stable on $\L^2(M)$. By interpolation with Proposition \ref{prop-stability-interpolation}, we conclude that $(\e^{t \Delta_p})_{t \geq 0}$ is uniformly exponentially stable on $\L^p(M)$. So by \eqref{uniformly-exponentially-stable}, there exist constants $C,\omega_p > 0$ such that
\[
\norm{\e^{t\Delta_p}}_{\L^p(M) \to \L^p(M)}
\leq C\e^{-\omega_p t},
\quad t \geq 0.
\]
Taking the limit when $t \to \infty$, we deduce from Proposition \ref{prop-strong-limit-projection} that the mean projection $P \co \L^p(M) \to \L^p(M)$ is equal to 0. Using \cite[Proposition 3.4]{Arh26b}, we can extend the equivalence \eqref{Lohoue} to elements of the space $\dom \nabla_p$. So, we see that the reverse Riesz estimate
\[
\bnorm{(-\Delta_p)^{\frac{1}{2}}f}_{\L^p(M)}
\lesssim_{p,M}
\norm{\nabla_p f}_{\L^p(M,\mathrm{T}M)}, \quad f \in \dom \nabla_p
\]
holds. Now, it suffices to use Theorem \ref{thm-direct-Poincare-via-negative-powers} with $\alpha=\frac{1}{2}$, $\partial=\nabla_p$ and $P=0$.
\end{proof}

In the particular case $M=\mathbb{H}^d$, the constant $\frac{p}{d-1}$ is sharp by \cite[Theorem 1.1]{NgN19} and is never achieved in $\W^{1,p}(\mathbb{H}^d)$. It is worth noting that the spectrum of the Laplacian $-\Delta_p$ on the Banach space $\L^p(\mathbb{H}^d)$ is explicitly known, see \cite[Theorem 5.7.1 p.~178]{Dav89}. It is natural to believe that Lohou\'e's curvature assumptions should not be needed. This leads us to conclude with the following conjecture.

\begin{conj}
Suppose that $1 < p < \infty$. Let $M$ be a Cartan-Hadamard manifold whose sectional curvature is bounded above by a strictly negative constant. Then
\begin{equation}
\label{Lohoue-improved}
\norm{\nabla f}_{\L^p(M,\mathrm{T}M)} 
\lesssim_{p,M} \bnorm{(-\Delta)^{\frac{1}{2}} f}_{\L^p(M)} ,\quad f \in \C^\infty_0(M).
\end{equation} 
\end{conj}
By Proposition \ref{prop-duality-Riesz}, this conjecture implies a reverse Riesz estimate.


\paragraph{Compact manifolds}
We continue with the standard case of compact Riemannian manifolds, which satisfies \eqref{spectral-gap-Kunst} by \cite[Theorem 4.3.1 p.~77]{Lab15} (see also \cite[Exercise 10.10 p.~282]{Gri09}). Riesz estimates are available with \cite[Proposition 4.1]{Arh26a}. So we recover the following result \cite[Theorem 2.10 p.~40]{Heb96}, \cite[Lemma 3.8 p.~24]{Heb00}. The cited formulation imposes $1 < p < n$, which is not needed here. See also \cite[Corollary A.1.2 p.~660]{Jos17} for the case $p=2$.

\begin{thm}
Let $M$ be a smooth connected compact Riemannian manifold. Suppose that $1 < p < \infty$. Then
\begin{equation*}
\norm{f- \frac{1}{\vol(M)}\int_M f}_{\L^p(M)}
\lesssim_{p,M} \norm{\nabla_p f}_{\L^p(M,\mathrm{T}M)}, \quad f \in \dom \nabla_p.
\end{equation*}
\end{thm}

\paragraph{Manifolds with cusps of rank one}
Manifolds with cusps of rank one are non-compact manifolds with finite volume satisfying also the condition $\sigma(-\Delta) = \{0,\lambda_1,\ldots,\lambda_r\} \cup [b,\infty)$ for some $b > 0$, see \cite{Mul87} and \cite{Web07}. So we obtain an $\L^p$-Poincar\'e inequality in this case.

\vspace{0.2cm}

Finally, we refer to \cite[p.~274]{Li08} and \cite{Li09} for other results on $\L^p$-Poincar\'e inequalities in the context of complete Riemannian manifolds.

\subsection{Poincar\'e inequalities on Riemannian spin manifolds}


We briefly recall the spin-geometric notation used below. 
Let $d \geq 2$. Consider the canonical two-fold covering $\lambda \co \Spin(d) \to \SO(d)$ of the special orthogonal group $\SO(d)$, which is the universal cover of $\SO(d)$ if $d \geq 3$. A spin structure on an oriented Riemannian manifold $M$ of dimension $d$ is a $\lambda$-reduction 
$(P_{\Spin(d)},\Theta)$ of the principal bundle $P_{\SO(d)}(M) \to M$ of oriented orthonormal frames of the tangent bundle $\mathrm{T}M$. This means that $P_{\Spin(d)}(M) \to M$ is a principal $\Spin(d)$-bundle and that $\Theta \co P_{\Spin(d)}(M) \to P_{\SO(d)}(M)$ is a bundle map over the identity of $M$ that is compatible  with the corresponding group actions, i.e., $\Theta(u \cdot s)=\Theta(u) \cdot \lambda(s)$ for any $s \in \Spin(d)$ and any $u \in P_{\Spin(d)}(M)$. Recall that a Riemannian spin manifold \cite[Definition 1.1.2 p.~2]{Gin09} is an orientable smooth Riemannian manifold admitting a spin structure. It is known that an oriented Riemannian manifold admits a spin structure if and only if $w_2(\T M)=0$ in the \v{C}ech cohomology group $\mathrm{H}^2(M,\mathbb{Z}_2)$ with coefficients in $\mathbb{Z}_2$. Moreover, when spin structures exist, their equivalence classes are parametrized by $\mathrm{H}^1(M,\mathbb{Z}_2)$, see \cite[Proposition 1.1.3, p.~2]{Gin09} and \cite[Proposition 3.34, p.~115]{BGV04}. In the sequel, a Riemannian spin manifold will always be equipped with a fixed orientation and a fixed spin structure.

%


Let $\Cl_d(\mathbb{C})
\ov{\mathrm{def}}{=}
\Cl_{d,0}(\R) \ot_{\R} \mathbb{C}$ be the complex Clifford algebra associated with the Euclidean space $\R^d$, with the convention $vw+wv=-2\la v,w \ra_{\R^d}\Id$, where $v,w \in \R^d$. Fix an irreducible complex representation $\rho_d \co \Cl_d(\mathbb{C}) \to \End_{\mathbb{C}}(\Delta_d)$, 
and denote by $
\kappa_d
\ov{\mathrm{def}}{=}
\rho_d|_{\Spin(d)}
\co
\Spin(d) \to \mathrm{GL}(\Delta_d)$ the associated complex spin representation, defined in \cite[Definition 5.14 p.~36]{LaM89}
, which is a faithful representation of the group $\Spin(d)$ by \cite[Proposition p.~20]{Fri00}. 

According to \cite[Proposition 3.26 p.~112]{BGV04} and \cite[Proposition p.~24]{Fri00}, there exists a Hermitian inner product $\la \cdot,\cdot\ra_{\Delta_d}$ on $\Delta_d$, unique up to multiplication by a positive constant, for which Clifford multiplication by real vectors is skew-Hermitian:
\[
\la \rho_d(v)\xi,\eta\ra_{\Delta_d}
=
-\la \xi,\rho_d(v)\eta\ra_{\Delta_d},
\quad
v\in\R^d,\quad \xi,\eta\in\Delta_d.
\]
We fix such an inner product. In particular, the representation $\kappa_d$ is unitary with respect to this inner product by \cite[Proposition 3.11 p.~107]{BGV04}. 

The spinor bundle of $M$ associated with the chosen spin structure is the complex vector bundle $
\Sigma M
\ov{\mathrm{def}}{=}
P_{\Spin(d)}(M)\times_{\kappa_d}\Delta_d$, see \cite[Definition 1.2.2, pp.~5--6]{Gin09}. The space of smooth sections of $\Sigma M$ is denoted by $\C^\infty(M,\Sigma M)$ and its elements are called spinor fields. The Hermitian inner product on $\Delta_d$ induces a canonical Hermitian metric on the vector bundle $\Sigma M$.

By \cite[Definition 2.4.9 p.~100]{Bar11}, 
the representation $\rho_d \co \Cl_d(\mathbb{C}) \to \End_{\mathbb{C}}(\Delta_d)$ induces a bundle map $
c \co \T M \to \mathrm{End}_{\Cbb}(\Sigma M)$. 
We write $X\cdot\varphi$ for $c(X)\varphi$ for any $X \in \T_xM$ and $\varphi \in \Sigma_xM$. By \cite[pp.~100-101]{Bar11} we have the Clifford relations $
X \cdot(Y\cdot\varphi) + Y\cdot(X\cdot\varphi)
=
-2g(X,Y)\varphi$ and Clifford multiplication by tangent vectors is skew-Hermitian, i.e., $
\la X \cdot\varphi,\psi \ra
=-\la\varphi,X \cdot \psi \ra$, where $X,Y \in \T_xM$ and $\varphi,\psi \in \Sigma_xM$.

By \cite[pp.~100-101]{Bar11},
the Levi-Civita connection on $P_{\SO(d)}(M)$ lifts through the spin structure and induces a canonical metric connection $
\nabla^\Sigma
\co
\C^\infty(M,\Sigma M)
\to
\C^\infty(M,\T^*M\ot\Sigma M)$, 
called the spin connection. By \cite[Lemma 2.4.12 p.~103]{Bar11} (see also \cite[Definition 1.2.2 (iii) p.~6]{Gin09}), it is compatible with Clifford multiplication in the sense that
\[
\nabla_X^\Sigma(Y\cdot\sigma)
=
(\nabla_XY)\cdot\sigma
+
Y\cdot\nabla_X^\Sigma\sigma,
\quad
X,Y\in\C^\infty(M,\T M),
\sigma\in\C^\infty(M,\Sigma M),
\]
where $\nabla$ denotes the Levi-Civita connection on $\T M$. The Dirac operator is then defined \cite[Definition 5.5.12 p.~406]{RuS17} as the composition of this connection with the Clifford multiplication
\[
D
\ov{\mathrm{def}}{=}
c \circ \nabla^{\Sigma}
\co \C^\infty_c(M,\Sigma M)
\xrightarrow{\nabla^{\Sigma}}
\C^\infty_c(M,\T^*M \ot \Sigma M)
\cong \C^\infty_c(M,\T M \ot \Sigma M)
\xrightarrow{c}
\C^\infty_c(M,\Sigma M),
\]
where we use the Riemannian metric to identify the cotangent bundle $\T^*M$ with the tangent bundle $\T M$. By \cite[Proposition 1.3.4 p.~12]{Gin09}, the Dirac operator is an elliptic and formally selfadjoint differential operator of first order.
From now on, we assume that $M$ is complete. Then the operator $
D\co
\C^\infty_c(M,\Sigma M)
\subset
\L^2(M,\Sigma M)
\to
\L^2(M,\Sigma M)$ is essentially selfadjoint by \cite[Theorem 5.1, p.~622]{Wol73} (see also \cite[Proposition 1.3.5, p.~13]{Gin09}).

Suppose that $1 < p < \infty$. It is essentially proved in \cite{AmG16} and \cite{ChG23} that the Dirac operator $D \co \C^\infty_c(M,\Sigma M) \subset \L^p(M,\Sigma M) \to \L^p(M,\Sigma M)$ is closable with closure denoted $D_p$ and that the domain of the closed operator $D_p$ is the completion $\W^{1,p}_D(M,\Sigma M)$ of the space $\C_c^\infty(M,\Sigma M)$ equipped with the norm
\begin{equation}
\norm{\sigma}_{\W^{1,p}_D(M,\Sigma M)}
\ov{\mathrm{def}}{=}
\norm{\sigma}_{\L^p(M,\Sigma M)} +\norm{D\sigma}_{\L^p(M,\Sigma M)}.
\end{equation}
Moreover, by \cite[Lemma B.2]{AmG16}, we have $D_p^*=D_{p^*}$ where $p^* \ov{\mathrm{def}}{=} \frac{p}{p-1}$
is the H\"older conjugate exponent. 

Following \cite[p.~60]{Cha06}, the scalar curvature $\Scal \co M \to \R$ of $M$ is defined by $\Scal(x) \ov{\mathrm{def}}{=} \tr(\widehat{\Ric}_x)$, where $x \in M$. Recall that the Ricci endomorphism $\widehat{\Ric}_x \co \T_x M \to \T_x M$ is defined in \eqref{Ricci-endomorphism}. It is worth noting that, by \cite[Theorem 1.3.8 p.~16]{Gin09}, the Dirac operator $D$ satisfies the Schr\"odinger-Lichnerowicz formula 
\begin{equation}
\label{Schro-Lich}
D^2
=(\nabla^\Sigma)^*\nabla^\Sigma+\frac{\Scal}{4}.
\end{equation}
By \cite[Theorem 7.3.1 p.~106]{Gin09} (see also \cite[Theorem 3.1 p.~70]{Bar09} for a more general statement), the bottom of the spectrum $\sigma(D_2^2)$ of the operator $D_2^2$ satisfies
\begin{equation}
\label{min-sigma}
\inf \sigma(D_2^2) 
\geq \frac{d}{4(d-1)} \inf_M \Scal.
\end{equation} 
By \cite[Theorem 3.5]{BBC20}, if the scalar curvature $\Scal$ is positive and if $1 < p < \infty$, we have the Riesz estimate
\begin{equation}
\label{Banuelos}
\norm{D \sigma}_{\L^p(M,\Sigma M)} 
\leq 6\left(\max\big\{p,\tfrac{p}{p-1}\}-1\right) \bnorm{(D^2)^{\frac{1}{2}} \sigma}_{\L^p(M,\Sigma M)}, \quad \sigma \in \C^\infty_c(M,\Sigma M).
\end{equation}
Although \cite[Theorem 3.5]{BBC20} is stated for even-dimensional spin manifolds, the same proof applies in odd dimension. 
Suppose that $1 <p < \infty$. If $M$ is complete with Ricci curvature and scalar curvature bounded from below, the semigroup $\e^{-t D_2^2}$ induces a strongly continuous semigroup of operators on the Banach space $\L^p(M,\Sigma M)$ with infinitesimal generator $-D_p^2$, which is the square of the closure of the operator $D \co \C^\infty_c(M,\Sigma M) \subset \L^p(M,\Sigma M) \to \L^p(M,\Sigma M)$ according to \cite[Corollary A.2]{ChG23}.

Now, we obtain a quantitative $\L^p$-Poincar\'e inequality.

\begin{prop}
\label{prop-spin-Poincare}
Suppose that $1<p<\infty$. Let $M$ be a complete Riemannian spin manifold of dimension $d\geq2$ whose Ricci curvature is bounded from below. Assume that there exists a constant $\kappa>0$ such that $
\Scal\geq\kappa$. Then
\begin{equation}
\label{eq-spin-Poincare}
\norm{\sigma}_{\L^p(M,\Sigma M)}
\leq
\frac{12}{\sqrt{\kappa}}\big(\max\big\{p,\tfrac{p}{p-1}\big\}-1\big)
\norm{D_p\sigma}_{\L^p(M,\Sigma M)},
\quad \sigma \in \W_D^{1,p}(M,\Sigma M).
\end{equation}
\end{prop}

\begin{proof}
In the terminology of \cite[Example 2.1 (i)]{ChG23}, the Clifford contraction is given by $\frac{\Scal}{4}$. Let $\Delta_{\HdR,0,2} \co \dom \Delta_{\HdR,0,2} \subset \L^2(M) \to \L^2(M)$ denote the $\L^2$-realization of the Hodge-de Rham Laplacian $\Delta_{\HdR,0}$ \textit{on functions}. Set $
H_2
\ov{\mathrm{def}}{=}
D_2^2-\frac{\kappa}{4}$. By the Schr\"odinger--Lichnerowicz formula, the operator $H_2$ is the
selfadjoint realization of the covariant Schr\"odinger operator
\begin{equation}
\label{eq-shifted-Dirac-square}
H_2
=D_2^2-\frac{\kappa}{4}
\overset{\eqref{Schro-Lich}}{=}
(\nabla^\Sigma)^*\nabla^\Sigma
+
\frac{\Scal-\kappa}{4}.
\end{equation}
Since $\Scal \geq \kappa$, its potential satisfies $
\frac{\Scal-\kappa}{4}
\geq
0$. Therefore, the Kato--Simon inequality for covariant Schr\"odinger
operators \cite[Theorem VII.8 p.~111]{Gun17} yields
\begin{equation}
\label{eq-shifted-Dirac-semigroup-domination}
\left|\e^{-tH_2}\sigma\right|
\leq
\e^{-t\Delta_{\HdR,0,2}}|\sigma|,
\quad
t \geq 0, \sigma \in \L^2(M,\Sigma M).
\end{equation}
where we use the notation $|\sigma| \co M \to [0,\infty)$  defined by $
|\sigma|(x)
\ov{\mathrm{def}}{=}
\sqrt{\la \sigma(x),\sigma(x) \ra_{\Sigma_x M}}$ for almost all $x \in M$. Since $
D_2^2
=
H_2+\frac{\kappa}{4}$, we have with \cite[(2.2) p.~60]{EnN00}
\begin{equation}
\label{eq-Dirac-semigroup-domination-positive-scal}
\left|\e^{-t D_2^2}\sigma\right|
=\e^{-\frac{\kappa}{4}t} \left|\e^{-t H_2}\sigma\right|
\overset{\eqref{eq-shifted-Dirac-semigroup-domination}}{\leq}
\e^{-\frac{\kappa}{4}t}
\e^{-t\Delta_{\HdR,0,2}}|\sigma|,
\quad
t \geq 0, \sigma \in \L^2(M,\Sigma M).
\end{equation}
Recall that the scalar heat semigroup $(\e^{-t\Delta_{\HdR,0,2}})_{t \geq 0}$ is contractive on the Banach space $\L^p(M)$ for any $1 \leq p \leq \infty$ and that $\norm{\sigma}_{\L^p(M,\Sigma M)}
= \bnorm{|\sigma|}_{\L^p(M)}$ if $\sigma \in \L^2(M,\Sigma M) \cap \L^p(M,\Sigma M)$. For any $\sigma \in \L^2(M,\Sigma M) \cap \L^p(M,\Sigma M)$, this implies that
\[
\bnorm{\e^{-tD_2^2}\sigma}_{\L^p(M,\Sigma M)}
\ov{\eqref{eq-Dirac-semigroup-domination-positive-scal}}{\leq}
\e^{-\frac{\kappa}{4}t}
\norm{\sigma}_{\L^p(M,\Sigma M)}.
\]
Since the subspace $\L^2(M,\Sigma M) \cap \L^p(M,\Sigma M)$ is dense in the Banach space $\L^p(M,\Sigma M)$ for $1 \leq p < \infty$, the operator $\e^{-tD_2^2}$ extends uniquely to a bounded operator $T_t^{(p)} \co \L^p(M,\Sigma M) \to \L^p(M,\Sigma M)$ satisfying
\[
\bnorm{T_t^{(p)}}_{\L^p(M,\Sigma M) \to \L^p(M,\Sigma M)}
\leq \e^{-\frac{\kappa}{4}t}.
\]
These extensions are consistent and satisfy the semigroup property. Moreover, arguing by interpolation from the strong continuity on $\L^2(M,\Sigma M)$, first on the dense subspace $
\L^1(M,\Sigma M)\cap
\L^2(M,\Sigma M)\cap
\L^\infty(M,\Sigma M)$, and then by density, we see that the semigroup $(T_t^{(p)})_{t \geq 0}$ is strongly continuous on the Banach space $\L^p(M,\Sigma M)$ if $1 < p < \infty$. By \cite[Corollary A.2]{ChG23}, its generator is $-D_p^2$. Consequently, we have
\begin{equation}
\label{eq-Dirac-semigroup-exp-decay}
\bnorm{\e^{-tD_p^2}}_{\L^p(M,\Sigma M)\to\L^p(M,\Sigma M)}
\leq
\e^{-\frac{\kappa}{4}t},
\quad
t \geq 0.
\end{equation}
In particular, the semigroup $(\e^{-t D_p^2})_{t \geq 0}$ is uniformly exponentially stable. Hence $0 \in \rho(D_p^2)$. Moreover, by \cite[Corollary 3.3.6]{Haa06}, we have
\begin{equation}
\label{eq-negative-square-root-Dirac}
(D_p^2)^{-\frac12}f
=
\frac1{\sqrt{\pi}}
\int_0^\infty
t^{-\frac12}\e^{-tD_p^2}f
\d t,
\quad
f \in \L^p(M,\Sigma M),
\end{equation}
where the integral converges in the strong sense. Using \eqref{eq-Dirac-semigroup-exp-decay}, we obtain
\begin{align}
\MoveEqLeft
\label{inter-45-inter-56}
\bnorm{(D_p^2)^{-\frac12}}_{\L^p(M,\Sigma M)\to\L^p(M,\Sigma M)}
\ov{\eqref{eq-negative-square-root-Dirac}}{\leq} \frac1{\sqrt{\pi}}
\int_0^\infty
t^{-\frac12}
\e^{-\frac{\kappa}{4}t}
\d t
=
\frac1{\sqrt{\pi}}
\Gamma\left(\frac12\right)
\left(\frac{\kappa}{4}\right)^{-\frac12}
=
\frac2{\sqrt{\kappa}}.
\end{align}
Now, we use the Riesz estimate. Applying \eqref{Banuelos} with $p^*$ in place of $p$, Proposition~\ref{Prop-derivation-closable-sgrp-bis} and Proposition~\ref{prop-duality-Riesz}, together with the identity $
D_p^*=D_{p^*}$, we obtain that
\begin{equation}
\label{eq-reverse-Riesz-Dirac-explicit}
\bnorm{(D_p^2)^{\frac12}\sigma}_{\L^p(M,\Sigma M)}
\leq 6\big(\max\big\{p,\tfrac{p}{p-1}\big\}-1\big) \norm{D_p\sigma}_{\L^p(M,\Sigma M)},
\quad \sigma \in \dom D_p^2.
\end{equation}
Since the subspace $\C_c^\infty(M,\Sigma M)$ is a core of the operator $D_p$ and is contained in $\dom D_p^2$, the estimate extends to any $\sigma\in\dom D_p$. Indeed, if $(\sigma_n)$ is a sequence of elements in $\C_c^\infty(M,\Sigma M)$ which converges to $\sigma$ in the graph norm of $D_p$, then \eqref{eq-reverse-Riesz-Dirac-explicit} applied to $\sigma_n-\sigma_m$ shows that $((D_p^2)^{\frac12}\sigma_n)_{n \geq 1}$ is a Cauchy sequence in the Banach space $\L^p(M,\Sigma M)$. Since the unbounded operator $(D_p^2)^{\frac12}$ is closed, we obtain $
\sigma\in\dom (D_p^2)^{\frac12}$ and
\begin{equation}
\label{eq-reverse-Riesz-Dirac-full-domain}
\bnorm{(D_p^2)^{\frac12}\sigma}_{\L^p(M,\Sigma M)}
\leq
6 \big(\max\big\{p,\tfrac{p}{p-1}\big\}-1\big)
\norm{D_p\sigma}_{\L^p(M,\Sigma M)},
\quad
\sigma\in\dom D_p.
\end{equation}
Since $0 \in \rho(D_p^2)$, the kernel of $D_p^2$ is trivial and its ergodic projection is equal to zero. Applying Theorem~\ref{thm-direct-Poincare-via-negative-powers} with $A=D_p^2$, $\alpha=\frac12$, $\partial=D_p$, $P=0$, we conclude that for any $\sigma \in \W_D^{1,p}(M,\Sigma M)$ we have
\begin{align*}
\norm{\sigma}_{\L^p(M,\Sigma M)}
&\leq
\bnorm{(D_p^2)^{-\frac12}}_{\L^p(M,\Sigma M)\to\L^p(M,\Sigma M)}
\bnorm{(D_p^2)^{\frac12}\sigma}_{\L^p(M,\Sigma M)} \\
&\ov{\eqref{eq-reverse-Riesz-Dirac-full-domain} \eqref{inter-45-inter-56}}{\leq}
\frac{12}{\sqrt{\kappa}}\big(\max\big\{p,\tfrac{p}{p-1}\big\}-1\big)
\norm{D_p\sigma}_{\L^p(M,\Sigma M)}.
\end{align*}
\end{proof}

\subsection{Poincar\'e inequalities with Markov semigroups on von Neumann algebras}
\label{sec-Markov}

We first record a quantitative consequence of an $\L^2$-spectral gap.

\begin{prop}
\label{prop-Markov-quantitative-gap}
Let $\cal{M}$ be a von Neumann algebra equipped with a normal faithful tracial state and let $(T_t)_{t \geq 0}$ be a Markov semigroup acting on $\cal{M}$. Suppose that $1 < p < \infty$. Let $-A_p$ denote its infinitesimal generator on $\L^p(\cal{M})$, let $\E_p \co \L^p(\cal{M})\to\L^p(\cal{M})$ be the mean ergodic projection and set $
Q_p
\ov{\mathrm{def}}{=}
\Id-\E_p$. Suppose that $
\gamma
\ov{\mathrm{def}}{=}
\gap(A_2)
>0$. We let $
\theta_p
\ov{\mathrm{def}}{=}
2\min\big\{\frac1p,\frac1{p^*}\big\}$. Then
\begin{equation}
\label{quantitative-decay-Markov}
\bnorm{T_tQ_p}_{\L^p(\cal{M})\to\L^p(\cal{M})}
\leq
2^{1-\theta_p}\e^{-\theta_p\gamma t},
\quad
t \geq 0.
\end{equation}
In particular, we have $0 \in \rho(A_{p,0})$. Moreover, for any $\alpha > 0$, we have
\begin{equation}
\label{quantitative-negative-power-Markov}
\bnorm{A_{p,0}^{-\alpha}}_{\L^p_0(\cal{M}) \to \L^p_0(\cal{M})}
\leq
\frac{2^{1-\theta_p}}{(\theta_p\gamma)^\alpha}.
\end{equation}
\end{prop}

\begin{proof}
If $\gamma=\infty$, then $A_2=0$, hence $T_t=\Id$ and $\L^p_0(\cal M)=\{0\}$ for every $1<p<\infty$. The assertions are then trivial. We may therefore assume that $0 < \gamma < \infty$. By the discussion in Section~\ref{sec-gaps}, we have $\L_0^p(\cal{M})
=
\ker\E_p
=
\ovl{\Ran A_p}$. Since $A_2$ is a positive selfadjoint operator and $\L^2_0(\cal{M})
=
(\ker A_2)^\perp$, the spectral theorem gives
\begin{equation}
\label{quantitative-decay-L2-Markov}
\bnorm{T_tQ_2}_{\L^2(\cal{M})\to\L^2(\cal{M})}
\leq
\e^{-\gamma t},
\quad
t\geq 0.
\end{equation}
On the other hand, the operators $T_t$ and $\E_r$ are contractions on $\L^r(\cal{M})$ if $1 \leq r\leq\infty$. Hence
\[
\bnorm{T_tQ_r}_{\L^r(\cal{M})\to\L^r(\cal{M})}
\leq
2,
\quad
r\in\{1,\infty\},
\quad
t\geq 0.
\]
Suppose first that $2 \leq p < \infty$. Interpolating the consistent operators $T_tQ_2$ and $T_tQ_\infty$ between $\L^2(\cal{M})$ and $\L^\infty(\cal{M})$, and using the equalities $
\frac1p
=
\frac{\theta_p}{2}$, $\theta_p=\frac2p$, we obtain
\[
\bnorm{T_tQ_p}_{\L^p(\cal{M})\to\L^p(\cal{M})}
\leq
2^{1-\theta_p}\e^{-\theta_p\gamma t}.
\]
If $1 < p \leq 2$, we interpolate instead between $\L^1(\cal{M})$ and $\L^2(\cal{M})$. Since $
\theta_p
=
\frac{2}{p^*}$, the same estimate follows. This proves \eqref{quantitative-decay-Markov}. If $f \in \L^p_0(\cal{M})$, then $Q_pf=f$. Consequently, we have
\[
\bnorm{T_t|_{\L^p_0(\cal{M})}}_{\L^p_0(\cal{M})\to\L^p_0(\cal{M})}
\leq
2^{1-\theta_p}\e^{-\theta_p\gamma t},
\quad
t\geq0.
\]
Thus the restricted semigroup on $\L^p_0(\cal{M})$ is uniformly exponentially stable and therefore $0 \in \rho(A_{p,0})$ by Proposition \ref{prop-reduced-exp-stability-closed-range}. Finally, by \cite[Corollary 3.3.6 p.~76]{Haa06}, for any $\alpha>0$ and any $f \in \L^p_0(\cal{M})$, we have
\begin{equation}
\label{intermi-HJK}
A_{p,0}^{-\alpha}f
=
\frac1{\Gamma(\alpha)}
\int_0^\infty
t^{\alpha-1}T_tf
\d t.
\end{equation}
Using \eqref{quantitative-decay-Markov}, we obtain
\begin{align*}
\MoveEqLeft
\bnorm{A_{p,0}^{-\alpha}f}_{\L^p(\cal{M})}
\ov{\eqref{intermi-HJK}}{\leq} \frac1{\Gamma(\alpha)}
\int_0^\infty t^{\alpha-1}\norm{T_tf}_{\L^p(\cal{M})} \d t\\
&\ov{\eqref{quantitative-decay-Markov}}{\leq} \frac{2^{1-\theta_p}}{\Gamma(\alpha)} \int_0^\infty t^{\alpha-1}\e^{-\theta_p\gamma t} \d t
\norm{f}_{\L^p(\cal{M})}
= \frac{2^{1-\theta_p}}{(\theta_p\gamma)^\alpha}
\norm{f}_{\L^p(\cal{M})}.
\end{align*}
This proves \eqref{quantitative-negative-power-Markov}.
\end{proof}

Here, we use the notations of Section \ref{sec-gaps}. We obtain a Poincar\'e inequality under suitable assumptions.

\begin{thm}
\label{Th-Poincare-von-Neumann}
Suppose that $1 < p < \infty$. Let $\cal{M}$ be a von Neumann algebra equipped with a normalized normal finite faithful trace. Consider a Markov semigroup $(T_t)_{t \geq 0}$ of operators acting on $\cal{M}$ with infinitesimal generator $-A_p$ on $\L^p(\cal{M})$. Suppose that $
\gamma
\ov{\mathrm{def}}{=}
\gap(A_2)
>0$. Let $\alpha \in (0,1)$. Suppose that $\partial \co \dom \partial \to \L^p(\tilde{\cal{M}})$ is an unbounded operator, where $\tilde{\cal{M}}$ is a von Neumann algebra equipped with a normal finite faithful trace. Assume that $\dom \partial \subset \dom A_p^\alpha$ and that
\begin{equation}
\label{eq-lower-Riesz-direct-Poincare-Markov}
\bnorm{A_p^\alpha(f)}_{\L^p(\cal{M})}
\leq
C_p
\norm{\partial(f)}_{\L^p(\tilde{\cal{M}})},
\quad
f\in\dom\partial,
\end{equation}
for some constant $C_p \geq 0$. Set $
\theta_p
\ov{\mathrm{def}}{=} 2\min\big\{\frac1p,\frac1{p^*}\big\}$. Then $0 \in \rho(A_{p,0})$ and
\begin{equation}
\label{Poincare-clean}
\norm{f-\E_p(f)}_{\L^p(\cal{M})}
\leq
\frac{2^{1-\theta_p}C_p}{(\theta_p\gamma)^\alpha}
\norm{\partial(f)}_{\L^p(\tilde{\cal{M}})},
\quad
f \in \dom \partial.
\end{equation}
\end{thm}

\begin{proof}
By Proposition~\ref{prop-Markov-quantitative-gap}, we have $0 \in \rho(A_{p,0})$ and
\begin{equation}
\label{estimate-fial-456}
\bnorm{A_{p,0}^{-\alpha}}_{\L^p_0(\cal{M})\to\L^p_0(\cal{M})}
\leq
\frac{2^{1-\theta_p}}{(\theta_p\gamma)^\alpha}.
\end{equation}
Now, Theorem~\ref{thm-direct-Poincare-via-negative-powers}, applied with $P=\E_p$, gives
\begin{align*}
\MoveEqLeft
\norm{f-\E_p(f)}_{\L^p(\cal{M})}
\ov{\eqref{eq-direct-Poincare}}{\leq}
C_p
\bnorm{A_{p,0}^{-\alpha}}_{\L^p_0(\cal{M})\to\L^p_0(\cal{M})}
\norm{\partial(f)}_{\L^p(\tilde{\cal{M}})}
\ov{\eqref{estimate-fial-456}}{\leq}
\frac{2^{1-\theta_p}C_p}{(\theta_p\gamma)^\alpha}
\norm{\partial(f)}_{\L^p(\tilde{\cal{M}})}.
\end{align*}
\end{proof}

Let $\cal{M}$ be a von Neumann algebra equipped with a normalized normal finite faithful trace. Consider a Markov semigroup $(T_t)_{t \geq 0}$ of operators acting on $\cal{M}$ with infinitesimal generator $-A_p$ on $\L^p(\cal{M})$. By the construction of Cipriani and Sauvageot \cite{CiS03} \cite{Cip97} \cite{Cip08} \cite {Cip16}, there exist a symmetric Hilbert $\cal{M}$-bimodule $(\cal{H},\Phi,\Psi,\J)$ and a $\mathrm{W}^*$-derivation $\partial \co \dom \partial \subset \cal{M} \to \cal{H}$ such that $\partial \co \dom \partial \subset \L^2(\cal{M}) \to \cal{H}$ is closable with closure denoted by $\partial_2$ and satisfying $A_2=\partial_2^*\partial_2$. In other words, this means that the opposite of the infinitesimal generator $-A_2$ of a symmetric Markov semigroup of operators can always be represented as the composition of a <<divergence>> $\partial_2^*$ with a <<gradient>> $\partial_2$.

In the situations considered below, we assume in addition that the first-order differential calculus admits a realization in a finite von Neumann algebra $\tilde{\cal{M}}$: $\dom \partial$ is a weak* dense $*$-subalgebra of $\cal{M}$ such that $\partial(\dom \partial) \subset\tilde{\cal{M}}$ and the Hilbertian norm $\norm{\cdot}_{\L^2(\tilde{\cal{M}})}$ induced by this realization agrees with the norm of the tangent bimodule on $\partial(\dom \partial )$. Throughout this subsection, the duality between $\L^p(\cal M)$ and $\L^{p^*}(\cal M)$ is the bilinear pairing $
\la x,y\ra_{\L^p(\cal M),\L^{p^*}(\cal M)}
\ov{\mathrm{def}}{=}
\tau(xy)$.

\begin{prop}
\label{prop-Markov-Riesz-equivalence}
Suppose that $1 < p < \infty$. Let $(T_t)_{t\geq0}$ be a Markov semigroup on a finite von Neumann algebra $\cal{M}$ and denote by $-A_r$ its generator on $\L^r(\cal{M})$, where $1<r<\infty$. Assume that there exist a finite von Neumann algebra $\tilde{\cal{M}}$, a weak* dense $*$-subalgebra $\dom \partial$ of $\cal{M}$ and a linear map $\partial \co \dom \partial \to \tilde{\cal{M}}$ such that
\begin{enumerate}
\item the subspace $\dom \partial$ is a core for the unbounded operators $A_p$ and $A_{p^*}$,
\item for any $x,y \in \dom \partial$, we have
\begin{equation}
\label{eq-Markov-energy-identity}
\la A_p x,y\ra_{\L^p(\cal{M}),\L^{p^*}(\cal{M})}
=
\la\partial x,\partial y\ra_{\L^p(\tilde{\cal{M}}),\L^{p^*}(\tilde{\cal{M}})}.
\end{equation}
\end{enumerate}
Assume in addition that
\begin{equation}
\label{eq-Markov-direct-Riesz}
\norm{\partial x}_{\L^p(\tilde{\cal{M}})}
\leq
C_p
\bnorm{A_p^{\frac12}x}_{\L^p(\cal{M})},
\qquad
\norm{\partial x}_{\L^{p^*}(\tilde{\cal{M}})}
\leq
C_{p^*}
\bnorm{A_{p^*}^{\frac12}x}_{\L^{p^*}(\cal{M})},
\quad
x \in \dom \partial.
\end{equation}
Then, for $r\in\{p,p^*\}$, the unbounded operator $
\partial
\co
\dom\partial
\subset
\L^r(\cal{M})
\to
\L^r(\tilde{\cal{M}})$ is closable. We denote its closure by $\partial_r$. We have
\begin{equation}
\label{eq-Markov-domain-Riesz}
\dom A_r^{\frac12}
=
\dom\partial_r,
\quad
r \in \{p,p^*\},
\end{equation}
and
\begin{equation}
\label{eq-Markov-direct-Riesz-closed}
\norm{\partial_r x}_{\L^r(\tilde{\cal{M}})}
\leq
C_r
\bnorm{A_r^{\frac12}x}_{\L^r(\cal{M})},
\quad
x \in \dom A_r^{\frac12}.
\end{equation}
Moreover, we have
\begin{equation}
\label{eq-Markov-reverse-explicit-r}
\bnorm{A_r^{\frac12}x}_{\L^r(\cal{M})}
\leq
C_{r^*}
\norm{\Id-\E_r}_{\L^r(\cal{M})\to\L^r(\cal{M})}
\norm{\partial_r x}_{\L^r(\tilde{\cal{M}})},
\quad
x \in \dom\partial_r,
\end{equation}
where $r^*$ denotes the H\"older conjugate exponent of $r$. Finally, we have $
A_p
=
(\partial_{p^*})^*\partial_p$. 
\end{prop}

\begin{proof}
Since the Markov semigroup is symmetric and its $\L^r$-realizations are consistent, we have $
A_p^*=A_{p^*}$. Consequently, the energy identity \eqref{eq-Markov-energy-identity} also holds with $p$ and $p^*$ interchanged. Indeed, for any $x,y \in \dom \partial$, we have
\begin{align*}
\MoveEqLeft
\la A_{p^*}x,y\ra_{\L^{p^*}(\cal{M}),\L^p(\cal{M})}
=
\la x,A_py\ra_{\L^{p^*}(\cal{M}),\L^p(\cal{M})}
=
\la A_py,x\ra_{\L^p(\cal{M}),\L^{p^*}(\cal{M})}\\
&\ov{\eqref{eq-Markov-energy-identity}}{=}
\la\partial y,\partial x\ra_{\L^p(\tilde{\cal{M}}),\L^{p^*}(\tilde{\cal{M}})}
=
\la\partial x,\partial y\ra_{\L^{p^*}(\tilde{\cal{M}}),\L^p(\tilde{\cal{M}})}.
\end{align*}
Fix $r\in\{p,p^*\}$ and denote by $r^*$ its H\"older conjugate exponent. Since $A_r^*=A_{r^*}$, the compatibility of the functional calculus with adjoints \cite[Corollary 5.2.4 p.~116]{MCSA01} gives $
\big(A_r^{\frac12}\big)^*
=
A_{r^*}^{\frac12}$. Moreover, by the product rule for fractional powers \cite[Proposition 3.1.1 (c) p.~61]{Haa06}, we have $
A_r^{\frac12}A_r^{\frac12}x
=
A_rx$ for any $x \in \dom A_r$. Hence, for any $x,y \in \dom \partial$, we obtain
\begin{align}
\MoveEqLeft
\label{eq-Markov-half-energy}
\big\la A_r^{\frac12}x,A_{r^*}^{\frac12}y
\big\ra_{\L^r(\cal{M}),\L^{r^*}(\cal{M})}
=
\big\la A_r^{\frac12}A_r^{\frac12}x,y
\big\ra_{\L^r(\cal{M}),\L^{r^*}(\cal{M})}
=
\la A_rx,y
\ra_{\L^r(\cal{M}),\L^{r^*}(\cal{M})}
\\
&=
\la\partial x,\partial y
\ra_{\L^r(\tilde{\cal{M}}),\L^{r^*}(\tilde{\cal{M}})}.
\nonumber
\end{align}
Using H\"older's inequality and \eqref{eq-Markov-direct-Riesz} at the exponent $r^*$, we infer that
\begin{align}
\MoveEqLeft
\label{eq-Markov-duality-core}
\Big|
\la A_r^{\frac12}x,A_{r^*}^{\frac12}y
\ra_{\L^r(\cal{M}),\L^{r^*}(\cal{M})}
\Big|
\ov{\eqref{eq-Markov-half-energy}}{\leq}
\norm{\partial x}_{\L^r(\tilde{\cal{M}})}
\norm{\partial y}_{\L^{r^*}(\tilde{\cal{M}})}
\\
&\ov{\eqref{eq-Markov-direct-Riesz}}{\leq}
C_{r^*}
\norm{\partial x}_{\L^r(\tilde{\cal{M}})}
\bnorm{A_{r^*}^{\frac12}y}_{\L^{r^*}(\cal{M})},
\quad
x,y\in\dom\partial.
\nonumber
\end{align}
Since $\dom\partial$ is a core for $A_{r^*}$, Proposition \ref{core-A-alpha} shows that it is also a core for
$A_{r^*}^{\frac12}$. Consequently, we have
\[
\ovl{
A_{r^*}^{\frac12}(\dom\partial)
}
=
\ovl{\Ran A_{r^*}^{\frac12}}.
\]
By \eqref{inclusion-range} and the description of the reduced space in Section
\ref{sec-gaps}, we have $
\ovl{\Ran A_{r^*}^{\frac12}}
=
\ovl{\Ran A_{r^*}}
=
\L^{r^*}_0(\cal{M})$. Therefore
\begin{equation}
\label{density-half-range-Markov}
\ovl{
A_{r^*}^{\frac12}(\dom\partial)
}
=
\L^{r^*}_0(\cal{M}).
\end{equation}
Let $z \in \L^{r^*}_0(\cal{M})$. By
\eqref{density-half-range-Markov}, there exists a sequence $(y_n)$ in
$\dom\partial$ such that $
A_{r^*}^{\frac12}y_n
\to z$ in $\L^{r^*}(\cal{M})$. Applying
\eqref{eq-Markov-duality-core} and passing to the limit, we obtain
\begin{equation}
\label{eq-Markov-duality-reduced}
\Big|
\big\la A_r^{\frac12}x,z \big\ra_{\L^r(\cal{M}),\L^{r^*}(\cal{M})} \Big|
\leq
C_{r^*}
\norm{\partial x}_{\L^r(\tilde{\cal{M}})}
\norm{z}_{\L^{r^*}(\cal{M})},
\quad x \in \dom\partial.
\end{equation}
Now, let $z\in\L^{r^*}(\cal{M})$. By \eqref{inclusion-range}, we have
\[
A_r^{\frac12}x
\in
\ovl{\Ran A_r^{\frac12}}
=
\ovl{\Ran A_r}
=
\L^r_0(\cal{M})
=
\ker\E_r.
\]
Moreover, the extensions of the conditional expectation satisfy $
\E_r^*
=
\E_{r^*}$. Consequently, we have
\begin{align*}
\la A_r^{\frac12}x,\E_{r^*}z
\ra_{\L^r(\cal{M}),\L^{r^*}(\cal{M})}
&=
\la \E_r A_r^{\frac12}x,z
\ra_{\L^r(\cal{M}),\L^{r^*}(\cal{M})}
=0.
\end{align*}
Thus
\[
\la A_r^{\frac12}x,z
\ra_{\L^r(\cal{M}),\L^{r^*}(\cal{M})}
=
\la A_r^{\frac12}x,(\Id-\E_{r^*})z
\ra_{\L^r(\cal{M}),\L^{r^*}(\cal{M})}.
\]
Since $(\Id-\E_{r^*})z$ belongs to $\L^{r^*}_0(\cal{M})$, we can use
\eqref{eq-Markov-duality-reduced} and obtain
\begin{align*}
\MoveEqLeft
\left|
\la A_r^{\frac12}x,z
\ra_{\L^r(\cal{M}),\L^{r^*}(\cal{M})}
\right|
\leq
C_{r^*}
\norm{\partial x}_{\L^r(\tilde{\cal{M}})}
\norm{(\Id-\E_{r^*})z}_{\L^{r^*}(\cal{M})}\\
&\leq
C_{r^*}
\norm{\Id-\E_{r^*}}_{\L^{r^*}(\cal{M})\to\L^{r^*}(\cal{M})}
\norm{\partial x}_{\L^r(\tilde{\cal{M}})}
\norm{z}_{\L^{r^*}(\cal{M})}.
\end{align*}
Since $
(\Id-\E_r)^*
=
\Id-\E_{r^*}$, we have
\[
\norm{\Id-\E_{r^*}}_{\L^{r^*}(\cal{M})\to\L^{r^*}(\cal{M})}
=
\norm{\Id-\E_r}_{\L^r(\cal{M})\to\L^r(\cal{M})}.
\]
Taking the supremum over all $z\in\L^{r^*}(\cal{M})$ satisfying $\norm{z}_{\L^{r^*}(\cal{M})}\leq1$, we conclude that
\begin{equation}
\label{eq-Markov-reverse-Riesz-core}
\bnorm{A_r^{\frac12}x}_{\L^r(\cal{M})}
\leq
C_{r^*}
\norm{\Id-\E_r}_{\L^r(\cal{M}) \to \L^r(\cal{M})}
\norm{\partial x}_{\L^r(\tilde{\cal{M}})},
\quad
x \in \dom\partial.
\end{equation}
Now, we prove that $\partial$ is closable on $\L^r(\cal{M})$. Let $(x_n)$ be a sequence in $\dom\partial$ such that $x_n \to 0$ in $\L^r(\cal{M})$ and $\partial x_n \to y$ in $\L^r(\tilde{\cal{M}})$ for some $y\in\L^r(\tilde{\cal{M}})$. In particular, the sequence $(\partial x_n)$ is Cauchy. By \eqref{eq-Markov-reverse-Riesz-core}, for any integers $n,m\geq1$,
\begin{align*}
\MoveEqLeft
\bnorm{
A_r^{\frac12}(x_n-x_m)
}_{\L^r(\cal{M})}
\leq
C_{r^*}
\norm{\Id-\E_r}_{\L^r(\cal{M})\to\L^r(\cal{M})}
\norm{\partial x_n-\partial x_m}_{\L^r(\tilde{\cal{M}})}.
\end{align*}
Hence $\big(A_r^{\frac12}x_n\big)$ is a Cauchy sequence in
$\L^r(\cal{M})$. Let $z\in\L^r(\cal{M})$ denote its limit. Since $
x_n\to 0$ and $
A_r^{\frac12}x_n\to z$ in $\L^r(\cal{M})$ and the operator $A_r^{\frac12}$ is closed, we obtain $
z=A_r^{\frac12}(0)=0$. Therefore $
A_r^{\frac12}x_n
\to 0$
in $\L^r(\cal{M})$. Using the direct Riesz estimate
\eqref{eq-Markov-direct-Riesz}, we infer that
\[
\norm{\partial x_n}_{\L^r(\tilde{\cal{M}})}
\leq
C_r
\bnorm{A_r^{\frac12}x_n}_{\L^r(\cal{M})}
\to 0.
\]
Since $\partial x_n \to y$, we conclude that $y=0$. This proves that $
\partial
\co
\dom\partial
\subset
\L^r(\cal{M})
\to
\L^r(\tilde{\cal{M}})$ is closable. We denote its closure by $\partial_r$. Now, we can extend both Riesz estimates from the common core. Since $\dom\partial$ is a core for $A_r$, Proposition \ref{core-A-alpha} shows that it is a core for $A_r^{\frac12}$. Hence part~1 of Proposition \ref{Prop-derivation-closable-sgrp-bis}, together with
\eqref{eq-Markov-direct-Riesz}, yields $
\dom A_r^{\frac12}
\subset
\dom\partial_r$ and
\begin{equation}
\label{eq-Markov-direct-closed-intermediate}
\norm{\partial_r x}_{\L^r(\tilde{\cal{M}})}
\leq
C_r
\bnorm{A_r^{\frac12}x}_{\L^r(\cal{M})},
\quad
x \in \dom A_r^{\frac12}.
\end{equation}
On the other hand, part~2 of Proposition
\ref{Prop-derivation-closable-sgrp-bis}, applied to
\eqref{eq-Markov-reverse-Riesz-core}, gives $
\dom\partial_r
\subset
\dom A_r^{\frac12}$ and
\begin{equation}
\label{eq-Markov-reverse-closed-intermediate}
\bnorm{A_r^{\frac12}x}_{\L^r(\cal{M})}
\leq
C_{r^*}
\norm{\Id-\E_r}_{\L^r(\cal{M})\to\L^r(\cal{M})}
\norm{\partial_r x}_{\L^r(\tilde{\cal{M}})},
\quad
x \in \dom\partial_r.
\end{equation}
Consequently, we have
\[
\dom A_r^{\frac12}
=
\dom\partial_r,
\quad
r \in \{p,p^*\},
\]
and \eqref{eq-Markov-direct-Riesz-closed} and
\eqref{eq-Markov-reverse-explicit-r} follow.

It remains to prove the factorization. Since $
\dom A_p
\subset
\dom A_p^{\frac12}
=
\dom\partial_p$ and
\[
\dom A_p^*
=
\dom A_{p^*}
\subset
\dom A_{p^*}^{\frac12}
=
\dom\partial_{p^*},
\]
the domain inclusions required in Proposition
\ref{prop-generator-gradient-factorization} are satisfied. Moreover,
$\dom\partial$ is a core for $A_p$ and $A_{p^*}=A_p^*$ by assumption,
and it is a core for $\partial_p$ and $\partial_{p^*}$ by the definition
of these operators as the closures of $\partial$. Finally,
\eqref{eq-Markov-energy-identity} is precisely the required identity on
the common core. Therefore Proposition
\ref{prop-generator-gradient-factorization}, applied with $
X=\L^p(\cal{M})$, $Y=\L^p(\tilde{\cal{M}})$, $C=\widetilde C=\dom\partial$, 
gives $
A_p
=
(\partial_{p^*})^*\partial_p$. 
\end{proof}

\begin{cor}
\label{cor-Markov-Poincare-from-Riesz}
Under the assumptions of Proposition \ref{prop-Markov-Riesz-equivalence}, suppose that
$\gamma\ov{\mathrm{def}}{=} \gap(A_2) > 0$. Then
\[
\norm{x-\E_p(x)}_{\L^p(\cal M)}
\leq
\frac{2^{2-\theta_p}C_{p^*}}
     {\sqrt{\theta_p\gamma}}
\norm{\partial_p x}_{\L^p(\tilde{\cal M})},
\quad
x\in\dom\partial_p,
\]
where $
\theta_p
\ov{\mathrm{def}}{=}
2\min\left\{\frac1p,\frac1{p^*}\right\}$.
\end{cor}

\begin{proof}
Combine Proposition~\ref{prop-Markov-Riesz-equivalence} with Theorem~\ref{Th-Poincare-von-Neumann}.
\end{proof}

Thanks to the discussion in Section \ref{sec-gaps}, this result can be used with semigroups with compact resolvent or hypercontractive semigroups. The preceding Poincar\'e result applies to the Poincar\'e consequences of the examples treated in \cite{JLZZ24}, including Poisson semigroups on free groups, Ornstein--Uhlenbeck semigroups associated with mixed $Q$-Gaussian von Neumann algebras, and various hypercontractive semigroups of Fourier multipliers on group von Neumann algebras.

\begin{prop}
\label{prop-OU-negative-square-root-logp}
Let $\cal{M}$ be a von Neumann algebra equipped with a normalized normal finite faithful trace. Consider a Markov semigroup $(T_t)_{t \geq 0}$ of operators acting on $\cal{M}$ with generator $-A$, which is contractive from $\L^2(\cal{M})$ into $\L^p(\cal{M})$ for any  $2 \leq p < \infty$ and any $t \geq \frac{\log(p-1)}{2}$. Let $A_{p,0}$ be the part of the operator $A_p$ in $\L^p_0(\cal{M})$. Suppose that $0 < \alpha < 1$. For any $2 \leq p <\infty$, we have
\begin{equation}
\label{eq-OU-negative-square-root-upper}
\bnorm{A_{p,0}^{-\alpha}}_{\L^p_0(\cal{M})\to\L^p_0(\cal{M})}
\leq
1+
\frac1{\Gamma(\alpha+1)}
\bigg(\frac{\log(p-1)}2\bigg)^\alpha.
\end{equation}
In particular, we have
$$
\bnorm{A_{p,0}^{-\alpha}}_{\L^p_0(\cal{M}) \to \L^p_0(\cal{M})}
\lesssim (\log p)^\alpha.
$$
\end{prop}

\begin{proof}
Let $2 \leq p < \infty$, and set $s_p \ov{\mathrm{def}}{=} \frac12\log(p-1)$. By the assumption, $T_{s_p}$ is a contraction from \(\L^2(\cal{M})\) into the space $\L^p(\cal{M})$. Taking $p=4$ in the hypercontractivity assumption and applying Proposition~\ref{prop-hypercontractivity-spectral-gap} with $t_0=\frac{\log 3}{2}$, we obtain $
\gap(A_2)
\geq 1$. Let $f \in \L^p_0(\cal{M})$. Since $\tau$ is normalized and $p\geq2$, we have the contractive inclusion $\L^p(\cal{M})\subset \L^2(\cal{M})$. Moreover, by compatibility of the conditional expectations, $f \in \L^2_0(\cal{M})$. Hence, by the spectral theorem,
\[
\norm{T_t(f)}_{\L^2(\cal{M})}
\leq
\e^{-t}\norm{f}_{\L^2(\cal{M})},
\quad t\geq0.
\]
If $t>s_p$, we can write
\[
T_t=T_{s_p}T_{t-s_p}.
\]
Since $T_{s_p}$ is contractive from $\L^2(\cal{M})$ into $\L^p(\cal{M})$, we obtain
\[
\norm{T_t(f)}_{\L^p(\cal{M})}
\leq
\norm{T_{t-s_p}(f)}_{\L^2(\cal{M})}
\leq
\e^{-(t-s_p)}\norm{f}_{\L^2(\cal{M})}
\leq
\e^{-(t-s_p)}\norm{f}_{\L^p(\cal{M})}.
\]
Consequently, the restriction of $(T_t)_{t\geq0}$ to $\L^p_0(\cal{M})$ is uniformly exponentially stable. Since $(T_t)_{t\geq0}$ is bounded holomorphic on $\L^p(\cal{M})$, Proposition~\ref{prop-reduced-exp-stability-closed-range} yields $0 \in \rho(A_{p,0})$. Consequently, we can use the (corrected) semigroup formula in \cite[Corollary 15.2.15 p.~449]{HvNVW23}. Since $A_{p,0}$ is boundedly invertible, it holds for any $f \in \L^p_0(\cal{M})$ and gives
\begin{equation}
\label{eq-subordination-N-minus-half}
A_{p,0}^{-\alpha}f
=
\frac1{\Gamma(\alpha)}
\int_0^\infty t^{\alpha-1}T_tf \d t,
\end{equation}
%
%
%
where the integral is understood as a Bochner integral. If $0 < t \leq s_p$, then contractivity of the operator $T_t$ gives $
\norm{T_t(f)}_{\L^p(\cal{M})}
\leq
\norm{f}_{\L^p(\cal{M})}$. 
Using \eqref{eq-subordination-N-minus-half}, we obtain
\begin{equation}
\label{fin-29-19}
\bnorm{A_{p,0}^{-\alpha}f}_{\L^p(\cal{M})}
\leq
\frac1{\Gamma(\alpha)}
\bigg( \int_0^{s_p}t^{\alpha-1} \d t
+
\int_{s_p}^\infty t^{\alpha-1}\e^{-(t-s_p)} \d t \bigg)
\norm{f}_{\L^p(\cal{M})}.
\end{equation}
The first integral is equal to $\frac{s_p^{\alpha}}{\alpha}$. For the second one, after the change of variable $u=t-s_p$, we have
\begin{equation}
\label{inter-23456ui}
\int_{s_p}^\infty
t^{\alpha-1}\e^{-(t-s_p)} \d t
=\int_0^\infty (u+s_p)^{\alpha-1} \e^{-u} \d u
\leq \int_0^\infty u^{\alpha-1} \e^{-u} \d u
=\Gamma(\alpha).
\end{equation}
Thus, we have
\begin{align*}
\MoveEqLeft
\bnorm{A_{p,0}^{-\alpha}f}_{\L^p(\cal{M})}
\ov{\eqref{fin-29-19} \eqref{inter-23456ui}}{\leq} \bigg(\frac1{\alpha\Gamma(\alpha)} \bigg(\frac{\log(p-1)}{2}\bigg)^{\alpha} + 1\bigg)\norm{f}_{\L^p(\cal{M})} \\      
&=\bigg(\frac1{\Gamma(\alpha+1)} \bigg(\frac{\log(p-1)}{2}\bigg)^{\alpha} + 1\bigg)\norm{f}_{\L^p(\cal{M})},
\end{align*}
which proves \eqref{eq-OU-negative-square-root-upper}.
\end{proof}

\subsection{Poincar\'e inequalities on metric measure spaces satisfying the Riemannian curvature dimension condition $\RCD(K,N)$}
\label{sec-metric-measure-spaces}

In this section, a triple $(\X,\dist,\mu)$ is said to be a metric measure space if $(\X,\dist)$ is a complete separable metric space and if $\mu$ is a Borel probability measure on $X$ with full support. In \cite[p.~4]{EKS15}, the authors introduced the class of metric measure spaces satisfying the Riemannian curvature-dimension condition $\RCD^*(K,N)$, where $K \in \R$ and $N \in [1,\infty)$. This condition may be regarded as a synthetic version, in the non-smooth setting, of the lower Ricci curvature bound $\mathrm{Ric} \geq K$ together with the upper dimension bound $\dim \leq N$. It strengthens the curvature-dimension condition $\mathrm{CD}(K,N)$ introduced by Lott and Villani \cite{LoV09} and Sturm \cite{Stu06a,Stu06b}, which is a generalization of the condition $\mathrm{CD}(K,N)$ of Bakry and Émery. One of its main Riemannian features is that it admits a heat semigroup $(T_t)_{t \geq 0}$, which is a symmetric sub-Markovian semigroup acting on $\L^p(\X) \ov{\mathrm{def}}{=} \L^p(\X,\mu)$ generated by an unbounded  operator $\Delta_p$, called Laplacian. We refer to \cite[Section 2.3]{Tew18} for a nice summary of its properties. Moreover, a differential $\d$ is equally present in this context. It is defined for reasonable functions on $\X$ and takes values in a space $\L^p(\mathrm{T}^*\X)$. For more details, we refer to \cite{AHT18,AHPT21,AMS19,CTT23,EKS15,JLZ16,KuK19,KuL21} and to the references therein.

Let $(\X,\dist,\mu)$ be a metric measure space satisfying the Riemannian curvature dimension condition $\RCD^*(K,N)$ with $K > 0$ and $N > 1$. Then by \cite[Theorem 4.22 p.~73]{EKS15} the spectrum of the Laplacian $-\Delta_2$ on the Hilbert space $\L^2(\X,\mu)$ is discrete and the first non-zero eigenvalue $\lambda_1$ satisfies the following bound:
\begin{align}
\label{eq:lichnerowitz}
\lambda_1
\geq \frac{N}{N-1}K. 
\end{align}
Since $-\Delta_2$ is a positive selfadjoint operator on $\L^2(\X)$, this means that
\begin{equation}
\label{eq-RCD-gap}
\gamma
\ov{\mathrm{def}}{=}
\Gap(-\Delta_2)
\geq
\frac{N}{N-1}K.
\end{equation}
Note that this estimate of the spectral gap is sharp. Cavalletti and Milman proposed in \cite[p.~123]{CM21} a variant $\RCD(K,N)$ of the $\RCD^*(K,N)$-condition and established in \cite[Corollary 13.7 p.~123]{CM21} that the $\RCD(K,N)$-condition is equivalent to the $\RCD^*(K,N)$-condition. The condition $\RCD(K,\infty)=\RCD^*(K,\infty)$ was previously introduced in \cite[p.~1408]{AGS14a}.

Riesz estimates are also available in this setting. Suppose that $1 < p < \infty$. Let $(\X,\dist,\mu)$ be a metric measure space satisfying the Riemannian curvature dimension condition $\RCD(K,\infty)$, with $K \in \mathbb{R}$. Let $K_- \ov{\mathrm{def}}{=} \max\{0, -K\}$. Then, by \cite[Theorem 1.1]{CTT23}, we have\footnote{\thefootnote. The preprint \cite{CTT23} contains some inconsistencies concerning the domains involved. In particular, the authors use $\W^{1,2}(\X)$, whereas the appropriate space is $\W^{1,p}(\X)$. We contacted the authors regarding this issue but, as of the time of writing, have not received any clarification. The inequality is correct for elements in $\W^{1,p}(\X)$.}
\[ 
\norm{ \d f }_{\L^p(\mathrm{T}^*\X)}
\leq c_p \norm{ \sqrt{ K_- - \Delta_p } f }_{\L^p(\X)}, \quad f \in \W^{1,p}(\X)
\]
and
\begin{equation}
\label{reverse-metric-measure-spaces}
\norm{ \sqrt{ K_{-} - \Delta_p } f }_{\L^p(\X)} 
\leq \sqrt{K_-} \norm{f}_{\L^p(\X)} + c_p \norm{\d f}_{\L^p(\mathrm{T}^*\X)}, \quad f \in \W^{1,p}(\X),
\end{equation}
where $c_p \ov{\mathrm{def}}{=} 16 \max\{p, \frac{p}{p-1} \}$. 
Now, we obtain an $\L^p$-Poincar\'e inequality.

\begin{thm}
\label{thm-RCD-quantitative-Poincare}
Let $(\X,\dist,\mu)$ be a metric measure space satisfying the Riemannian curvature dimension condition $\RCD(K,N)$ with $K > 0$ and $N \in (1,\infty)$. Suppose that $1 < p < \infty$. Set $
\theta_p
\ov{\mathrm{def}}{=}
2\min\big\{\frac1p,\frac1{p^*}\big\}$ and $
c_p
\ov{\mathrm{def}}{=}
16\max\big\{p,\frac{p}{p-1}\big\}$. Then
\begin{equation}
\label{Poincare-manifold-bundle}
\norm{f-P_p(f)}_{\L^p(\X)}
\leq
2^{1-\theta_p}c_p
\sqrt{\frac{N-1}{\theta_pNK}}  \norm{\d f}_{\L^p(\mathrm{T}^*\X)},
\quad
f \in \W^{1,p}(\X),
\end{equation}
where $
P_p\co\L^p(\X)\to\L^p(\X)$ is the mean ergodic projection onto the subspace $\ker(-\Delta_p)$.
\end{thm}


\begin{proof}
The heat semigroup is conservative, symmetric and Markovian. Hence it can be regarded as a Markov semigroup on $\L^\infty(\X)$, whose infinitesimal generator on $\L^p(\X)$ is $-A_p=\Delta_p$. Its mean ergodic projection on $\L^p(\X)$ is $P_p$.  
Applying Proposition~\ref{prop-Markov-quantitative-gap} with
$\alpha=\frac12$, we obtain
\begin{align}
\MoveEqLeft
\bnorm{(-\Delta_{p,0})^{-\frac12}}_
{\ovl{\Ran\Delta_p}\to\ovl{\Ran\Delta_p}}
\leq
\frac{2^{1-\theta_p}}{\sqrt{\theta_p\gamma}}
\ov{\eqref{eq-RCD-gap}}{\leq} 
2^{1-\theta_p}
\sqrt{\frac{N-1}{\theta_pNK}}.
\label{eq-RCD-negative-half-power}
\end{align}
Since $K>0$, we have $K_-=0$. Moreover, $\RCD(K,N)$ with
$N<\infty$ implies $\RCD(K,\infty)$. Therefore
\eqref{reverse-metric-measure-spaces} gives
\begin{equation}
\label{eq-RCD-reverse-Riesz}
\bnorm{(-\Delta_p)^{\frac12}f}_{\L^p(\X)}
\leq
c_p
\norm{\d f}_{\L^p(\mathrm{T}^*\X)},
\quad
f \in \W^{1,p}(\X).
\end{equation}
Now, Theorem~\ref{thm-direct-Poincare-via-negative-powers}, applied with $A_p=-\Delta_p$, $\alpha=\frac12$, $\partial=\d$, $P=P_p$, together with \eqref{eq-RCD-negative-half-power} and
\eqref{eq-RCD-reverse-Riesz}, yields
\begin{align*}
\MoveEqLeft
\norm{f-P_p(f)}_{\L^p(\X)}
\ov{\eqref{eq-direct-Poincare}}{\leq} 
\bnorm{(-\Delta_{p,0})^{-\frac12}}
\bnorm{(-\Delta_p)^{\frac12}f}_{\L^p(\X)}
\ov{\eqref{eq-RCD-negative-half-power}\eqref{eq-RCD-reverse-Riesz}}{\leq} 
2^{1-\theta_p}c_p
\sqrt{\frac{N-1}{\theta_pNK}}
\norm{\d f}_{\L^p(\mathrm{T}^*\X)}.
\end{align*}
This proves \eqref{Poincare-manifold-bundle}.
\end{proof}

\subsection{Poincar\'e inequalities with subelliptic gradients on compact Lie groups}
\label{Sec-Lie-groups}

Let $G$ be a connected Lie group equipped with a left Haar measure $\mu_G$. We consider a finite sequence $X \ov{\mathrm{def}}{=}(X_1,\ldots,X_m)$ of left invariant smooth vector fields which generate the Lie algebra $\g$ of the group $G$ such that the vectors $X_1(e),\ldots, X_m(e)$ are linearly independent. We say that it is a family of left-invariant H\"ormander vector fields. For any $r > 0$ and any $x \in G$, we denote by $B(x,r)$ the open ball with respect to the Carnot-Carath\'eodory metric centered at $x$ and of radius $r$. Its volume $V(r) \ov{\mathrm{def}}{=} \mu_G(B(x,r))$ does not depend on $x$. It is well-known, e.g.~\cite[p.~124]{VSCC92}, that there exists an integer $d \geq 1$ and some constants $c,C > 0$ such that
\begin{equation}
\label{local-dim}
c r^d 
\leq V(r) 
\leq C r^d, \quad r \in ]0, 1[.
\end{equation}
The integer $d$ is called the local dimension of $(G,X)$. When $r \geq 1$, only two situations may occur, independently of the choice of $X$ (see e.g.~\cite[p.~26]{DtER03}): either $G$ has polynomial volume growth, which means that there exist an integer $D \in \N$ and some constants $c',C' > 0$ such that
\begin{equation}
\label{poly-growth}
c' r^D 
\leq V(r) 
\leq 
C' r^D, \quad r \geq 1,
\end{equation}
or $G$ has exponential volume growth, which means that there exist $c_1,C_1, c_2,C_2 > 0$ such that
$$
c_1 \e^{c_2 r}
\leq  V(r)
\leq C_1 \e^{C_2 r}, \quad r \geq 1.
$$
When $G$ has polynomial volume growth, the integer $D$ in \eqref{poly-growth} is called the dimension at infinity of $G$. Note that, contrary to $d$, it only depends on $G$ and not on $X$, see \cite[Chapter 4]{VSCC92}. The case $D=0$ corresponds to uniformly bounded volume, hence to compactness, see \cite[pp.~256-257]{Rob91} and \cite[p.~26]{DtER03}. Note that by \cite[II.4.5 p.~26]{DtER03} or \cite[p.~381]{Rob91}, each connected Lie group of polynomial growth is unimodular.

 
If the connected Lie group $G$ is unimodular, we define the subelliptic Laplacian $\Delta$ on $G$ by $ 
\Delta 
\ov{\mathrm{def}}{=} -\sum_{k=1}^{m} X_k^2$. We can consider the smallest closed extension $\Delta_p \co \dom \Delta_p \subset \L^p(G) \to \L^p(G)$ of the closable unbounded operator $\Delta|_{\C^\infty_c(G)}$ on the Banach space $\L^p(G)$, where $\C^\infty_c(G)$ is the space of smooth functions on $G$ with compact support. Finally, we denote by $(T_t)_{t \geq 0}$ the strongly continuous semigroup of positive contractive convolution operators on the Banach space $\L^p(G)$ if $1 \leq p <\infty$, see \cite[pp.~20-21]{VSCC92} and \cite[p.~301]{Rob91}. This semigroup is bounded holomorphic on $\L^p(G)$ by \cite[Proposition 5.4 p.~51, Lemma 3.1 p.~26]{JMX06} if $1 < p < \infty$. 

Suppose that $1 < p < \infty$ and that the connected Lie group $G$ has polynomial volume growth. By \cite[Theorem 2 p.~692]{Ale92} and \cite[p.~339]{CRT01}, we have the Riesz equivalence
\begin{equation*}
\bnorm{\Delta^{\frac12}(f)}_{\L^p(G)}
\approx_{p,G,X} \sum_{k=1}^{m} \bnorm{X_k(f)}_{\L^p(G)}, \quad f \in \C^\infty_c(G).
\end{equation*}
If $\nabla \co \C^\infty_c(G) \to \L^p(G,\ell^p_m)$, $f \mapsto (X_1 f,\ldots,X_m f)$
denotes the subelliptic gradient with closure $\nabla_p$, then we can reformulate and extend very slightly by \cite[Proposition 3.4 p.~461]{Arh24a} the previous result in the following way. We have $\dom \nabla_p= \dom \Delta_p^{\frac12}$ and
\begin{equation}
\label{Riesz1}
\bnorm{\Delta_p^{\frac12}(f)}_{\L^p(G)}
\approx_{p,G,X}
\norm{\nabla_p f}_{\L^p(G,\ell^p_m)},
\quad f \in \dom \nabla_p.
\end{equation}

If the connected Lie group $G$ is compact, we assume that the measure $\mu_G$ is normalized, i.e., $\mu_G(G)=1$. 
In this case, we have $\ker \Delta_p = \mathbb{C} 1$ by \cite[Lemma 3.5 (3) p.~462]{Arh24a} and the mean ergodic projection $P \co \L^p(G) \to \L^p(G)$ is given by $
P(f)
=\int_G f \d\mu_G$. Finally, recall that $\ovl{\Ran \Delta_p}=\{ f \in \L^p(G) : \int_G f \d\mu_G =0\}$. 

  
Now, we obtain the following $\L^p$-Poincar\'e inequality. The case $p=2$ seems to be a particular case of \cite[Theorem 1.1 p.~107]{RuS11}.

\begin{cor}
\label{cor-intro-Poincare}
Let $G$ be a connected compact Lie group equipped with its normalized Haar measure $\mu_G$. Let $(X_1,\ldots,X_m)$ be a family of left-invariant H\"ormander vector fields. Suppose that $1 < p < \infty$. We have
\begin{equation}
\label{Poincare-intro-bis}
\norm{f-\int_G f}_{\L^p(G)}
\lesssim_{p,G,X} \norm{\nabla_p(f)}_{\L^p(G,\ell^p_m)}, \quad f \in \dom \nabla_p.
\end{equation}
\end{cor}

\begin{proof}
By \cite[Lemma 4.4 p.~467]{Arh24a}, the operator $\Delta_2 \co \dom \Delta_2 \subset \L^2(G) \to \L^2(G)$ has compact resolvent. By interpolation \cite[p.~78]{Are04}, the unbounded operator $\Delta_p$ also has compact resolvent. By \eqref{compact-resolvent}, this gives $0 \in \rho(\Delta_{p,0})$. The estimate \eqref{eq-lower-Riesz-direct-Poincare} is satisfied by \eqref{Riesz1} with $\alpha=\frac{1}{2}$. So it suffices to use Theorem \ref{thm-direct-Poincare-via-negative-powers} (or Theorem \ref{Th-Poincare-von-Neumann}) with $\partial = \nabla_p$.
\end{proof}

Finally, we estimate the constant of \eqref{eq-direct-Poincare} and \eqref{eq-Poincare-div-full} with respect to $p$ in this context.

\begin{prop}
\label{prop-negative-square-root-Lie group}
Let $G$ be a connected compact Lie group equipped with its normalized Haar measure $\mu_G$. Let $X=(X_1,\ldots,X_m)$ be a family of left-invariant H\"ormander vector fields. Suppose that $1 < p < \infty$. Let $\Delta_{p,0}$ be the part of the subelliptic Laplacian $\Delta_p$ in $\L^p_0(G)$. Then there exists a constant $C_{G,X}$ which depends only on $(G,X)$ such that 
\begin{equation}
\label{eq-negative-square-root-torus-uniform}
\bnorm{\Delta_{p,0}^{-\frac12}}_{\L^p_0(G) \to \L^p_0(G)}
\leq C_{G,X}
\end{equation}
for any $1 < p < \infty$.
\end{prop}

\begin{proof}
On the space $\ovl{\Ran \Delta_p}=\Ran \Delta_p$ (note Proposition \ref{prop-reduced-exp-stability-closed-range}), the formula \cite[Corollary 15.2.15 p.~449]{HvNVW23} gives
\begin{equation}
\label{eq-subordination-torus}
\Delta_{p,0}^{-\frac12}(f)
=
\frac1{\sqrt\pi}
\int_0^\infty t^{-\frac12} T_t(f) \d t,
\quad f \in \Ran \Delta_p.
\end{equation}
We cut the integral in two parts. Using the contractivity of the operator $T_t \co \L^p(G) \to \L^p(G)$, we obtain for any function $f \in \L^p(G)$
\begin{align}
\label{estim-Lie}
\MoveEqLeft
\norm{\int_0^1 t^{-\frac12} T_t(f) \d t}_{\L^p(G)} 
\leq  \int_0^1 t^{-\frac12} \bnorm{T_t(f)}_{\L^p(G)} \d t        
\leq \bigg(\int_0^1 t^{-\frac12}  \d t \bigg) \norm{f}_{\L^p(G)}
=2\norm{f}_{\L^p(G)}.
\end{align}
It remains to estimate the large-time part. According to \cite[p.~38]{DtER03}, $T_t(\Id-P)$ is a convolution operator with kernel $K_t^0$ and there exists a constant $\omega=\omega_{X,G} > 0$ such that
\begin{equation}
\label{estim-kernel}
\bnorm{K_t^0}_{\L^\infty(G)}
\lesssim_{G,X} \e^{-\omega t}.
\end{equation}
Hence, by Young's inequality \cite[p.~519]{Rob91}, we see that for any $t \geq 1$ and any $f \in \Ran \Delta_p$
\begin{align}
\MoveEqLeft
\label{Young}
\bnorm{T_t(f)}_{\L^p(G)}
=\bnorm{T_t(\Id-P)(f)}_{\L^p(G)}
\leq
\bnorm{K_t^0}_{\L^1(G)} \norm{f}_{\L^p(G)} \\
&\leq \bnorm{K_t^0}_{\L^\infty(G)} \norm{f}_{\L^p(G)}
\ov{\eqref{estim-kernel}}{\lesssim_{G,X}} \e^{-\omega t} \norm{f}_{\L^p(G)}.  \nonumber       
\end{align}
Consequently, we obtain the estimate
\begin{align*}
\MoveEqLeft
\norm{\int_1^\infty t^{-\frac12} T_t f \d t}_{\L^p(G)}
\leq \int_1^\infty t^{-\frac12} \bnorm{T_t f}_{\L^p(G)} \d t \\
&\ov{\eqref{Young}}{\lesssim_{G,X}} 
\bigg(\int_1^\infty t^{-\frac12}\e^{-\omega t} \d t \bigg)
\norm{f}_{\L^p(G)}
\lesssim_{G,X}  \norm{f}_{\L^p(G)}.
\end{align*}
Combining this with \eqref{estim-Lie}, we conclude that the formula \eqref{eq-subordination-torus} provides the estimate
\[
\bnorm{\Delta_{p,0}^{-\frac12}f}_{\L^p(G)}
\lesssim_{G,X} \norm{f}_{\L^p(G)}.
\]
\end{proof}
\subsection{Poincar\'e inequalities on quantum tori}
\label{sec-NC-tori}

Consider an integer $d \geq 2$. To each $d \times d$ real skew-symmetric matrix $\theta$, one can associate a 2-cocycle $\sigma_\theta \co \Z^d \times \Z^d \to \mathbb{T}$ of the discrete group $\Z^d$ defined by $\sigma_\theta (m,n) \ov{\mathrm{def}}{=} \e^{2\pi \i \la m, \theta n\ra}$, where $m,n \in \Z^d$. 
We introduce the $d$-dimensional noncommutative torus $\L^\infty(\mathbb{T}_{\theta}^d)$ as the twisted group von Neumann algebra $\VN(\Z^d,\sigma_\theta)$. We refer to the book \cite{GVF01} and to the papers \cite{CXY13} and \cite{XXY18} for background on the noncommutative tori.  One can provide a concrete realization in the following manner. If $(\epsi_n)_{n \in \Z^d}$ is the canonical basis of the complex Hilbert space $\ell^2_{\Z^d}$ and if $m \in \Z^d$, we can consider the bounded operator $U_m \co \ell^2_{\Z^d} \to \ell^2_{\Z^d}$ defined by 
\begin{equation}
\label{def-lambdas}
U_m(\epsi_n)
\ov{\mathrm{def}}{=} \sigma_\theta(m,n) \epsi_{m+n}, \quad n \in \Z^d.	
\end{equation}
The $d$-dimensional noncommutative torus $\L^\infty(\mathbb{T}_{\theta}^d)$ is the von Neumann subalgebra of $\B(\ell^2_{\Z^d})$ generated by the $*$-algebra $
\cal{P}_{\theta}
\ov{\mathrm{def}}{=} \mathrm{span} \big\{ U_m \ : \ m \in \Z^d \big\}$. Recall that for any $m,n \in \Z^d$ we have the equalities
\begin{equation}
\label{product-adjoint-twisted}
U_m U_n 
= \sigma_\theta(m,n) U_{m+n}
\quad \text{and} \quad 
(U_m)^* 
= U_{-m}.	
\end{equation}
The von Neumann algebra $\L^\infty(\mathbb{T}_{\theta}^d)$ is finite with normalized trace given by $\tau(x) \ov{\mathrm{def}}{=} \la\epsi_{0},x(\epsi_{0}) \ra_{\ell^2_{\Z^d}}$, where $x \in \L^\infty(\mathbb{T}_{\theta}^d)$. In particular, we have $\tau(U_m) = \delta_{m=0}$ for any $m \in \Z^d$.

Following \cite[p.~765]{CXY13}, for any integer $j \in \{1,\ldots,d\}$, the partial derivations $\partial_j \co \cal{P}_{\theta} \to \L^p(\mathbb{T}_{\theta}^d)$ are defined by
\begin{equation}
\label{partial-quantum-tori}
\partial_j(U_m)
\ov{\mathrm{def}}{=} 2\pi \i\, m_j U_m, \quad m \in \Z^d.
\end{equation}
Consider the von Neumann algebra $
\tilde{\cal{M}}_d
\ov{\mathrm{def}}{=}
\bigoplus_{j=1}^d \L^\infty(\mathbb{T}_\theta^d)$ equipped with the normal finite faithful trace $
\widetilde{\tau}(x_1,\ldots,x_d)
\ov{\mathrm{def}}{=}
\sum_{j=1}^d\tau(x_j)$. We can introduce the gradient operator $\nabla \co \cal{P}_{\theta} \to \tilde{\cal{M}}_d$ by 
\begin{equation}
\label{gradient-quantum-tori}
\nabla(x)
\ov{\mathrm{def}}{=}
(\partial_1x,\ldots,\partial_dx), \quad x \in \cal{P}_{\theta}.
\end{equation}
If $1 \leq p < \infty$, we have the canonical isometric identification $
\L^p(\widetilde{\cal{M}}_d)
=
\ell_d^p(\L^p(\mathbb{T}_\theta^d))$. If $p>1$, the operator $
\nabla\co\cal{P}_\theta
\subset
\L^p(\mathbb{T}_\theta^d)
\to
\L^p(\widetilde{\cal{M}}_d)$ is closable by \cite[Theorem 5.28 p.~168]{Kat76}. Indeed, on the dense subspace
$\cal{P}_\theta^d$ of $\L^{p^*}(\tilde{\cal{M}}_d)$, its formal adjoint is given by $
\nabla^\dagger(y_1,\ldots,y_d)
=
-\sum_{j=1}^d\partial_jy_j$. We denote by $\nabla_p$ the closure of $\nabla$.

If $1 < p < \infty$ then the Laplacian operator $-A \ov{\mathrm{def}}{=} \partial_1^2+\cdots +\partial_d^2$ is closable on the noncommutative $\L^p$-space $\L^p(\mathbb{T}_{\theta}^d)$ and its closure $-A_p$ is the infinitesimal generator of a strongly continuous semigroup $(T_t)_{t \geq 0}$ of completely positive contractions, called the noncommutative heat semigroup. This semigroup is bounded holomorphic by \cite[Proposition 5.4 p.~51, Lemma 3.1 p.~26]{JMX06}. We have 
\begin{equation}
\label{Laplacian-quantum-tori}
A_p(U_m) 
= 4\pi^2 |m|^2U_m, \quad m \in \Z^d,
\end{equation}
where $|m| \ov{\mathrm{def}}{=} \sqrt{m_1^2+\cdots+m_d^2}$. Moreover, $\cal{P}_\theta$ is a core of $A_p$. The mean ergodic projection $P \co \L^p(\mathbb{T}_{\theta}^d) \to \L^p(\mathbb{T}_{\theta}^d)$ of the semigroup $(T_t)_{t \geq 0}$ is given by $
P(x)
=\tau(x)1$, where $x \in \L^p(\mathbb{T}_{\theta}^d)$. In the next result, the noncommutative $\L^p$-spaces are equipped with their canonical structures of operator spaces \cite{Pis98}.

\begin{prop}
\label{prop-coordinate-Riesz-transform-quantum-torus}
Suppose that $1 < p < \infty$. For any $j \in \{1,\ldots,d\}$, define the $j$-th Riesz transform $R_j \co \cal{P}_{\theta} \to \L^p(\mathbb{T}_{\theta}^d)$ by
\begin{equation}
\label{def-Riesz-quantum-tori}
R_{j}(U_m)
=\begin{cases}
\i\frac{m_j}{|m|}U_m &\text{ if } m \in \Z^d \setminus \{0\} \\
0& \text{ if } m=0
\end{cases}.
\end{equation}
Then each $R_{j}$ extends to a completely bounded Fourier multiplier on $\L^p(\mathbb{T}_\theta^d)$. Finally, we have $
A_p
=
(\nabla_{p^*})^*\nabla_p$, the equality $\dom\nabla_p = \dom A_p^{\frac12}$ and
\begin{equation}
\label{equiv-Riesz-tori-123}
\bnorm{A_p^{\frac12}(x)}_{\L^p(\mathbb{T}_{\theta}^d)}
\approx_{p,d}
\norm{\nabla_p x}_{\ell^p_d(\L^p(\mathbb{T}_{\theta}^d))},
\quad x \in \dom \nabla_p.
\end{equation}
\end{prop}

\begin{proof}
Recall that the classical Riesz transforms on $\L^p(\mathbb{T}^d)$ are completely bounded by a standard argument which is well known to experts, relying on \cite[Remark 1.8 p.~549]{JMP18}. By the transference theorem \cite[Theorem 7.3 p.~785]{CXY13} for completely bounded Fourier multipliers on quantum tori, the multiplier $R_j$ is completely bounded on $\L^p(\mathbb{T}_\theta^d)$ if and only if the corresponding one is completely bounded on $\L^p(\mathbb{T}^d)$.  So each $R_{j}$ is completely bounded on $\L^p(\mathbb{T}_\theta^d)$. In particular, for any $x \in \cal{P}_\theta$, we have
\begin{align}
\MoveEqLeft
\label{eq-direct-Riesz-quantum-torus}
\norm{\nabla x}_{\ell_d^p(\L^p(\mathbb{T}_\theta^d))}
\ov{\eqref{gradient-quantum-tori}}{=} \norm{(\partial_1x,\ldots,\partial_dx)}_{\ell_d^p(\L^p(\mathbb{T}_\theta^d))}
=\bigg(\sum_{j=1}^d \norm{\partial_j x}_{\L^p(\mathbb{T}_\theta^d)}^p \bigg)^{\frac1p} \\
&=\bigg(\sum_{j=1}^d \norm{R_jA_p^{\frac12}x}_{\L^p(\mathbb{T}_\theta^d)}^p \bigg)^{\frac1p}
\leq
C_{p,d}
\bnorm{A_p^{\frac12}x}_{\L^p(\mathbb{T}_\theta^d)}, \nonumber
\end{align}
with $
C_{p,d}
\ov{\mathrm{def}}{=}
d^{\frac1p}
\sup_{1\leq j\leq d}
\norm{R_j}_{\L^p(\mathbb{T}_\theta^d) \to \L^p(\mathbb{T}_\theta^d)}$ and similarly for $p^*$ instead of $p$. For any $x,y\in\cal{P}_\theta$, integration by parts gives
\begin{align*}
\la A_px,y\ra_{\L^p(\mathbb{T}_\theta^d),
\L^{p^*}(\mathbb{T}_\theta^d)}
=
-\sum_{j=1}^d
\la\partial_j^2x,y\ra
=
\sum_{j=1}^d
\la\partial_jx,\partial_jy\ra
=
\la\nabla x,\nabla y\ra_
{\L^p(\tilde{\cal{M}}_d),
\L^{p^*}(\tilde{\cal{M}}_d)}.
\end{align*}
Thus all the assumptions of Proposition
\ref{prop-Markov-Riesz-equivalence} are satisfied with $
\cal{M}
=
\L^\infty(\mathbb{T}_\theta^d)$, $\tilde{\cal{M}}
=
\tilde{\cal{M}}_d$, $\dom\partial
=
\cal{P}_\theta$ and $\partial=\nabla$. The conclusion follows from Proposition \ref{prop-Markov-Riesz-equivalence}.
\end{proof}

\begin{remark}
\normalfont
Under the canonical operator space structure induced by $
\ell_d^p(\L^p(\mathbb{T}_\theta^d))
=
\L^p(\tilde{\cal{M}}_d)$, the vector Riesz transform $\cal{R} \co \L^p(\mathbb{T}_\theta^d) \to \L^p(\tilde{\cal{M}}_d)$ defined by $
\cal{R}x
\ov{\mathrm{def}}{=}
(R_1x,\ldots,R_dx)$ is completely bounded. Indeed, we have
\[
\norm{\cal{R}}_{\cb,\L^p(\mathbb{T}_\theta^d) \to \L^p(\tilde{\cal{M}}_d)}
\leq
\bigg(
\sum_{j=1}^d\norm{R_j}_{\cb,\L^p(\mathbb{T}_{\theta}^d) \to \L^p(\mathbb{T}_{\theta}^d)}^p \bigg)^{\frac1p}.
\]
\end{remark}

Now, we recover the result \cite[p.~4 and Theorem 2.12 p.~25]{XXY18}.

\begin{cor}
Suppose that $1 < p < \infty$. We have
\begin{equation*}
\norm{x-\tau(x) 1}_{\L^p(\mathbb{T}_{\theta}^d)}
\lesssim_{p,d} \norm{\nabla_p x}_{\ell^p_d(\L^p(\mathbb{T}_{\theta}^d))}, \quad x \in \dom \nabla_p.
\end{equation*}
\end{cor}

\begin{proof}
According to \cite[Theorem 7.1]{Arh24b}, the completely bounded local Coulhon--Varopoulos dimension of the noncommutative heat semigroup $(T_t)_{t \geq 0}$ on the noncommutative torus is finite. This property implies, by \cite[Lemma 4.4]{Arh24b}, that the unbounded operator $A_2$ has compact resolvent. 
By \eqref{compact-resolvent}, we deduce that $0 \in \rho(A_{2,0})$. Alternatively, we can observe by \eqref{Laplacian-quantum-tori} that $\gap(A_2)=4\pi^2$ since the family $(U_m)_{m \in \Z^d}$ is an orthonormal basis of the Hilbert space $\L^2(\mathbb{T}_\theta^d)$. Moreover, the mean ergodic projection is given by $\E_p(x)=\tau(x)1$, where $x \in \L^p(\mathbb{T}_{\theta}^d)$. Now, the result follows immediately from Corollary~\ref{cor-Markov-Poincare-from-Riesz} and Proposition~\ref{prop-coordinate-Riesz-transform-quantum-torus}.
\end{proof}

\subsection{Poincar\'e inequalities on compact Lie groups equipped with bi-invariant metrics}
\label{sec-bi-invariant}

Recall that a Riemannian metric $g$ on a Lie group $G$ is left-invariant if the left translation $L_s \co G \to G$ is an isometry for all $s \in G$. Similarly, right-invariant metrics are those for which the right translations are isometries. A metric is said to be bi-invariant if it is both left invariant and right invariant. If $G$ is connected, for a left-invariant Riemannian metric this is equivalent by \cite[Corollary 8.1.13 p.~237]{LeD25} to the condition that the scalar product $\la \cdot,\cdot \ra$ on the Lie algebra $\g$ of $G$ defined by the metric $g$ satisfies
\begin{equation}
\label{ad-skew-adjoint}
\la \ad_z x,  y \ra
= -\la x, \ad_z y \ra, \quad x,y,z \in \g.
\end{equation}
This means that the operator $\ad_z \co \g \to \g$, $x \mapsto [z,x]$ is skew-adjoint for any $z \in \g$. For a left-invariant Riemannian metric on a possibly disconnected Lie group, according to \cite[Theorem 2.2.2 p.~107]{Ham17}, we have to require the stronger condition
\[
\la \Ad_s x,\Ad_s y \ra
=\la x,y\ra,
\quad s \in G, x,y \in \g,
\]
where $\Ad \co G \to \Aut(\g)$ is the adjoint representation of $G$. Indeed, there is a canonical bijection between bi-invariant Riemannian metrics on $G$ and scalar products on the Lie algebra $\mathfrak{g}$ with this property. According to \cite[Corollary 2.2.4 p.~108]{Ham17} or \cite[Theorem 8.1.15 p.~237]{LeD25}, every compact Lie group admits a bi-invariant Riemannian metric. 


Let $G$ be a compact Lie group endowed with a bi-invariant Riemannian metric and let $\g$ be its Lie algebra of left invariant vector fields. Let $X=(X_1,\ldots,X_d)$ be an orthonormal basis of $\mathfrak{g}$. We introduce the operator $A \ov{\mathrm{def}}{=} -X_1^2-\cdots -X_d^2$, which is independent of the orthonormal basis according to \cite[Proposition 1.3.1 p.~18]{App14}, and the associated heat semigroup $(T_t)_{t \geq 0}$ defined by $T_t \ov{\mathrm{def}}{=} \e^{-A t}$. We refer to \cite[Theorem 1 p.~38]{Ste70} for some properties of this semigroup. If $1 < p < \infty$, we denote by $-A_p$ the generator of the consistent extension of the semigroup $(T_t)_{t \geq 0}$ on the Banach space $\L^p(G)$. The subspace $\C^\infty(G)$ is a core of the operator $A_p$. 

Since the left-invariant vector fields $X_1,\ldots,X_d$ form a global orthonormal frame of the tangent bundle $\mathrm{T}G$, they induce an isometric identification $
\L^p(G,\mathrm{T}_{\mathbb{C}}G)
\cong
\L^p(G,\ell_d^2)$. We introduce the gradient operator
\begin{equation}
\label{gradient-32}
\nabla \co \C^\infty(G) \subset \L^p(G) \to \L^p(G,\ell^2_d),f \mapsto (X_1 f,\ldots,X_d f).
\end{equation}
Note that $\nabla$ admits the densely defined formal adjoint $
\nabla^\dagger
\co
\C^\infty(G,\ell_d^2)
\subset
\L^{p^*}(G,\ell_d^2)
\to
\L^{p^*}(G)$, $g=(g_1,\ldots,g_d) \mapsto-\sum_{j=1}^d X_jg_j$. By \cite[Theorem 5.28 p.~168]{Kat76}, the unbounded operator $\nabla \co \C^\infty(G) \subset \L^p(G) \to \L^p(G,\ell^2_d)$ is closable. We denote its closure by $\nabla_p$. 

Riesz estimates in this context are available. Suppose that $1 < p < \infty$. Arcozzi proved in \cite[Theorem 1 p.~203]{Arc98} that the vector Riesz transform $\nabla A_p^{-\frac12}$ restricted to \textit{real-valued} functions of the subspace $\ovl{\Ran A_p}$ has norm at most 
\begin{equation*}
C_p
\ov{\mathrm{def}}{=} 2\big(\max\{p ,\tfrac{p}{p-1}\}-1\big).
\end{equation*}
Applying this estimate to $A_p^{\frac12}f$ gives the estimate for \textit{real-valued} functions
\begin{equation}
\label{Arccozi}
\norm{\nabla(f)}_{\L^p(G,\ell^2_d)} 
\leq C_p \bnorm{A_p^{\frac12}f}_{\L^p(G)}, \quad f \in \C^\infty(G,\R).
\end{equation}  

Now, we obtain an $\L^p$-Poincar\'e inequality.

\begin{cor}
\label{cor-Poincare-Lie-2}
Let $G$ be a compact Lie group equipped with a bi-invariant metric $g$ and its normalized Haar measure. Suppose that $1 < p < \infty$. Let $P \co \L^p(G) \to \L^p(G)$ denote the mean ergodic projection onto $\ker A_p$. We have
\begin{equation}
\label{Poincare-intro-tri}
\norm{f-P(f)}_{\L^p(G)}
\lesssim_{p,G,g} \norm{\nabla_p(f)}_{\L^p(G,\ell^2_d)}, \quad f \in \dom \nabla_p.
\end{equation}
\end{cor}

\begin{proof}
According to \cite[Theorem 1 viii) p.~38]{Ste70}, for any $t > 0$ the operator $T_t$ is a convolution operator by a function $K_t \co G \to \mathbb{C}$ of class $\C^\infty$. In particular, each $K_t$ is square-integrable on the compact group $G$. By \cite[Theorem 3.1 and Proposition 3.1 p.~222]{App09}, the operator $T_t$ is trace-class for any $t > 0$. By \cite[Theorem 4.29 p.~119]{EnN00}, we deduce that the unbounded operator $A_2 \co \dom A_2 \subset \L^2(G) \to \L^2(G)$ has compact resolvent. By interpolation \cite[p.~78]{Are04}, the unbounded operator $A_p$ also has compact resolvent. Part 1 of Proposition \ref{Prop-derivation-closable-sgrp-bis} and \eqref{Arccozi} imply that $
\dom A_p^{\frac12}\subset \dom\nabla_p$ and
\begin{equation}
\label{eq-Riesz-Lie-direct}
\norm{\nabla_p(f)}_{\L^p(G,\ell_d^2)}
\leq
2C_p\bnorm{A_p^{\frac12}f}_{\L^p(G)},
\quad
f \in \dom A_p^{\frac12}
\end{equation}
(here, we use complex functions). For any functions $f,h \in \C^\infty(G)$, the skew-adjointness of the vector fields $X_1,\ldots,X_d$ gives
\begin{align*}
\MoveEqLeft
\la A_pf,h \ra_{\L^p(G),\L^{p^*}(G)}
=-\sum_{j=1}^d \la X_j^2f,h \ra_{\L^p(G),\L^{p^*}(G)}
=\sum_{j=1}^d \la X_jf,X_jh \ra_{\L^p(G),\L^{p^*}(G)} \\
&\ov{\eqref{gradient-32}}{=} \la \nabla f,\nabla h \ra_{\L^p(G,\ell_d^2),\L^{p^*}(G,\ell_d^2)}.
\end{align*}
Recall that the subspace $\dom A_p$ is a subspace (and even a core) for the unbounded operator $A_p^{\frac{1}{2}}$ by \cite[Proposition 3.1.1 (h) p.~61]{Haa06}. Since the heat semigroup is symmetric and its $\L^p$-realizations are
consistent, we have $A_p^*=A_{p^*}$. Applying \eqref{eq-Riesz-Lie-direct} with $p$ replaced by $p^*$ also gives $
\dom A_{p^*}^{\frac12}\subset\dom\nabla_{p^*}$. Since the subspace $\C^\infty(G)$ is a core for the operators $A_p$ and $A_{p^*}$ and, by definition, a core for $\nabla_p$ and $\nabla_{p^*}$, Proposition \ref{prop-generator-gradient-factorization} yields $A_p=(\nabla_{p^*})^*\nabla_p$. By duality with Proposition \ref{prop-duality-Riesz}, we obtain
\begin{equation*}
\bnorm{A_p^{\frac12}f}_{\L^p(G)} 
\approx_p \norm{\nabla_p(f)}_{\L^p(G,\ell^2_d)}, \quad f \in \C^\infty(G).
\end{equation*}
With \cite[Proposition 3.4]{Arh26b}, we obtain that $\dom A_p^{\frac12} = \dom \nabla_p$ and that
\begin{equation}
\label{final}
\bnorm{A_p^{\frac12}f}_{\L^p(G)} 
\approx_p \norm{\nabla_p(f)}_{\L^p(G,\ell^2_d)}, \quad f \in \dom \nabla_p.
\end{equation}
Hence the estimate \eqref{eq-lower-Riesz-direct-Poincare}, with $\alpha=\frac{1}{2}$, is satisfied as a particular case of \eqref{final}. By \cite[Corollary 4.11 p.~344]{EnN00} combined with Proposition \ref{prop-reduced-exp-stability-closed-range}, we have $0 \in \rho(A_{p,0})$. So it suffices to use Theorem \ref{thm-direct-Poincare-via-negative-powers} (or Theorem \ref{Th-Poincare-von-Neumann}) with $\partial=\nabla_p$.
\end{proof}

In the connected case, we can describe the mean ergodic projection.

\begin{prop}
\label{prop-connected}
Let $G$ be a connected compact Lie group equipped with a bi-invariant metric and its normalized Haar measure. Suppose that $1 < p < \infty$. Then the mean ergodic projection onto the subspace $\ker A_p$ is given by the contractive map $P \co \L^p(G) \to \L^p(G)$, $f \mapsto \left(\int_G f \d\mu_G\right)1$. 
\end{prop}

\begin{proof}
Let $f \in \ker A_p$. We have $A_pf=0$. Then $Af=0$ in the distributional sense. Recall that, according to \cite[Theorem (b) p.~35]{Ste70}, $A$ is elliptic (indeed it coincides with the Laplace-Beltrami operator by \cite[p.~36]{Ste70}). By elliptic regularity \cite[(23.22.8)]{Dieu88}, the function $f$ is smooth, i.e., $f \in \C^\infty(G)$. Hence, using the skew-adjointness \cite[Proposition 3.1 p.~17]{Rob91} of each $X_j$, we can write 
\[
0
=
\la Af,f\ra_{\L^2(G)}
=\bigg\la -\sum_{j=1}^d X_j^2f,f \bigg\ra_{\L^2(G)}
=\sum_{j=1}^d \la  X_jf,X_jf \ra_{\L^2(G)}
=\sum_{j=1}^d \norm{X_jf}_{\L^2(G)}^{2},
\]
where $d$ is the dimension of $G$. Thus $X_jf=0$ in $\L^2(G)$ for any integer $j \in \{1,\ldots,d\}$. Since $f$ is smooth, the equality $X_j f=0$ in $\L^2(G)$ implies $X_j f=0$ pointwise on $G$ for any $j \in \{1,\ldots,d\}$. The vector fields $X_1,\ldots,X_d$ form a global frame of the tangent bundle $\mathrm{T}G$. Let $s \in G$ and consider a tangent vector $v \in \mathrm{T}_sG$. There exist real numbers $a_1,\ldots,a_d$ such that $v=\sum_{j=1}^d a_j X_j(s)$. Consequently, using \cite[(17.14.1.1)]{Dieu72} in the last equality, we obtain
\[
\d_s f(v)
=\d_s f\bigg(\sum_{j=1}^d a_j X_j(s)\bigg)
=\sum_{j=1}^d a_j \d_s f(X_j(s))
=\sum_{j=1}^d a_j X_j f(s)
=0.
\]
Since $s \in G$ and $v \in \mathrm{T}_sG$ are arbitrary, we obtain $\d f=0$. Therefore the function $f$ is locally constant. Since $G$ is connected, we conclude that $f$ is constant. Conversely, by \cite[Theorem (d) p.~35]{Ste70}, $A$ maps constant functions to zero. Consequently, we have $\ker A_p=\mathbb{C}1$. 

The mean ergodic projection $P$ associated with the heat semigroup $(T_t)_{t \geq 0}$ is the projection onto $\ker A_p$. Let $f \in \L^p(G)$. There exists $a_f \in \mathbb{C}$ such that $P(f)=a_f1$. Recall that we established  that $0 \in \rho(A_{p,0})$ in the proof of Corollary \ref{cor-Poincare-Lie-2}. So, we can apply Proposition \ref{prop-strong-limit-projection} which says that $T_t(f) \to P(f)$ in $\L^p(G)$ when $t \to \infty$, thus in the space $\L^1(G)$. Since each operator $T_t$ of the semigroup satisfies $T_t(1)=1$ by \cite[Theorem 1 (vi) p.~38]{Ste70}, the classical argument \eqref{trace-preserving} shows that each operator $T_t$ preserves the normalized Haar measure $\mu_G$. Consequently, we obtain
$$
\int_G f \d \mu_G
=\int_G T_t(f) \d \mu_G
\xra[t \to \infty]{} \int_G P(f)\d \mu_G
=\int_G a_f1 \d \mu_G
=a_f.
$$
We conclude that $P(f)=\left(\int_G f \d\mu_G\right)1$. Finally, we have $
\norm{P(f)}_{\L^p(G)}
=
\left|\int_G f \d\mu_G\right|\norm{1}_{\L^p(G)}
\leq
\norm{f}_{\L^p(G)}$. Thus $\norm{P}_{\L^p(G)\to\L^p(G)} \leq 1$. 
\end{proof}

%

The optimal constant in \eqref{Arccozi} is unknown. We recall that Arcozzi works with real-valued functions by convention. Even in the Euclidean setting, the exact norm of the vector Riesz transform $\cal{R} \co \L^p(\R^d) \to \L^p(\R^d,\ell^2_d)$ defined by $\cal{R} f \ov{\mathrm{def}}{=} (R_1f,\ldots,R_df)$, where $f \in \L^p(\R^d)$, is unknown when $d \geq 2$. See \cite[p.~202]{Arc98} and the recent discussion in \cite[Section~5]{Ban26}. For \textit{real-valued} functions on $\R^d$, it has been conjectured that
\[
\norm{\cal{R}}_{\L^p_\R(\R^d)\to\L^p_\R(\R^d,\ell_d^2)}
= \cot\bigg(\frac{\pi}{2\max\{p,\frac{p}{p-1}\}}\bigg),
\]
see \cite[(5.36)]{Ban26}. Moreover, \cite[Theorem~5.9]{Ban26} shows that, for any fixed $d \geq 2$, this conjectured value is asymptotically correct as $p \to \infty$. Motivated by this Euclidean conjecture and by the results of Arcozzi for compact Lie groups, we propose the following compact analogue.

\begin{conj}
\label{conj-optimal-vector-Riesz-compact-Lie}
Let $G$ be a compact Lie group equipped with a bi-invariant Riemannian metric. Suppose that $1 < p < \infty$. Then, for any real-valued function $f \in \C^\infty(G)$, we have
\[
\norm{\nabla(f)}_{\L^p(G,\ell_d^2)}
\leq
\cot\bigg(
\frac{\pi}
{2\max\{p,\frac{p}{p-1}\}}
\bigg)
\bnorm{A_p^{\frac12}f}_{\L^p(G)}.
\]
\end{conj}

Note that $\cot(\frac{\pi}{2p})
\sim
\frac{2p}{\pi}$ when $p \to \infty$. 

\subsection{Poincar\'e inequalities with $q$-Ornstein-Uhlenbeck semigroups}
\label{sec-q-Ornstein}

We consider in this section a classical noncommutative deformation of the Ornstein--Uhlenbeck semigroup called $q$-Ornstein--Uhlenbeck semigroup. Here $q \in [-1,1)$ is a parameter. We refer to 
\cite{BS91}, \cite{BS94}, \cite{BKS97} and \cite{Lus99}  
for more information on this setting. 

We start by recalling several facts about $q$-Gaussian algebras. For any integer $n \geq 1$, we denote by $\mathrm{S}_n$ the symmetric group. If $\sigma$ is a permutation in $\mathrm{S}_n$ we denote by $\Inv(\sigma) \ov{\mathrm{def}}{=} \card \big\{ (i,j) : 1 \leq i < j \leq n, \sigma(i) > \sigma(j) \big\}$ the number of inversions of $\sigma$. Let $H$ be a real Hilbert space with complexification $H_{\mathbb{C}} \ov{\mathrm{def}}{=} H+\i H$. We consider the algebraic Fock space
 $\cal{F}_{\mathrm{alg}}(H_{\mathbb{C}}) 
\ov{\mathrm{def}}{=} \bigoplus_{n=0}^{\infty} H_{\mathbb{C}}^{\ot n}$, 
where $H_{\mathbb{C}}^{\otimes 0} \ov{\mathrm{def}}{=} \mathbb{C} \Omega$ for a unit vector $\Omega$. The $q$-Fock space 
$\cal{F}_{q}(H_{\mathbb{C}})$ is the completion of this space for the scalar product 
$$
\la h_1 \ot \cdots \ot h_n ,k_1 \ot \cdots \ot k_m \ra_{q}
\ov{\mathrm{def}}{=} \delta_{m,n} \sum_{\sigma \in \mathrm{S}_n} q^{\Inv(\sigma)} \la h_1,k_{\sigma(1)}\ra_{H_{\mathbb{C}}} \cdots \la h_n,k_{\sigma(n)}\ra_{H_{\mathbb{C}}},
$$
where $\delta_{m,n}$ is the Kronecker symbol. If $q=-1$, we must first divide out by the null space, and we obtain the usual antisymmetric Fock space. For any vector $e \in H$, we can introduce as in \cite[Definition 1.1 p.~133]{BKS97} the creation operator $\ell(e) \co \cal{F}_{q}(H_{\mathbb{C}}) \to \cal{F}_{q}(H_{\mathbb{C}})$, $h_1 \ot \cdots \ot h_n \mapsto e \ot h_1 \ot  \cdots \ot h_n$. 
These operators satisfy the $q$-relation $\ell(f)^*\ell(e)-q\ell(e)\ell(f)^* 
= \langle f,e\rangle_{H} \Id_{\cal{F}_{q}(H_{\mathbb{C}})}$ of \cite[p.~133]{BKS97}. These relations interpolate between the bosonic and fermionic commutation relations. Following \cite[Definition 2.1 p.~136]{BKS97}, we consider for any vector $e \in H$ the selfadjoint operator
\begin{equation}
\label{def-sq}
s_q(e) 
\ov{\mathrm{def}}{=} \ell(e)+\ell(e)^* \co \cal{F}_{q}(H_{\mathbb{C}}) \to \cal{F}_{q}(H_{\mathbb{C}}).
\end{equation}
The $q$-Gaussian von Neumann algebra $\Gamma_q(H)$ is the von Neumann algebra over $\cal{F}_q(H_{\mathbb{C}})$ generated by the operators $s_q(e)$ where $e \in H$. By \cite[Proposition 2.3 p.~136]{BKS97}, the von Neumann algebra $\Gamma_q(H)$ is finite and admits the normal finite faithful trace $\tau$ defined by $\tau(x) \ov{\mathrm{def}}{=} \la \Omega,x(\Omega) \ra_{\mathcal{F}_{q}(H_{\mathbb{C}})}$ where $x \in \Gamma_q(H)$. We also define the $*$-algebra $\cal{A}_q(H) \ov{\mathrm{def}}{=} \ast\mathrm{-alg} \{  s_q(e) : e \in H  \}$.


Let $H$ and $K$ be real Hilbert spaces and $T \co H \to K$ be a contraction with complexification $T_{\mathbb{C}} \co H_{\mathbb{C}} \to K_{\mathbb{C}}$. We can consider the second quantization map $
 \cal{F}_q(T) \co \cal{F}_{q}(H_{\mathbb{C}}) \to \cal{F}_{q}(K_{\mathbb{C}})$, $h_1 \ot \cdots \ot h_n  \mapsto T_{\mathbb{C} }(h_1) \ot \cdots \ot T_{\mathbb{C} }(h_n)$. According to \cite[Theorem 2.11 p.~140]{BKS97}, there exists a unique map $\Gamma_q(T) \co \Gamma_q(H) \to \Gamma_q(K)$ such that $
(\Gamma_q(T)(x))\Omega
=\mathcal{F}_q(T)(x\Omega)$ for any $x \in \Gamma_q(H)$ and this map is weak* continuous, unital, completely positive and trace preserving. 

Let $(a_t)_{t \geq 0}$ be a strongly continuous semigroup of selfadjoint contractions on $H$. For any $t \geq 0$, set $T_t \ov{\mathrm{def}}{=} \Gamma_q(a_t)$. Then by \cite[Lemma 9.3 p.~97]{JMX06} $(T_t)_{t \geq 0}$ is a weak* continuous semigroup of normal unital completely positive maps, which induces a strongly continuous semigroup of contractions on the noncommutative $\L^p$-space $\L^p(\Gamma_q(H))$ for any $1 \leq p < \infty$. In the special case where $a_t=\mathrm{e}^{-t}\Id_H$, the semigroup $(T_t)_{t \geq 0}$ is called the $q$-Ornstein--Uhlenbeck semigroup, with generator $-A_p$. 

According to \cite[Proposition 2.3 p.~136]{BKS97}, $\Omega$ is a cyclic and separating trace vector for $\Gamma_q(H)$. In particular, the map $\Gamma_q(H) \to \cal{F}_q(H_{\mathbb C})$, $x \mapsto x\Omega$, is injective. Moreover, by \cite[Remark 2.6 and Proposition 2.7 pp.~136--137]{BKS97}, for any $\xi_1,\ldots,\xi_n \in H_{\mathbb C}$ there exists a unique element $\W_q(\xi_1\ot\cdots\ot\xi_n) \in \cal{A}_q(H)$ such that $
\W_q(\xi_1\ot\cdots\ot\xi_n)\Omega
=
\xi_1\ot\cdots\ot\xi_n$. These Wick words linearly span the $*$-algebra $\cal{A}_q(H)$, i.e., 
$$
\cal{A}_q(H)
=
\Span\left\{
\W_q(\xi_1\ot\cdots\ot\xi_n):
\xi_1,\ldots,\xi_n\in H_{\mathbb C},\ n\geq 0
\right\}.
$$ 
Finally, we have
\begin{equation}
\label{Gamma-q-Wick}
\Gamma_q(T)(\W_q(\xi_1 \ot \cdots \ot \xi_n)) 
= \W_q(T_{\mathbb{C}} \xi_1\ot \dots \ot T_{\mathbb{C}} \xi_n), \quad \xi_1, \ldots, \xi_n \in H_{\mathbb{C}}.
\end{equation}

Consider the linear maps $j_0 \co H_{\mathbb{C}} \to H_{\mathbb{C}} \oplus H_{\mathbb{C}}$, $h \mapsto h \oplus 0$ and $j_1 \co H_{\mathbb{C}} \to H_{\mathbb{C}} \oplus H_{\mathbb{C}}$, $h \mapsto 0 \oplus h$. In this setting, following \cite[Lemma 5.1]{BGJ23}, we can introduce the map $\partial \co \cal{A}_q(H) \to \cal{A}_q(H \oplus H)$  defined by
\begin{equation}
\label{partial-Gao}
\partial
\bigl(\W_q(h_1\ot \cdots \ot h_k)\bigr)
\ov{\mathrm{def}}{=}
\sum_{i=1}^k
\W_q(j_0(h_1) \ot \cdots \ot j_1(h_i) \ot \cdots \ot j_0(h_k)),
\end{equation}
where $h_1,\ldots,h_k \in H_{\mathbb C}$. In the next result, we identify this operator with the gradient introduced by Lust--Piquard in \cite[p.~544]{Lus99}.

\begin{prop}
\label{prop-Lust-Gao-derivations-coincide}
Let $H$ be a real Hilbert space and let $-1 \leq q < 1$. Consider the second quantization $h_H \ov{\mathrm{def}}{=} \Gamma_q(j_0) \co \Gamma_q(H) \to \Gamma_q(H \oplus H)$ of the first-copy embedding. Then the derivation $\partial$ coincides on $\cal{A}_q(H)$ with the gradient $\nabla'$ of \cite[p.~544]{Lus99}, i.e., $\partial=\nabla'$. If $H$ is finite-dimensional and $D'$ denotes the transferred second differential quantization appearing in \cite[p.~544]{Lus99}, then $\partial=\i D' h_H$ on $\cal{A}_q(H)$.
\end{prop}

\begin{proof}
First assume that $H$ is finite-dimensional, say $H=\R^n$. 
By \cite[p.~520]{Lus99}, the algebra $\Gamma_A(\R^n)$ associated with the matrix $A=[q]_{1 \leq i,j \leq n}$ is identified with $\Gamma_q(H)$. If we consider the matrix $
A'
\ov{\mathrm{def}}{=}
A \ot
\begin{pmatrix}
1 & 1 \\
1 & 1
\end{pmatrix}$, the algebra $\Gamma_{A'}(\R^{2n})$ is identified with $\Gamma_q(H \oplus H)$. Following \cite[p.~524]{Lus99}, we introduce the vectors $
e
\ov{\mathrm{def}}{=}
\begin{bmatrix}
1 \\
0
\end{bmatrix}$ and $
f
\ov{\mathrm{def}}{=}
\begin{bmatrix}
0 \\
1
\end{bmatrix}$ of the space $\R^2$. Under the identification $H_{\mathbb C}\oplus H_{\mathbb C}
\simeq
H_{\mathbb C}\ot\mathbb C^2$, we have
\begin{equation}
\label{j0-j1}
j_0(h)=h \ot e
\quad \text{and} \quad
j_1(h)=h \ot f, \quad h \in H_{\mathbb C}.
\end{equation}
As in \cite[p.~522]{Lus99}, we consider the Hermitian matrix $
P
\ov{\mathrm{def}}{=}
\begin{bmatrix}
0 & \i \\
-\i & 0
\end{bmatrix}$. Then
\begin{equation}
\label{inter-35}
Pe
=-\i f
\quad \text{and} \quad
Pf=\i e.
\end{equation}
Let $D$ denote the second differential quantization of the map $\Id_{H_{\mathbb C}} \ot P \co H_{\mathbb C} \ot \mathbb{C}^2 \to H_{\mathbb C} \ot \mathbb{C}^2$ introduced in \cite[p.~530]{Lus99}. Then, by \cite[Lemma 1.1 c) (i) p.~527]{Lus99}, for any $h_1,\ldots,h_k\in H_{\mathbb C}$, we have
\begin{align*}
\MoveEqLeft
\i\, D\pi_k\big(j_0(h_1) \ot \cdots \ot j_0(h_k)\big) \\
&=\i\,\pi_k\bigg(\sum_{i=1}^k \bigg(\Id_{} \ot \cdots \ot (\Id_{H_{\mathbb C}} \ot P) \ot \cdots \ot \Id_{}\bigg)(j_0(h_1) \ot \cdots \ot j_0(h_k))\bigg)\\
&=\i\,\sum_{i=1}^k\pi_k\big(j_0(h_1) \ot \cdots \ot (\Id_{H_{\mathbb C}} \ot P)(j_0(h_i)) \ot \cdots \ot j_0(h_k)\big)\\
&\ov{\eqref{j0-j1}}{=} \i\,\sum_{i=1}^k\pi_k\big(j_0(h_1) \ot \cdots \ot (\Id_{H_{\mathbb C}} \ot P)(h_i \ot e) \ot \cdots \ot j_0(h_k)\big)\\
&=\i\, \sum_{i=1}^k \pi_k\big(j_0(h_1) \ot \cdots \ot (h_i \ot P(e)) \ot \cdots \ot j_0(h_k)\big) \\
&\ov{\eqref{inter-35}}{=} \sum_{i=1}^k\pi_k\big(j_0(h_1) \ot \cdots \ot (h_i \ot f) \ot \cdots \ot j_0(h_k)\big) \\
&\ov{\eqref{j0-j1}}{=} \sum_{i=1}^k \pi_k\big(j_0(h_1) \ot \cdots \ot j_1(h_i) \ot \cdots \ot j_0(h_k) \big).         
\end{align*}
With the notation of \cite{Lus99}, if $
u\in
(H_{\mathbb C}\oplus H_{\mathbb C})^{\ot k}$ then $\phi_{2n}(\W_q(u))=\pi_k(u)$.
Since $D'=\phi_{2n}^{-1}D\phi_{2n}$ (see \cite[p.~544]{Lus99}), we obtain
$$
\i\, D'\W_q\big(j_0(h_1) \ot \cdots \ot j_0(h_k)\big)
=\sum_{i=1}^k \W_q\big(j_0(h_1) \ot \cdots \ot j_1(h_i) \ot \cdots \ot j_0(h_k) \big).
$$
Since $
h_H\bigl(\W_q(h_1 \ot \cdots \ot h_k)\bigr)
=\Gamma_q(j_0)\bigl(\W_q(h_1 \ot \cdots \ot h_k)\bigr)
\ov{\eqref{Gamma-q-Wick}}{=}
\W_q(j_0(h_1) \ot \cdots \ot j_0(h_k))$, we finally deduce with \eqref{partial-Gao} that
\[
\i D'h_H\bigl(\W_q(h_1 \ot \cdots \ot h_k)\bigr)
=\partial(\W_q(h_1\ot \cdots \ot h_k)).
\]
Hence $
\partial
=\i D'h_H$ on all Wick words. Since Wick words linearly span $\cal A_q(H)$, we obtain $
\partial=\i D'h_H$ on $\cal A_q(H)$.
By \cite[Proposition 3.2 p.~544]{Lus99}, Lust--Piquard's transferred gradient is
$
\nabla'
=\i D' h_H$ since $h_H=h_{n,2n}^{-1}$ (with the notation $h_{mn}$ of \cite{Lus99}, see \cite[p.~535]{Lus99}).

Now, let $H$ be an arbitrary real Hilbert space. Given $h_1,\ldots,h_k \in H_{\mathbb C}$, set 
$$
H_0
\ov{\mathrm{def}}{=}
\Span_{\mathbb R}
\{\Re h_1,\mathrm{Im} h_1,\ldots,\Re h_k,\mathrm{Im} h_k\}.
$$ 
Then $H_0$ is finite-dimensional, and the Wick word
$\W_q(h_1\ot\cdots\ot h_k)$ belongs to $\cal A_q(H_0)$. Applying the finite-dimensional result to $H_0$ and using the canonical inclusions induced by second quantization yields $
\partial
\bigl(\W_q(h_1\ot\cdots\ot h_k)\bigr)
=
\nabla'
\bigl(\W_q(h_1\ot\cdots\ot h_k)\bigr)$. Since Wick words linearly span $\cal A_q(H)$, the conclusion follows.
%
%
%
%
%
\end{proof}

In this framework, if $1 < p < \infty$, the Riesz equivalence was obtained by Lust-Piquard in \cite[Theorem 3.3 p.~546]{Lus99} (see also \cite{Lus98}). Indeed, we have
\begin{equation}
\label{equiv-Lust}
\bnorm{A_p^{\frac{1}{2}}(f)}_{\L^p(\Gamma_q(H))}
\approx_p \norm{\partial(f)}_{\L^p(\Gamma_q(H \oplus H))}, \quad f \in \cal{A}_q(H),
\end{equation}
where $-A_p$ is the generator of the $q$-Ornstein-Uhlenbeck semigroup $(T_t)_{t \geq 0}$ on the Banach space $\L^p(\Gamma_q(H))$. By \cite[Proposition 3.4]{Arh26b}, we deduce that the unbounded operator $\partial \co \cal{A}_q(H) \subset \L^p(\Gamma_q(H)) \to \L^p(\Gamma_q(H \oplus H))$ is closable with closure denoted $\partial_p$, that $\dom A_p^{\frac{1}{2}}=\dom \partial_p$ and that the equivalence \eqref{equiv-Lust} extends to elements of the space $\dom \partial_p$.

The semigroup $(T_t)_{t \geq 0}$ is hypercontractive by \cite[Theorem 3 p.~461]{Bia97}. As explained in Section \ref{sec-gaps}, this implies a spectral gap. Combined with the reverse Riesz equivalence provided by \eqref{equiv-Lust}, this allows us to apply Theorem \ref{Th-Poincare-von-Neumann}. So we obtain the following Poincar\'e inequality, essentially stated in \cite[p.~39]{JLZZ24}.

\begin{thm}
Suppose that $1 < p < \infty$ and $-1 \leq q < 1$. Then
\begin{equation*}
\norm{f-\tau(f)1}_{\L^p(\Gamma_q(H))}
\lesssim_p \norm{\partial_p(f)}_{\L^p(\Gamma_q(H \oplus H))}, \quad f \in \dom \partial_p.
\end{equation*}
\end{thm}

The same argument extends to mixed $Q$-Gaussian von Neumann algebras, allowing us to recover the inequality of \cite[p.~39]{JLZZ24}.

\subsection{Poincar\'e inequalities on group von Neumann algebras}
\label{sec-group-von-Neumann-algebras}

Consider the group von Neumann algebra $\VN(G)$ of a discrete group $G$, generated by the unitary operators $\lambda_s \co \ell^2_G \to \ell^2_G$, where $s \in G$. Let $(T_t)_{t \geq 0}$ on $\VN(G)$ be a symmetric Markov semigroup of Fourier multipliers on $\VN(G)$. These semigroups admit a nice description. Indeed, by \cite[Proposition 2.3 p.~33]{ArK22}, there exists a unique real-valued conditionally negative definite function $\psi \co G \to \R$ satisfying $\psi(e) = 0$ such that 
\begin{equation}
\label{divers-100}
T_t(\lambda_s) 
= \e^{-t \psi(s)}\lambda_s, \quad t \geq 0,\quad s \in G
\end{equation}
and there exists a \textit{real} Hilbert space $H$ together with a mapping $b_\psi \co G \to H$ and a homomorphism $\pi \co G \to \mathrm{O}(H)$ into the orthogonal group $\mathrm{O}(H)$ of $H$ such that the $1$-cocycle law holds 
$\pi_s(b_\psi(t))
=b_\psi(st)-b_\psi(s),
$ 
for any $s,t \in G$ and such that $\psi(s) = \norm{b_\psi(s)}_H^2$ for any $s \in G$. 

Suppose that $-1 \leq q < 1$. In this context, if $\cal{P}_G \ov{\mathrm{def}}{=} \Span \{\lambda_s : s \in G \}$ then we can consider the map $\partial_{\psi,q} \co \cal{P}_G \to \Gamma_q(H) \rtimes_\alpha G$ defined by
\begin{equation}
\label{deriv-group}
\partial_{\psi,q}(\lambda_s)
\ov{\mathrm{def}}{=} s_q(b_\psi(s)) \rtimes \lambda_s, \quad s \in G.
\end{equation}
For any $1 < p < \infty$, according to \cite[Proposition 3.4 p.~84]{ArK22} (relying on the case $q=1$ of \cite[p.~544]{JMP18}), the operator $\partial_{\psi,q}$ is closable and if we denote by $\partial_{\psi,q,p}$ its closure, the same result says that $\dom \partial_{\psi,q,p} =\dom A_p^{\frac{1}{2}}$ and provides the Riesz equivalence 
\begin{equation}
\label{group-Riesz-56}
\bnorm{A_p^{\frac{1}{2}}(x)}_{\L^p(\VN(G))}
\approx_p \norm{\partial_{\psi,q,p}(x)}_{\L^p(\Gamma_q(H) \rtimes_\alpha G)}, \quad x \in \dom \partial_{\psi,q,p},
\end{equation}
where $\alpha \co G \to \Aut(\Gamma_q(H))$, $s \mapsto \Gamma_q(\pi_s)$. See \cite[p.~58]{ArK22} for more information. 

Now, we establish a Poincar\'e inequality in this setting. The following theorem may be viewed as saying that an $\L^2$-spectral gap characterizes an $\L^p$-Poincar\'e inequality. We refer to \cite{JuZ15b} and to \cite{JLZZ24} for particular cases of this inequality, with different assumptions. We recall that $
\ker\psi
\ov{\mathrm{def}}{=}
\{s\in G:\psi(s)=0\}$. Since $\psi(s)=\norm{b_\psi(s)}_H^2$, we have $
\ker\psi
=
\{s\in G:b_\psi(s)=0\}$. Moreover, the mean ergodic projection is the Fourier multiplier given by
\begin{equation}
\label{eq-expectation-kernel-psi}
\E_p(\lambda_s)
=
\begin{cases}
\lambda_s,&s\in\ker\psi,\\
0,&s\notin\ker\psi.
\end{cases}
\end{equation}
In particular, if $
c_\psi
\ov{\mathrm{def}}{=}
\inf_{s\notin\ker\psi}\norm{b_\psi(s)}_H^2$, with the convention $c_\psi=\infty$ if $\ker\psi=G$, then, whenever $c_\psi < \infty$, we have $
\Gap(A_2)
=c_\psi$. Thus $c_\psi > 0$ is precisely the $\L^2$-spectral gap condition for the reduced generator.


\begin{thm}
\label{Th-Poincare-group}
Suppose that $1 < p < \infty$ and $-1 \leq q < 1$. Let $G$ be a discrete group. Let $(T_t)_{t \geq 0}$ be a symmetric Markov semigroup of Fourier multipliers acting on the group von Neumann algebra $\VN(G)$ as in \eqref{divers-100}. Then the following assertions are equivalent.
\begin{enumerate}
\item We have $
c_\psi
>0$, with the convention $c_\psi=\infty$ if $\ker\psi=G$.
\item There exists a constant $C>0$ such that
\begin{equation}
\label{eq-group-Poincare-characterization}
\norm{x-\E_p(x)}_{\L^p(\VN(G))}
\leq C \norm{\partial_{\psi,q,p}(x)}_{\L^p(\Gamma_q(H) \rtimes_\alpha G)},
\quad
x \in \dom \partial_{\psi,q,p},
\end{equation}
where $\E_p \co \L^p(\VN(G)) \to \L^p(\VN(G))$ is the mean ergodic projection.
\end{enumerate}
Moreover, if $0 < c_\psi < \infty$ and $C_{p,q,\psi}$ denotes the best constant in \eqref{eq-group-Poincare-characterization}, then
\begin{equation}
\label{eq-group-optimal-Poincare-constant}
\frac1{m_{p,q}\sqrt{c_\psi}}
\leq
C_{p,q,\psi},
\end{equation}
where
\begin{equation}
\label{def-mpq}
m_{p,q}
\ov{\mathrm{def}}{=}
\norm{s_q(1)}_{\L^p(\Gamma_q(\R))}.
\end{equation}
Let $R_p > 0$ be such that
\begin{equation}
\label{eq-group-reverse-Riesz-constant}
\bnorm{A_p^{\frac12}x}_{\L^p(\VN(G))}
\leq
R_p
\norm{\partial_{\psi,q,p}(x)}_
{\L^p(\Gamma_q(H) \rtimes_\alpha G)},
\quad
x \in \dom\partial_{\psi,q,p},
\end{equation}
then
\begin{equation}
\label{eq-group-optimal-Poincare-two-sided}
\frac1{m_{p,q}\sqrt{c_\psi}}
\leq
C_{p,q,\psi}
\leq
\frac{2^{1-\theta_p}R_p}
{\sqrt{\theta_pc_\psi}}
\quad \text{where} \quad
\theta_p
\ov{\mathrm{def}}{=}
2\min\bigg\{\frac1p,\frac1{p^*}\bigg\}.
\end{equation}
\end{thm}

\begin{proof}
$1. \Rightarrow 2.$:
If $\ker\psi=G$, then $\psi=0$. Hence $T_t=\Id$, $A_p=0$ and $\E_p=\Id$, so \eqref{eq-group-Poincare-characterization} is trivial. Now, suppose that $\ker\psi \neq G$. Then $0<c_\psi<\infty$. On the
Hilbert space $\L^2(\VN(G))$, the operator $A_2$ is diagonal with respect to the orthonormal basis $(\lambda_s)_{s\in G}$ and satisfies
\[
A_2(\lambda_s)
=\norm{b_\psi(s)}_H^2\lambda_s,
\quad s \in G.
\]
Consequently, we have $\Gap(A_2)=c_\psi$. Proposition~\ref{prop-Markov-quantitative-gap}, applied with $\gamma=c_\psi$ and $\alpha=\frac12$, gives
\begin{equation}
\label{eq-group-negative-half-power}
\bnorm{A_{p,0}^{-\frac12}}_
{\L^p_0(\VN(G)) \to \L^p_0(\VN(G))}
\leq
\frac{2^{1-\theta_p}}
{\sqrt{\theta_pc_\psi}}.
\end{equation}
In particular, we have $0 \in \rho(A_{p,0})$. By \eqref{group-Riesz-56}, there exists a constant $R_p>0$ such that
\begin{equation}
\label{fine-466}
\bnorm{A_p^{\frac12}x}_{\L^p(\VN(G))}
\leq
R_p
\norm{\partial_{\psi,q,p}(x)}_
{\L^p(\Gamma_q(H)\rtimes_\alpha G)},
\quad
x \in \dom\partial_{\psi,q,p}.
\end{equation}
Applying Theorem~\ref{thm-direct-Poincare-via-negative-powers} with $\alpha=\frac12$, we obtain
\begin{align*}
\MoveEqLeft
\norm{x-\E_p(x)}_{\L^p(\VN(G))}
\ov{\eqref{eq-direct-Poincare}}{\leq}
\bnorm{A_{p,0}^{-\frac12}}_
{\L^p_0(\VN(G))\to\L^p_0(\VN(G))}
\bnorm{A_p^{\frac12}x}_{\L^p(\VN(G))} \\
&\ov{\eqref{fine-466}\eqref{eq-group-negative-half-power}}{\leq}
\frac{2^{1-\theta_p}R_p}
{\sqrt{\theta_pc_\psi}}
\norm{\partial_{\psi,q,p}(x)}_
{\L^p(\Gamma_q(H)\rtimes_\alpha G)},
\end{align*}
for any $x \in \dom\partial_{\psi,q,p}$. This proves assertion~2 as well as the upper estimate in
\eqref{eq-group-optimal-Poincare-two-sided}.


$2. \Rightarrow 1.$: If $\ker\psi = G$, then $c_\psi = \infty$ by convention, so there is nothing to prove. Suppose that $\ker\psi \neq G$, and let $s \notin \ker\psi$. By \eqref{eq-expectation-kernel-psi}, we have $\E_p(\lambda_s) = 0$. Since $\lambda_s$ is unitary and the canonical trace of $\VN(G)$ is normalized, we have $\norm{\lambda_s}_{\L^p(\VN(G))} = 1$. Moreover, right multiplication by the unitary $1 \rtimes \lambda_s$ preserves the $\L^p$-norm. Hence
\begin{equation}
\label{interminable}
\norm{\partial_{\psi,q,p}(\lambda_s)}_{\L^p(\Gamma_q(H) \rtimes_\alpha G)}
\ov{\eqref{deriv-group}}{=} \norm{s_q(b_\psi(s)) \rtimes \lambda_s}_{\L^p(\Gamma_q(H) \rtimes_\alpha G)}
=\norm{s_q(b_\psi(s))}_{\L^p(\Gamma_q(H))}.
\end{equation}
For any $h \in H$, the linearity of the map $h\mapsto s_q(h)$ and the orthogonal invariance of the $q$-Gaussian distribution give
\begin{equation}
\label{interminable-2}
\norm{s_q(h)}_{\L^p(\Gamma_q(H))}
=m_{p,q}\norm{h}_H, \quad h \in H,
\end{equation}
where $m_{p,q}$ is defined in \eqref{def-mpq}. 
Applying \eqref{eq-group-Poincare-characterization} to $\lambda_s$, we obtain
\begin{align*}
\MoveEqLeft
1
=
\norm{\lambda_s-\E_p(\lambda_s)}_{\L^p(\VN(G))}
\ov{\eqref{eq-group-Poincare-characterization}}{\leq}
C
\norm{\partial_{\psi,q,p}(\lambda_s)}_
{\L^p(\Gamma_q(H)\rtimes_\alpha G)}\\
&\ov{\eqref{interminable}}{=}
C
\norm{s_q(b_\psi(s))}_{\L^p(\Gamma_q(H))}
\ov{\eqref{interminable-2}}{=}
Cm_{p,q}\norm{b_\psi(s)}_H.
\end{align*}
Consequently, we conclude that
\[
\norm{b_\psi(s)}_H^2
\geq
\frac1{C^2m_{p,q}^2}.
\]
Taking the infimum over $s\notin\ker\psi$, we obtain
\begin{equation*}
c_\psi
\geq
\frac1{C^2m_{p,q}^2}.
\end{equation*}
In particular, $c_\psi>0$, which proves assertion~1. Equivalently, every admissible constant $C$ in \eqref{eq-group-Poincare-characterization} satisfies $
C
\geq
\frac1{m_{p,q}\sqrt{c_\psi}}$. Taking the infimum over all such constants $C$ yields
\[
C_{p,q,\psi}
\geq
\frac1{m_{p,q}\sqrt{c_\psi}},
\]
which proves \eqref{eq-group-optimal-Poincare-constant} and completes the proof.
\end{proof}

We refer to \cite[Theorem 7.1 p.~36]{JLZZ24} (see also \cite[pp. 13--15]{ArK22}) for concrete examples to which the preceding theorem applies.

\subsection{Poincar\'e inequalities with semigroups of Schur multipliers}
\label{sec-Poincare-Schur}

Let $I$ be a non-empty index set. Let $(T_t)_{t \geq 0}$ be a symmetric Markov semigroup of Schur multipliers acting on the von Neumann algebra $\B(\ell^2_I)$ of bounded operators acting on the complex Hilbert space $\ell^2_I$. In this situation, by \cite[Proposition 5.4 p.~415]{Arh13} there exists a \textit{real} Hilbert space $H$ and a family 
$
(\alpha_i)_{i \in I}
$
of vectors of $H$ such that for any $t \geq 0$, the Schur multiplier $T_t \co \B(\ell^2_I) \to \B(\ell^2_I)$ is associated to the matrix
\begin{equation}
\label{Semigroup-Schur}
\big[\e^{-t\|\alpha_i-\alpha_j\|_{H}^2}\big]_{i,j \in I}.	
\end{equation}
If $1 \leq p < \infty$, this semigroup induces a strongly continuous semigroup on the Schatten space $S^p_I \ov{\mathrm{def}}{=} S^p(\ell^2_I)$ with infinitesimal generator $-A_p$. We denote by $\M_{I,\fin}$ the dense subspace of the Schatten space $S^p_I$ of matrices with a finite number of nonzero entries. Suppose that $-1 \leq q < 1$. Recall that the von Neumann algebra $\Gamma_q(H)$ is defined in Section \ref{sec-q-Ornstein}. We equip it with its canonical normal finite faithful trace. Following \cite[(2.95) p.~62]{ArK22}, we can consider the unbounded operator $\partial_{\alpha,q} \co \M_{I,\fin} \subset S^p_I \to \L^p(\Gamma_q(H) \otvn \B(\ell^2_I))$ defined by
\begin{equation}
\label{deriv-Schur}
\partial_{\alpha,q}(e_{ij})
=s_q(\alpha_i-\alpha_j) \ot e_{ij},
\end{equation}
where the $q$-Gaussian $s_q(\alpha_i-\alpha_j)$ is defined in \eqref{def-sq}. Moreover, if $1 < p < \infty$, the operator $\partial_{\alpha,q}$ is closable by \cite[Proposition 3.11 p.~121]{ArK22}. If we denote by $\partial_{\alpha,q,p}$ its closure, the same result says that $\dom \partial_{\alpha,q,p} =\dom A_p^{\frac{1}{2}}$ and provides the Riesz equivalence
\begin{equation}
\label{Riesz-equivalence-Schur}
\bnorm{A_p^{\frac{1}{2}}(x)}_{S^p_I}
\approx_p \norm{\partial_{\alpha,q,p}(x)}_{\L^p(\Gamma_q(H) \otvn \B(\ell^2_I))}, \quad x \in \dom \partial_{\alpha,q,p}.
\end{equation}

Now, we establish a characterization of a Poincar\'e inequality in this setting. Since the Schatten space $S^p_I$ is reflexive for $1<p<\infty$ and the semigroup $(T_t)_{t\geq0}$ is contractive, it is mean ergodic by Example~\ref{bounded-ergodic}. We denote by $
P_p\co S^p_I\to S^p_I$ its mean ergodic projection. It follows from \eqref{Semigroup-Schur} that
\begin{equation}
\label{eq-Schur-ergodic-projection}
P_p(e_{ij})
=
\begin{cases}
e_{ij},&\alpha_i=\alpha_j,\\
0,&\alpha_i\neq\alpha_j,
\end{cases}
\quad i,j \in I.
\end{equation}
Set $
c_\alpha
\ov{\mathrm{def}}{=}
\inf\big\{
\norm{\alpha_i-\alpha_j}_H^2 : i,j \in I, \alpha_i\neq\alpha_j \big\}$, with the convention $c_\alpha=\infty$ if the set defining $c_\alpha$ is empty. On the Hilbert space $S^2_I$, the operator $A_2$ is diagonal with respect to the orthonormal basis $(e_{ij})_{i,j\in I}$ and satisfies
\[
A_2(e_{ij})
=
\norm{\alpha_i-\alpha_j}_H^2e_{ij},
\quad
i,j \in I.
\]
Consequently, whenever $c_\alpha<\infty$, we have $
\Gap(A_2)
=c_\alpha$. Thus $c_\alpha > 0$ is precisely the $\L^2$-spectral gap condition for the reduced generator.

The formula \eqref{eq-Schur-ergodic-projection} also defines a normal conditional expectation $P_\infty \co \B(\ell^2_I) \to \B(\ell^2_I)$ onto the block-diagonal von Neumann algebra associated with the equivalence relation $i \sim j \iff \alpha_i=\alpha_j$. In particular, $P_\infty$ is contractive. Let $P_1 \co S^1_I \to S^1_I$ be its preadjoint. Then $P_1$ is contractive as well, and the projections $(P_p)_{1 \leq p \leq \infty}$ are consistent. For $1 < p < \infty$, the projection $P_p$ coincides with the mean ergodic projection of $(T_t)_{t \geq 0}$ on $S^p_I$.

\begin{thm}
\label{Th-Poincare-Schur}
Suppose that $1 < p < \infty$ and $-1 \leq q < 1$. Let $I$ be a non-empty index set and let $(T_t)_{t\geq0}$ be a symmetric Markov semigroup of Schur multipliers acting on $\B(\ell^2_I)$ associated with \eqref{Semigroup-Schur}. Then the following assertions are equivalent.
\begin{enumerate}
\item We have $c_\alpha > 0$. 
\item There exists a constant $C > 0$ such that
\begin{equation}
\label{eq-Schur-Poincare-characterization}
\norm{x-P_p(x)}_{S^p_I}
\leq
C
\norm{\partial_{\alpha,q,p}(x)}_{\L^p(\Gamma_q(H) \otvn \B(\ell^2_I))},
\quad
x\in\dom\partial_{\alpha,q,p}.
\end{equation}
\end{enumerate}
Moreover, if $0 < c_\alpha < \infty$ and $C_{p,q,\alpha}$ denotes the best constant in \eqref{eq-Schur-Poincare-characterization}, then
\begin{equation}
\label{eq-Schur-optimal-Poincare-constant}
\frac1{m_{p,q}\sqrt{c_\alpha}}
\leq
C_{p,q,\alpha},
\end{equation}
where $m_{p,q}$ is defined in \eqref{def-mpq}. Let $R_p>0$ be such that
\begin{equation}
\label{eq-Schur-reverse-Riesz-constant}
\bnorm{A_p^{\frac12}x}_{S^p_I}
\leq
R_p
\norm{\partial_{\alpha,q,p}(x)}_{\L^p(\Gamma_q(H)\otvn\B(\ell^2_I))},
\quad
x \in \dom \partial_{\alpha,q,p}.
\end{equation}
Then
\begin{equation}
\label{eq-Schur-optimal-Poincare-two-sided}
\frac1{m_{p,q}\sqrt{c_\alpha}}
\leq
C_{p,q,\alpha}
\leq
\frac{2^{1-\theta_p}R_p}{\sqrt{\theta_pc_\alpha}},
\quad \text{where} \quad
\theta_p
\ov{\mathrm{def}}{=}
2\min\bigg\{\frac1p,\frac1{p^*}\bigg\}.
\end{equation}
\end{thm}

\begin{proof}
$1. \Rightarrow 2.$ If the set defining $c_\alpha$ is empty, then all the vectors $\alpha_i$ are equal. Hence $T_t=\Id$, $A_p=0$ and $P_p=\Id$, so \eqref{eq-Schur-Poincare-characterization} is trivial. Now, suppose that the set defining $c_\alpha$ is non-empty. Then $0<c_\alpha<\infty$. Set $
Q_p
\ov{\mathrm{def}}{=}
\Id-P_p$ and $
S^p_{I,0}
\ov{\mathrm{def}}{=}
\ker P_p
=
\ovl{\Ran A_p}$. Since $\Gap(A_2)=c_\alpha$, the spectral theorem gives
\begin{equation}
\label{eq-Schur-decay-L2}
\bnorm{T_tQ_2}_{S^2_I\to S^2_I}
\leq
\e^{-c_\alpha t},
\quad t \geq 0.
\end{equation}
On the other hand, both $T_t$ and $P_r$ are contractions on $S^r_I$ for $r\in\{1,\infty\}$. Therefore
\[
\bnorm{T_tQ_r}_{S^r_I\to S^r_I}
\leq
2,
\quad r \in \{1,\infty\}, t \geq 0.
\]
Suppose first that $2\leq p<\infty$. Interpolating between $S^2_I$ and $S^\infty_I$ and using 
$
\frac1p
=
\frac{\theta_p}{2}$, $\theta_p
=
\frac2p$, we obtain
\[
\bnorm{T_tQ_p}_{S^p_I\to S^p_I}
\leq
2^{1-\theta_p}
\e^{-\theta_pc_\alpha t}.
\]
If $1<p\leq2$, we interpolate instead between $S^1_I$ and $S^2_I$. Since $
\theta_p
=
\frac2{p^*}$ the same estimate follows. Hence
\begin{equation}
\label{eq-Schur-decay-Lp}
\bnorm{T_tQ_p}_{S^p_I\to S^p_I}
\leq
2^{1-\theta_p}
\e^{-\theta_pc_\alpha t},
\quad t \geq 0.
\end{equation}
In particular, the restriction of $(T_t)_{t\geq0}$ to $S^p_{I,0}$ is uniformly exponentially stable. Thus $0 \in \rho(A_{p,0})$, where $A_{p,0}$ denotes the part of $A_p$ in $S^p_{I,0}$. Using the semigroup representation of the negative square root, for any $x\in S^p_{I,0}$ we have
\[
A_{p,0}^{-\frac12}x
=
\frac1{\sqrt{\pi}}
\int_0^\infty
t^{-\frac12}T_tx
\d t.
\]
Therefore, by \eqref{eq-Schur-decay-Lp}, we obtain
\begin{align*}
\MoveEqLeft
\bnorm{A_{p,0}^{-\frac12}x}_{S^p_I}
\leq
\frac{2^{1-\theta_p}}{\sqrt{\pi}}
\int_0^\infty
t^{-\frac12}
\e^{-\theta_pc_\alpha t}
\d t
\norm{x}_{S^p_I}
=
\frac{2^{1-\theta_p}}{\sqrt{\theta_pc_\alpha}}
\norm{x}_{S^p_I}.
\end{align*}
Consequently, we have
\begin{equation}
\label{eq-Schur-negative-half-power}
\bnorm{A_{p,0}^{-\frac12}}_{S^p_{I,0}\to S^p_{I,0}}
\leq
\frac{2^{1-\theta_p}}{\sqrt{\theta_pc_\alpha}}.
\end{equation}
By \eqref{Riesz-equivalence-Schur}, there exists a constant $R_p>0$ such that
\begin{equation}
\label{fine-Rp}
\bnorm{A_p^{\frac12}x}_{S^p_I}
\leq
R_p
\norm{\partial_{\alpha,q,p}(x)}_{\L^p(\Gamma_q(H)\otvn\B(\ell^2_I))},
\quad
x \in \dom\partial_{\alpha,q,p}.
\end{equation}
Applying Theorem~\ref{thm-direct-Poincare-via-negative-powers} with $\alpha=\frac12$, we deduce that 
\begin{align*}
\MoveEqLeft
\norm{x-P_p(x)}_{S^p_I}
\ov{\eqref{eq-direct-Poincare}}{\leq}
\bnorm{A_{p,0}^{-\frac12}}_{S^p_{I,0}\to S^p_{I,0}}
\bnorm{A_p^{\frac12}x}_{S^p_I}
\ov{\eqref{eq-Schur-negative-half-power} \eqref{fine-Rp}}{\leq} 
\frac{2^{1-\theta_p}R_p}{\sqrt{\theta_pc_\alpha}}
\norm{\partial_{\alpha,q,p}(x)}_{\L^p(\Gamma_q(H)\otvn\B(\ell^2_I))},
\end{align*}
for any $x \in \dom \partial_{\alpha,q,p}$. This proves assertion~2 as well as the upper estimate in \eqref{eq-Schur-optimal-Poincare-two-sided}.

$2. \Rightarrow 1.$
If the set defining $c_\alpha$ is empty, then $c_\alpha=\infty$ by convention, so there is nothing to prove. Suppose that this set is non-empty and let $i,j\in I$ satisfy $\alpha_i\neq\alpha_j$. By \eqref{eq-Schur-ergodic-projection}, we have $P_p(e_{ij})=0$. Moreover, we have $
\norm{e_{ij}}_{S^p_I}=1$. We have
\begin{align}
\MoveEqLeft
\label{eq-Schur-gradient-matrix-unit}
\norm{\partial_{\alpha,q,p}(e_{ij})}_{\L^p(\Gamma_q(H) \otvn \B(\ell^2_I))}
\ov{\eqref{deriv-Schur}}{=}
\norm{s_q(\alpha_i-\alpha_j)\ot e_{ij}}_{\L^p(\Gamma_q(H) \otvn \B(\ell^2_I))} \\
&=
\norm{s_q(\alpha_i-\alpha_j)}_{\L^p(\Gamma_q(H))}
=
m_{p,q}\norm{\alpha_i-\alpha_j}_H. \nonumber
\end{align}
Applying \eqref{eq-Schur-Poincare-characterization} to $e_{ij}$ gives
\begin{align*}
1
=\norm{e_{ij}-P_p(e_{ij})}_{S^p_I}
\ov{\eqref{eq-Schur-Poincare-characterization}}{\leq}
C \norm{\partial_{\alpha,q,p}(e_{ij})}_{\L^p(\Gamma_q(H)\otvn\B(\ell^2_I))}
\ov{\eqref{eq-Schur-gradient-matrix-unit}}{=}
Cm_{p,q}\norm{\alpha_i-\alpha_j}_H.
\end{align*}
Consequently, we have
\[
\norm{\alpha_i-\alpha_j}_H^2
\geq
\frac1{C^2m_{p,q}^2}.
\]
Taking the infimum over all $i,j \in I$ such that $\alpha_i\neq\alpha_j$, we obtain
\[
c_\alpha
\geq
\frac1{C^2m_{p,q}^2}.
\]
In particular, $c_\alpha>0$, which proves assertion~1. Equivalently, every admissible constant $C$ in \eqref{eq-Schur-Poincare-characterization} satisfies
\[
C
\geq
\frac1{m_{p,q}\sqrt{c_\alpha}}.
\]
Taking the infimum over all such constants $C$ yields
\[
C_{p,q,\alpha}
\geq
\frac1{m_{p,q}\sqrt{c_\alpha}},
\]
which proves \eqref{eq-Schur-optimal-Poincare-constant} and completes the proof.
\end{proof}

\begin{remark} \normalfont
In ongoing joint work with C.~Kriegler, the author studies the problem of obtaining this inequality with sharp dependence on $p$.
\end{remark}

\section{A concentration inequality}
\label{sec-concentration}

This part is inspired by \cite[Section 5.2]{JuW26} and \cite[Theorem 2.7]{HuT21}. Here, we use the complex interpolation described in the books \cite{BeL76} and \cite{Lun18}.

\begin{thm}
\label{thm-moment-Poincare-to-concentration-log-loss}
Let $\cal{M}$ be a von Neumann algebra equipped with a normal faithful tracial state $\tau$. Suppose that $Y_\infty \subset Y_1$ is a contractive inclusion of Banach spaces. We consider the complex interpolation space $Y_p \ov{\mathrm{def}}{=} (Y_\infty,Y_1)_{\frac{1}{p}}$ where $1 < p < \infty$. Let $P_p \co \L^p(\cal{M}) \to \L^p(\cal{M})$ be a family of bounded projections which are compatible on $\cal M$, in the sense that there exists a linear map $P \co \cal M \to \cal M$ such that \(P_p(x)=P(x)\) for any \(x \in \cal M\) and any $p \in (1, \infty)$. Let $\partial \co \dom \partial \subset \cal{M} \to Y_\infty$ be an unbounded linear operator. Assume that there exist constants $K > 0$, $\beta > 0$, $\alpha \geq 0$, and \(p_0 > 1\) such that, for any $p \geq p_0$,
\begin{equation}
\label{eq-Lp-Poincare-growth-log}
\norm{x-P(x)}_{\L^p(\cal{M})}
\leq
K p^\beta(\log p)^\alpha
\norm{\partial x}_{Y_p},
\quad x \in \dom\partial.
\end{equation}
Then, for any $x \in \dom\partial$ such that $\norm{\partial x}_{Y_\infty}> 0$ there exist constants $c_{\alpha,\beta} > 0$ and $T_{\alpha,\beta,p_0} > 0$ such that
\begin{equation}
\label{eq-concentration-log-loss}
\tau\left(
1_{(t,\infty)}\bigl(|x-P(x)|\bigr)
\right)
\leq
\exp\left[
-c_{\alpha,\beta}
\frac{
\left(
\frac{t}{K\norm{\partial x}_{Y_\infty}}
\right)^{\frac1\beta}
}{
\left(
\log\frac{t}{K\norm{\partial x}_{Y_\infty}}
\right)^{\frac\alpha\beta}
}
\right]
\end{equation}
for any $
t \geq
T_{\alpha,\beta,p_0}
K\norm{\partial x}_{Y_\infty}$.
\end{thm}

\begin{proof}
Let \(x\in\dom\partial\), and put $a \ov{\mathrm{def}}{=}x-P(x)$. For any $p \in (1,\infty)$ we have
$\norm{\partial x}_{Y_p}
\leq
\norm{\partial x}_{Y_\infty}$ by \cite[Proposition 2.4 p.~50]{Lun18}. Thus \eqref{eq-Lp-Poincare-growth-log} gives, for any \(p \geq p_0\),
\begin{equation}
\label{eq-moment-bound-log}
\norm{a}_{\L^p(\cal M)}
\leq
K p^\beta(\log p)^\alpha
\norm{\partial x}_{Y_\infty}.
\end{equation}
By functional calculus, we have $
t^p1_{(t,\infty)}(|a|)
\leq
|a|^p$ for any $t > 0$. Applying the trace, we obtain the noncommutative Chebyshev inequality
\begin{equation}
\label{eq-nc-Chebyshev-log}
\tau\left(1_{(t,\infty)}(|a|)\right)
\leq
t^{-p}\norm{a}_{\L^p(\cal M)}^p,
\quad 1\leq p < \infty.
\end{equation}
Combining \eqref{eq-moment-bound-log} and \eqref{eq-nc-Chebyshev-log}, we get
\begin{equation}
\label{eq-tail-before-optim-log}
\tau\left(1_{(t,\infty)}(|a|)\right)
\leq
\left(
\frac{
Kp^\beta(\log p)^\alpha
\norm{\partial x}_{Y_\infty}
}{t}
\right)^p,
\quad p\geq p_0.
\end{equation}
Now we assume that $\norm{\partial x}_{Y_\infty} > 0$. Set $
u\ov{\mathrm{def}}{=}
\frac{t}{K\norm{\partial x}_{Y_\infty}}$. We shall assume that $u$ is sufficiently large. Define
\begin{equation}
\label{eq-choice-pt-log}
p_t
\ov{\mathrm{def}}{=}
\frac{\beta^{\frac{\alpha}{\beta}}}{\e}
\frac{u^{\frac{1}{\beta}}}{(\log u)^{\frac{\alpha}{\beta}}}.
\end{equation}
For \(u\) large enough, depending only on \(\alpha,\beta,p_0\), we have $p_t \geq p_0$. Moreover, since \(\alpha \geq 0\), again for $u$ large enough, $
\log p_t
\leq
\frac{1}{\beta}\log u$. Consequently, we have
\begin{equation}
\label{inter-921}
p_t^\beta(\log p_t)^\alpha
\leq \bigg(
\frac{\beta^{\frac{\alpha}{\beta}}}{\e} \bigg)^\beta
\frac{u}{(\log u)^\alpha}
\bigg(
\frac{1}{\beta}\log u \bigg)^\alpha
=
\e^{-\beta}u.
\end{equation}
Hence
\[
\frac{Kp_t^\beta(\log p_t)^\alpha \norm{\partial x}_{Y_\infty}}{t}
=
\frac{p_t^\beta(\log p_t)^\alpha}{u}
\ov{\eqref{inter-921}}{\leq} 
\e^{-\beta}.
\]
Taking \(p=p_t\) in \eqref{eq-tail-before-optim-log}, we obtain $
\tau\left(1_{(t,\infty)}(|a|)\right)
\leq
\left(\e^{-\beta}\right)^{p_t}
=
\exp(-\beta p_t)$. Using the definition \eqref{eq-choice-pt-log}, this gives
\[
\tau\left(1_{(t,\infty)}(|a|)\right)
\leq
\exp\left[
-\frac{\beta^{1+\frac{\alpha}{\beta}}}{\e}
\frac{u^{\frac{1}{\beta}}}{(\log u)^{\frac{\alpha}{\beta}}}
\right].
\]
Choose \(T_{\alpha,\beta,p_0}>0\) large enough such that, whenever \(u \geq T_{\alpha,\beta,p_0}\), we have \(u>1\), \(p_t\geq p_0\), and $\log p_t\leq \frac1\beta\log u$. Then the preceding argument applies for every such \(u\). This is precisely the condition $
t 
\geq T_{\alpha,\beta,p_0}K\norm{\partial x}_{Y_\infty}$. Thus we obtain \eqref{eq-concentration-log-loss}, with $c_{\alpha,\beta}=\frac{\beta^{1+\frac{\alpha}{\beta}}}{\e}$.
\end{proof}

\paragraph{Competing interests} The author declares that he has no competing interests.

\paragraph{Data availability statement}
The author confirms that the paper does not use any datasets. No datasets were generated during the current study.



{\footnotesize

\vspace{0.2cm}

\noindent C\'edric Arhancet\\ 
\noindent 6 rue Didier Daurat, 81000 Albi, France\\
URL: \href{http://sites.google.com/site/cedricarhancet}{https://sites.google.com/site/cedricarhancet}\\
cedric.arhancet@protonmail.com\\
ORCID: 0000-0002-5179-6972 
}

\end{document}